\documentclass[reqno,a4paper]{amsart}
\usepackage[english,activeacute]{babel}
\usepackage{amssymb,amsmath,amsthm,amsfonts,mathrsfs,hyperref,mathtools,stmaryrd}
\usepackage[font=footnotesize]{caption}
\usepackage{tikz}
\usetikzlibrary{intersections,shapes,trees,arrows,spy,positioning,decorations,arrows.meta,decorations.pathreplacing,patterns,calc}
\usepackage{tikz-cd}
\usepackage{tkz-euclide}
\usepackage{pgfplots}
\usepackage{xcolor}
\usepackage{url} 
\usepackage[numbers,square,sort]{natbib}
\usepackage{tabularx}
\usepackage{enumitem}  %Proper enumerate 
\usepackage[T1]{fontenc}
\usepackage[foot]{amsaddr}
\usepackage{url}  
\usepackage[pagewise,displaymath, mathlines]{lineno} 

\usepackage{tabularx}
\usepackage{booktabs}
\usepackage[labelfont=bf,format=plain,justification=raggedright,singlelinecheck=false]{caption}
\theoremstyle{plain}
\newtheorem{theorem}{Theorem}[section]

\newtheorem{lemma}[theorem]{Lemma}
\newtheorem{proposition}[theorem]{Proposition}
\newtheorem{corollary}[theorem]{Corollary}
\theoremstyle{definition}
\newtheorem{definition}{Definition}[section]
\theoremstyle{remark}
\newtheorem{remark}{Remark}[section]

\numberwithin{equation}{section}

\newcommand{\noo}[1]{{]{#1}]}}

\newcommand{\Bl}[2]{{\varrho_{\Db_{0,T}^{}}^{}}({#1},{#2})}
\newcommand{\Db}  {{\mathbb D}}
\newcommand{\Eb}  {{\mathbb E}}
\newcommand{\Nb}  {{\mathbb N}}
\newcommand{\Rb}  {{\mathbb R}}

\newcommand{\Pb}  {{\mathbb P}}

\renewcommand{\P}{{\mathbb P}}

\newcommand{\E}{{\mathbb E}}

\newcommand{\N}{{\mathbb N}}
\newcommand{\Sim}{{\rm{Sim}}}

\newcommand{\Gsb}{{\bar{\mathcal{G}}}}
\newcommand{\Gsbq}{{\bar{\mathcal{G}}_T^\omega}}
\newcommand{\As} {{\mathcal A}}
\newcommand{\Bs} {{\mathcal B}}
\newcommand{\Cs} {{\mathcal C}}

\newcommand{\Fs} {{\mathcal F}}
\newcommand{\Gs} {{\mathcal G}}
\newcommand{\Hs} {{\mathcal H}}

\newcommand{\Ls} {{\mathcal L}}
\newcommand{\Ms} {{\mathcal M}}

\newcommand{\Ps} {{\mathcal P}}

\newcommand{\Vs} {{\mathcal V}}

\newcommand{\ze}{{\textbf{0}}}
\newcommand{\dd} {{\rm d}}
\newcommand{\defeq} {{\coloneqq}}
\newcommand{\eqdef} {{\eqqcolon}}
\newcommand{\omb}{\bar{\omega}}
\newcommand{\om}{\omega}

\renewcommand{\phi}{\varphi}

\newcommand{\ind}{1\!\kern-1pt \mathrm{I}}
\newcommand{\rsto}{]\!\kern-1.8pt ]}
\newcommand{\lsto}{[\!\kern-1.7pt [}

\newcommand\F{\mbox{I\kern-2pt F}}

\newcommand{\bindist}[2]{\textrm{Bin}({#1},{#2})}
\newcommand{\hypdist}[3]{\textrm{Hyp}({#1},{#2},{#3})} 
\newcommand{\bz}{{\bf{0}}}

\pgfdeclarepatternformonly{soft crosshatch}{\pgfqpoint{-1pt}{-1pt}}{\pgfqpoint{4pt}{4pt}}{\pgfqpoint{3pt}{3pt}}%
{
	\pgfsetstrokeopacity{1}
	\pgfsetlinewidth{0.4pt}
	\pgfpathmoveto{\pgfqpoint{3.1pt}{0pt}}
	\pgfpathlineto{\pgfqpoint{0pt}{3.1pt}}
	\pgfpathmoveto{\pgfqpoint{0pt}{0pt}}
	\pgfpathlineto{\pgfqpoint{3.1pt}{3.1pt}}
	\pgfusepath{stroke}
}
\newcommand{\frompast}[6]{% points, advance, rand factor, options, end label
\draw[#4,thick] (0,2)
\foreach \x in {1,...,#1}
{   -- ++(#2,rand*#3)
}
;
\draw[#5,thick] (2,1)
\foreach \x in {#1,...,400}
{   -- ++(#2,rand*#3)
}
;
\draw[#6,thick] (8,2)
\foreach \x in {400,...,600}
{   -- ++(#2,rand*#3)
}
;

\draw[#6,thick] (12,2.5)
\foreach \x in {600,...,748}
{   -- ++(#2,rand*#3)
}
;
\pgfmathsetseed{13399}
\draw[#6,opacity=0.4,thick] (5,2)
\foreach \x in {#1,...,250}
{   -- ++(#2,rand*#3)
}
;
\pgfmathsetseed{132995}
\draw[#6,opacity=0.4,thick] (8,1.3)
\foreach \x in {400,...,600}
{   -- ++(#2,rand*#3)
}
;

\draw[#6,opacity=0.4,thick] (12,2.7)
\foreach \x in {600,...,748}
{   -- ++(#2,rand*#3)
}
;

}
\newcommand{\bafo}[6]{% points, advance, rand factor, options, end label
\draw[#4,opacity=0.5] (0,2)
\foreach \x in {1,...,#1}
{   -- ++(#2,rand*#3)
}
;
\draw[#5,opacity=0.5] (2,1)
\foreach \x in {#1,...,400}
{   -- ++(#2,rand*#3)
}
;
\draw[#6,opacity=0.5] (8,2)
\foreach \x in {400,...,600}
{   -- ++(#2,rand*#3)
}
;

\draw[#6,opacity=0.5] (12,2.5)
\foreach \x in {600,...,748}
{   -- ++(#2,rand*#3)
}
;
}
\newcommand{\Emmett}[6]{% points, advance, rand factor, options, end label
\draw[#4,thick] (0,2)
\foreach \x in {1,...,#1}
{   -- ++(#2,rand*#3)
}
;
\draw[#5,thick] (2,1)
\foreach \x in {#1,...,400}
{   -- ++(#2,rand*#3)
}
;
\draw[#6,thick] (8,2)
\foreach \x in {400,...,600}
{   -- ++(#2,rand*#3)
}
;

\draw[#6,thick] (12,2.5)
\foreach \x in {600,...,724}
{   -- ++(#2,rand*#3)
}
;
\draw[#6,thick] (14.5,2.3)
\foreach \x in {725,...,748}
{   -- ++(#2,rand*#3)
}
;

}

\newcommand{\various}[6]{% points, advance, rand factor, options, end label
\draw[#6, opacity=#4] (0,3)
\foreach \x in {1,...,#1}
{   -- ++(#2,rand*#3)
}
;
\draw[#6,opacity=#4] (2,1.8)
\foreach \x in {#1,...,449}
{   -- ++(#2,rand*#3)
}
;
\draw[#6,opacity=#5] (9,2.3)
\foreach \x in {400,...,450}
{   -- ++(#2,rand*#3)
}
;
\draw[#6,opacity=#5] (10,3.33)
\foreach \x in {550,...,650}
{   -- ++(#2,rand*#3)
}
;

\draw[#6,opacity=#5] (12,3.375)
\foreach \x in {600,...,748}
{   -- ++(#2,rand*#3)
}
;
}
\newcommand{\variousb}[6]{% points, advance, rand factor, options, end label
\draw[#6, opacity=#4] (0,3)
\foreach \x in {1,...,#1}
{   -- ++(#2,rand*#3)
}
;
\draw[#6, opacity=#4] (1,3.7)
\foreach \x in {51,...,150}
{   -- ++(#2,rand*#3)
}
;

\pgfmathsetseed{1330}
\draw[#6,opacity=#4] (3,2.91)
\foreach \x in {300,...,449}
{   -- ++(#2,rand*#3)
}
;
\draw[#6,opacity=#4] (6,3.39)
\foreach \x in {400,...,548}
{   -- ++(#2,rand*#3)
}
;
\draw[#6,opacity=#5] (9,1.35)
\foreach \x in {550,...,599}
{   -- ++(#2,rand*#3)
}
;

\draw[#6,opacity=#5] (10,1.07)
\foreach \x in {600,...,698}
{   -- ++(#2,rand*#3)
}
;
\pgfmathsetseed{1334}
\draw[#6,opacity=#5] (12,2.93)
\foreach \x in {700,...,850}
{   -- ++(#2,rand*#3)
}
;
}

\title[Moran models and Wright-Fisher diffusions in random environment]{Moran models and Wright--Fisher diffusions with selection and mutation in a one-sided random environment}

\author{F. Cordero$^{1}$}

\address{$1$ Faculty of Technology, Bielefeld University, Box 100131, 33501 Bielefeld, Germany}
\email{fcordero@techfak.uni-bielefeld.de}

\author{G. V\'{e}chambre$^{2}$}
\email{vechambre@amss.ac.cn}

\address{$^2$ Hua Loo-Keng Center for Mathematical Sciences, Academy of Mathematics and Systems Science, Chinese Academy of Sciences, No. 55, Zhongguancun East Road, Haidian District, Beijing, China}

\date{\today}%
\begin{document}

%\linenumbers
\maketitle
\vspace{-1cm}
\begin{abstract}
Consider a two-type Moran population of size $N$ with selection and mutation, where the selective advantage of the fit individuals is amplified at extreme environmental conditions. Assume selection and mutation are weak with respect to $N$, and extreme environmental conditions rarely occur. We show that, as $N\to\infty$, the type frequency process with time speed up by $N$ converges to the solution of a Wright-Fisher-type SDE with a jump term modeling the effect of the environment. We use an extension of the \emph{ancestral selection graph} (ASG) to describe the model's genealogical picture. Next, we show that the type frequency process and the line-counting process of a pruned version of the ASG satisfy a moment duality. This relation yields a characterization of the asymptotic type distribution. We characterize the ancestral type distribution using an alternative pruning of the ASG. Most of our results are stated in annealed and quenched form.
\end{abstract}

\smallskip { \footnotesize

}

\setcounter{tocdepth}{1}
\tableofcontents
\section{Introduction}\label{S1}
The \emph{Wright--Fisher diffusion with mutation and selection} describes the evolution of the type composition of an infinite two-type haploid population, which is subject to mutation and selection. Fit individuals reproduce at rate $1+\sigma$, $\sigma\geq 0$, whereas unfit ones reproduce at rate $1$. In addition, individuals mutate at rate $\theta$ to the fit type with probability $\nu_0\in[0,1]$, and to the unfit type with probability {$\nu_1:=1-\nu_0$}. The proportion of fit individuals evolves forward in time according to the SDE
\begin{equation}\label{WFD}
 \dd X(t)=\left[\theta\nu_0(1-X(t))-\theta\nu_1 X(t) + \sigma X(t)(1-X(t))\right]\,\dd t+\sqrt{2 X(t)(1-X(t))}\,\dd B(t),\quad t\geq 0,
\end{equation}
where $(B(t))_{t\geq 0}$ is a standard Brownian motion. The solution of \eqref{WFD} arises as the \emph{diffusion approximation} of (properly normalized) continuous-time Moran models and discrete-time Wright--Fisher models. In the neutral case, it also appears as the limit of a large class of Cannings models (see \cite{M01}). The genealogical counterpart to $X$ is the \emph{ancestral selection graph} (ASG), which is a branching-coalescing process coding the potential ancestors of an untyped sample of the population at present. It was introduced by Krone and Neuhauser in \citep{KroNe97,NeKro97} and later extended to models evolving under general neutral reproduction mechanisms (\citep{EGT10,BLW16}), and {to general} forms of frequency dependent selection (see \citep{Ne99,BCH18,GS18,CHS19}).

For $\theta=0$, the process $(R(t))_{t\geq 0}$ that counts the lines in the ASG is moment dual with $1-X$, i.e.
\begin{equation}\label{mdintro}
\Eb\left[(1-X(t))^n\mid X(0)=x\right]=\Eb\left[(1-x)^{R(t)}\mid R(0)=n\right],\qquad n\in\Nb, \, x\in[0,1],\, t\geq 0.
\end{equation}
This relation yields an expression for the absorption probability of $X$ at $0$ in terms of the stationary distribution of $R$. {For $\theta>0$}, two variants of the ASG dynamically resolve mutation events and encode relevant information of the model: \textit{the killed ASG} and \textit{the pruned lookdown ASG}. The killed ASG was introduced in \citep{BW18} to determine if a sample consists only of unfit individuals. Its line-counting process extends Eq. \eqref{mdintro} to the case $\theta>0$ \cite[Prop. 1]{BW18} (see \cite[Prop. 2.2]{CM19} for a generalization). This allows to characterize the stationary distribution of $X$. The pruned lookdown ASG in turn was introduced in \citep{LKBW15} (see \citep{BLW16, FC17} for extensions) as a tool to determine the \emph{ancestral type distribution}, i.e. the type distribution of the individuals that have been successful in the long run. 
\smallskip

In many biological situations the strength of selection fluctuates in time. The influence of random fluctuations in selection intensities on the growth of populations has been the object of extensive research in the past (see e.g.\cite{G72,KL74,KL74b,KL75,Bu87, BG02, SJV10}), and it is currently experiencing renewed interest (see e.g. \cite{BCM19,CSW19, BEK19, ChK19, GJP18, GPS19}). A concrete example is given in \cite{PM14}, where different antibiotic treatment strategies against a bacterial population are compared. There it is proved that a constant administration of antibiotics is not optimal, and that the best treatment strategies depend on the length of treatment. In this paper, we consider the scenario where the selective advantage of fit individuals is accentuated by exceptional environmental conditions (e.g. extreme temperatures, precipitation, humidity variation, abundance of resources, etc.). As an example, consider a population consisting of fit and unfit individuals which is subject to catastrophes. Assume that only fit individuals are resistant to the catastrophes. Hence, shortly after a catastrophe the population may drop below its carrying capacity 
and subsequently grow quickly. Since fit individuals have a reproductive advantage, it is likely that their relative frequency will grow fast after a catastrophe. One may also think of a population consisting of individuals that are specialized to high temperatures, as well as wild-type individuals accustomed, but not specialized to them. The environment is characterized by moderately high temperatures, present most of the time, and short periods of extreme heat. It is then likely that specialized individuals have a (slight) reproductive advantage under moderate temperatures, and a more prominent advantage at extreme temperatures. 

%
%Think for example of plants, where one type grows faster than the other, allowing it to be more exposed to the sun and thus have a better chance of flowering and being pollinated. In addition, their seeds will be more vulnerable to wind gusts and will have a better chance of spreading. Another example is the color of a plant or animal, which can improve both its ability to blend in with its surroundings and its thermoregulation. 
\smallskip

To model the previously described scenario we use a two-type Moran population with mutation and selection, immersed in a varying environment. The environment is modeled via a countable collection of points $(t_i,p_i)_{i\in I}$ {in $(-\infty,\infty)\times(0,1)$, satisfying that $\sum_{t_i\in[s,t]} p_i<\infty$ for all $s<t$.} Each $t_i$ represents a time of an instantaneous environmental change; the \emph{peak} $p_i$ models the strength of this event: at time $t_i$ each fit individual independently reproduces with probability $p_i$. Each offspring replaces a different individual in the population, so that the population size remains constant. The summability of the peaks assures that the number of reproductions in any compact time interval is almost surely finite. 
In this context, we show that the type-frequency process is continuous with respect to the environment. The proof uses coupling techniques that uncover the effect of small environmental changes. 
\smallskip

Next, we consider a random environment given by a Poisson point process on {$(-\infty,\infty) \times (0,1)$} with intensity measure $\dd t \times \mu$, where $\dd t$ stands for the Lebesgue measure and $\mu$ is a measure on $(0,1)$ satisfying $\int x \mu(\dd x) < \infty$. {Then, we let population size grow to infinity, and} we show that, in an appropriate parameter regime, the fit-type frequency process converges to the solution of the SDE
 \begin{equation}\label{WFSDE}
 \dd X(t)=\theta\left(\nu_0(1-X(t))-\nu_1 X(t)\right)\dd t +X(t-)(1-X(t-))\dd S(t) +\sqrt{2 X(t)(1-X(t))}\dd B(t),\quad t\geq 0,
\end{equation}
where $S(t)\coloneqq \sigma t+J(t)$, and $J$ is a pure-jump subordinator with L\'evy measure $\mu$, independent of $B$, which represents the cumulative effect of the environment. We refer to $X$ as the \emph{Wright--Fisher diffusion in random environment}. We prove the convergence in an annealed setting, i.e. when the environment is random. For environments given by compound Poisson processes, we show that the convergence also holds in a quenched sense, i.e. when a realization of the environment is fixed. For $\theta=0$, Eq. \eqref{WFSDE} is a particular case of \cite[Eq. 3.3]{BCM19}, which arises as the {large population limit of} a family of discrete-time Wright--Fisher models \cite[Thm. 3.2]{BCM19}.
\smallskip

Next, we generalize the construction of the ASG, the killed ASG and the {pruned lookdown} ASG to incorporate the effect of the environment. In the annealed case, we establish a relation between $X$, the line-counting process of the killed {ASG, and} the total increment of the environment; we refer to this relation as a \emph{reinforced moment duality}. The latter is a central tool to characterize the asymptotic type frequencies. We also express the ancestral type distribution in terms of the line-counting process of the pruned lookdown ASG. Analogous results are obtained in the more involved quenched setting. 
\smallskip

As an application of our results, we compare the long-term behavior of two Wright--Fisher diffusions without mutations; the first one having parameter $\sigma=0$ and an environment with L\'evy measure $\mu$; the second one having parameter $\sigma=\int_{(0,1)} y\mu(\dd y)$ and no environment. We prove that the probability of fixation of the fit type is smaller under first model than under the second one, provided that the initial frequency of fit individuals is sufficiently large.

\smallskip

The analysis of a more realistic scenario where environmental changes are not always favorable to the same type can not be done via the methods presented in this paper. The main reason is that in such a setting the frequency process does not admit a moment dual. To circumvent this problem one has to take into account the whole combinatorics of the ASG, which is a cumbersome object. This is the object of a forthcoming study. 

We would also like to mention a parallel development by \citet{CSW19}. They study the accessibility of the boundaries and the fixation probabilities of a generalization of the SDE \eqref{WFSDE} with $\theta=0$. \cite{CSW19} makes only use of the ASG and does not cover the case $\theta>0$, where the killed and the pruned lookdown ASG play a pivotal role. Moreover, the reinforced moment duality, and all the results obtained in the quenched setting, are to the best of our knowledge new. 
\smallskip

\smallskip

The article is organized as follows. An outline of the paper containing our main results is given in Section~\ref{S2}. In Section~\ref{S3} we prove the continuity of the type frequency process in the Moran model with respect to the environment, that \eqref{WFSDE} is well-posed, and that it arises in the large population {limit of} the type frequency process of a sequence of Moran models. In Section \ref{S4} we give more detailed definitions of the ASG, the killed ASG and the pruned lookdown ASG. Section~\ref{S5} is devoted to the proofs of: (i) the annealed moment duality between the process $X$ and the line-counting process of the killed ASG, (ii) the long-term behavior of the annealed type frequency process, and (iii) the annealed ancestral type distribution. The quenched versions of these results are proved in Section \ref{S6}. {Section \ref{S7} provides additional (quenched) results for environments having finitely many jumps in any compact time interval.}

%%%%%%%%%%%%%%%%%%%%%%%%%%%%%%%%%%%%%%%%%%%%%%%%%%%%%%%%%%%%%%%%%%%%%%%%%%%%%%%%%%%%%%%%%%%%%%%%
%%%%%%%%%%%%%%%%%%%%%%%%%%%%%%%%%%%%%%%%%%%%%%%%%%%%%%%%%%%%%%%%%%%%%%%%%%%%%%%%%%%%%%%%%%%%%%%%
%%%%%%%%%%%%%%%%%%%%%%%%%%%%%%%Section 2%%%%%%%%%%%%%%%%%%%%%%%%%%%%%%%%%%%%%%%%%%%%%%%%%%%%%%%%
%%%%%%%%%%%%%%%%%%%%%%%%%%%%%%%%%%%%%%%%%%%%%%%%%%%%%%%%%%%%%%%%%%%%%%%%%%%%%%%%%%%%%%%%%%%%%%%%
%%%%%%%%%%%%%%%%%%%%%%%%%%%%%%%%%%%%%%%%%%%%%%%%%%%%%%%%%%%%%%%%%%%%%%%%%%%%%%%%%%%%%%%%%%%%%%%%
\section{Description of the model and main results}\label{S2} 
{\bf Notation.} The positive integers are denoted by $\N$, and we set $\N_0\coloneqq \N\cup\{0\}$. For $m\in \N$, \[[m]\coloneqq \{1,\ldots,m\},\quad [m]_0\coloneqq [m]\cup\{0\},\quad \text{and} \quad \noo{m}\coloneqq [m]\setminus\{1\}.\]
For $s<t$, we denote by $\Db_{s,t}$ (resp. $\Db$) the space of c\`{a}dl\`{a}g functions from $[s,t]$ (resp. $\Rb$) to $\Rb$, which is endowed with the Billingsley metric inducing the $J_1$-Skorokhod topology and makes the space complete (see Appendix \ref{A1}). For any Borel set $S\subset\Rb$, denote by $\Ms_f(S)$ (resp. $\Ms_1(S)$) the set of finite (resp. probability) measures on $S$. We use $\xrightarrow[]{(d)}$ to denote convergence in distribution of random variables and $\xRightarrow[]{(d)}$ for convergence in distribution of c\`{a}dl\`{a}g process.
\smallskip

For $n\in \N_0$ and $k,m\in[n]_0$, we write $K\sim \hypdist{n}{m}{k}$ if $K$ is a hypergeometric random variable with parameters $n,m$, and $k$, i.e \[\P(K=i)= {\binom{n-m}{k-i} \binom{m}{i}}/{\binom{n}{k}},\quad i\in[k\wedge m]_0.\]
For~$x\in[0,1]$ and~$n\in \N$, we write $B\sim \bindist{n}{x}$ if $B$ is a binomial random variable with parameters $n$ and $x$, i.e. \[\P(B=i)= \binom{n}{i} x^i (1-x)^{n-i},\quad i\in[n]_0.\]
Relevant notations introduced in the forthcoming sections are collected in Section \ref{Nota}. 
\subsection{Moran models in deterministic pure-jump environments}\label{s21}

Consider a population of size $N$ with two types, type $0$ and $1$, subject to mutation {and selection influenced by} a deterministic environment. {The latter} is modeled by an at most countable collection $\zeta\coloneqq (t_k, p_k)_{k \in I}$ of points in $(-\infty,\infty) \times (0,1)$ satisfying for any $s,t\in\Rb$ with $s\leq t$ that 
\begin{equation} \label{summableassumption}
\sum_{t_k \in[s, t]} p_k < \infty.
\end{equation}
We refer to $p_k$ as the \emph{peak of the environment} at time $t_k$. The individuals in the population undergo the following dynamic. Each individual independently mutates at rate $\theta_N\geq 0$ with probability $\nu_{0}\in[0,1]$ (resp. $\nu_1\coloneqq 1-\nu_0$) to type $0$ (resp. $1$). Reproduction occurs independently from mutation. Individuals of type $1$ reproduce at rate $1$, whereas individuals of type $0$ reproduce at rate $1+\sigma_N$, $\sigma_N\geq0$\footnote{The subscript $N$ in the parameters $\sigma_N$ and $\theta_N$ emphasizes their dependence on $N$. In Theorem \ref{thm2.2} we will require that they are asymptotically proportional to $1/N$.}. Thus, we refer to type $0$ (resp. type $1$) as the \emph{fit} (resp. \emph{unfit}) type. In addition, at time $t_k$ each type $0$ individual independently reproduces with probability $p_k$. At reproduction times: (a) each individual produces at most one offspring, which inherits the parent's type, and (b) if  $n$ individuals are born, $n$ individuals are randomly sampled without replacement from {the population present before the reproduction event (including the parents) } to die, hence keeping the size of the population constant. 

\subsection*{Graphical representation}%\label{s212}
In the absence of environmental factors (i.e. $\zeta=\emptyset$), it's classical to describe the evolution of the population by means of the graphical representation as an {interacting particle system (IPS).} This decouples the randomness of the model coming from the initial type configuration from the one coming from mutations and reproductions. {We now extend the graphical representation to incorporate the effect of the environment (see Section \ref{s31} for a more detailed description). 
\smallskip

In the graphical representation individuals are represented by horizontal lines at levels $i\in [N]$} {(see Fig. \ref{particlepicture})}. {Time runs forward from left to right. Potential reproduction events are depicted by arrows, with the (potential) parent at the tail and the offspring at the tip. We distinguish between neutral and selective arrows. \emph{Neutral arrows} have a filled arrowhead; they occur at rate $1/N$ per pair of lines. \emph{Selective arrows} have an open arrowhead;  they occur in two independent ways: first, at rate $\sigma_N/N$ per pair of lines, and second, at any time $t_k$, $k\in I$, a random number $n_k\sim\bindist{N}{p_k}$ of lines shoot selective arrows to $n_k$ individuals in the population. Furthermore, beneficial (deleterious) mutations, depicted as circles (crosses), occur at rate $\theta_N \nu_0$ (at rate $\theta_N \nu_1$) per line. 
\smallskip

Note that for any $s<t$, the number of non-environmental graphical elements present in $[s,t]$ is almost surely finite. Moreover, thanks to Assumption \eqref{summableassumption}, we have
\begin{align}
\Eb\left[\sum_{t_k\in[s,t]}n_k\right]=N\sum_{t_k\in[s,t]}p_k<\infty.\label{finitenbofevents}
\end{align}
Hence, $\sum_{t_k\in[s,t]}n_k<\infty$ almost surely, i.e. the number of arrows in $[s,t]$ due to peaks of the environment is almost surely finite.
\smallskip

Once the graphical elements in $[s,t]\times [N]$ are drawn, we specify the initial conditions by assigning types to the $N$ lines at time $s$ and propagate them forward in time according to the following rules: the type of a line right after a circle (resp. cross) is $0$ (resp. $1$); type $0$ propagates through neutral arrows \emph{and} selective arrows; type $1$ propagates only through neutral arrows. }
\begin{figure}[t!]
\begin{minipage}{.4\textwidth}
    \centering
    \scalebox{0.65}{\begin{tikzpicture} 
        
        \draw[dashed, opacity=0.4] (1,-0.2) --(1,1) (1,2)--(1,3);
        \draw[dashed, opacity=0.4] (6.2,-0.2) --(6.2,0) (6.2,1) --(6.2,2);
        
        \node [right] at (-0.2,-0.5) {$s$};
        \node [right] at (0.8,-0.5) {$t_0$};
        \node [right] at (6,-0.5) {$t_1$};
        \node [right] at (8.3,-0.5) {$s+T$};
        \node [right] at (3.5,-0.8) {$t$};

        \draw[-{triangle 45[scale=5]},thick] (4.5,2) -- (4.5,1) node[text=black, pos=.6, xshift=7pt]{};
        \draw[thick] (0,1) -- (8.5,1);
        \draw[thick] (8.5,2) -- (0,2);
        \draw[thick] (8.5,3) -- (0,3);
        \draw[thick] (0,4) -- (8.5,4);
        \draw[thick] (0,0) -- (8.5,0);
        
        \draw[-{triangle 45[scale=5]},thick] (.4,1) -- (.4,0);
        \draw[-{open triangle 45[scale=5]},thick] (7.6,1) -- (7.6,4);
        \draw[-{open triangle 45[scale=5]},thick] (1.7,2) -- (1.7,4);
        \draw[-{open triangle 45[scale=5]},thick] (1,3) -- (1,4);
        \draw[-{open triangle 45[scale=5]},thick] (1,1) -- (1,2);
        \draw[-{open triangle 45[scale=5]},thick] (6.2,4) -- (6.2,2);
        \draw[-{open triangle 45[scale=5]},thick] (6.2,0) -- (6.2,1);
        \draw[-{open triangle 45[scale=5]},thick] (2.5,0) -- (2.5,1);
        \draw[-{open triangle 45[scale=5]},thick] (7,2) -- (7,0);
        \draw[-{triangle 45[scale=5]},thick] (3,3) -- (3,1);
        
         \draw[-{angle 45[scale=5]}] (2.5,-0.5) -- (4.5,-0.5) node[text=black, pos=.6, xshift=7pt]{};
        %%%%mutations%%%%%%
        \node[ultra thick] at (3.5,4) {$\bigtimes$} ;
        \node[ultra thick] at (6.6,4) {$\bigtimes$} ; 
        \draw (2.1,2) circle (1.5mm)  [fill=white!100];  
        \draw (4.1,0) circle (1.5mm)  [fill=white!100];
        
        %%%%%%%types%%%%%
        \node [right] at (-0.5,0) {$1$};
        \node [right] at (-0.5,1) {$1$};
        \node [right] at (-0.5,2) {$1$};
        \node [right] at (-0.5,3) {$0$};
        \node [right] at (-0.5,4) {$1$};
        \node [right] at (8.5,0) {$0$};
        \node [right] at (8.5,1) {$0$};
        \node [right] at (8.5,2) {$0$};
        \node [right] at (8.5,3) {$0$};
        \node [right] at (8.5,4) {$0$};

        \end{tikzpicture}    }\end{minipage}\begin{minipage}{.4\textwidth}
        \centering
    \scalebox{0.65}{\begin{tikzpicture}   
       %%%%%jumps of the environment%%%%  
        \draw[dashed,thick,opacity=0.3] (1,-0.2) --(1,1) (1,2)--(1,3) (1,4)--(1,4);
        \draw[dashed,thick,opacity=0.3] (6.2,-0.2) --(6.2,0) (6.2,1) --(6.2,2) (6.2,4)--(6.2,4);
   %%%%%time%%%%%%%%%%%%%%%%%%%%%%     
        \node [right] at (-0.2,-0.6) {$T$};
         \node [right] at (0.3,-0.6) {$s+T-t_0$};
        \node [right] at (5.5,-0.6) {$s+T-t_1$};
        \node [right] at (8.3,-0.6) {$0$};
        \node [right] at (3.5,-0.8) {$\beta$};

        \draw[-{triangle 45[scale=5]}] (4.5,-0.5) -- (2.5,-0.5) node[text=black, pos=.6, xshift=7pt]{};

       %%%%%%%non-ancestral lines%%%%%%%%%%%%      
        \draw[thick,opacity=0.3] (0,1) -- (8.5,1);
        \draw[thick,opacity=0.3] (8.5,2) -- (0,2);
        \draw[thick,opacity=0.3] (8.5,3) -- (0,3);
        \draw[thick,opacity=0.3] (0,0) -- (8.5,0);
        \draw[thick,opacity=0.3] (0,4) -- (8.5,4);
   %%%%%%%%ancestral lines%%%%%%%%%%%%%%%%%     
        \draw[thick] (0,4) -- (6.2,4);
        \draw[thick] (0,2) -- (8.5,2);
        \draw[thick] (4.5,1) -- (8.5,1);
        \draw[thick] (0.4,0) -- (6.2,0);
        \draw[thick] (1,1) -- (0.0,1);
        \draw[thick] (1,3) -- (0.0,3);
   %%%%%%%%%%%%% relevant arrows%%%%%%%%%%
        \draw[-{triangle 45[scale=5]},thick] (4.5,2) -- (4.5,1) node[text=black, pos=.6, xshift=7pt]{};
        \draw[-{triangle 45[scale=5]},thick] (0.4,1) -- (0.4,0);
        \draw[-{open triangle 45[scale=5]},thick] (1.7,2) -- (1.7,4);
        \draw[-{open triangle 45[scale=5]},thick] (1,3) -- (1,4);
        \draw[-{open triangle 45[scale=5]},thick] (1,1) -- (1,2);
        \draw[-{open triangle 45[scale=5]},thick] (6.2,4) -- (6.2,2);
        \draw[-{open triangle 45[scale=5]},thick] (6.2,0) -- (6.2,1);
   %%%%%%%%%%%%irrelevant arrows%%%%%%%%%
        \draw[-{triangle 45[scale=5]},thick, opacity=0.3] (3,3) -- (3,1);
        \draw[-{open triangle 45[scale=5]},thick, opacity=0.3] (7.6,1) -- (7.6,4);
        \draw[-{open triangle 45[scale=5]},thick, opacity=0.3] (2.5,0) -- (2.5,1);
        \draw[-{open triangle 45[scale=5]},thick, opacity=0.3] (7,2) -- (7,0);
          %%%%%mutations%%%%%%
        \node[ultra thick] at (3.5,4) {$\bigtimes$} ;
        \draw[thick] (2.1,2) circle (1.5mm)  [fill=white!100];  
        \draw[thick] (4.1,0) circle (1.5mm)  [fill=white!100];
        \node[ultra thick,opacity=0.3] at (6.6,4) {$\bigtimes$} ; 
        
            %%%%%%%types%%%%%
        \node [right,white] at (-0.5,0) {$1$};
        \node [right,white] at (-0.5,1) {$1$};
        \node [right,white] at (-0.5,2) {$1$};
        \node [right,white] at (-0.5,3) {$0$};
        \node [right,white] at (-0.5,4) {$1$};
        \node [right,white] at (8.5,0) {$0$};
        \node [right,white] at (8.5,1) {$0$};
        \node [right,white] at (8.5,2) {$0$};
        \node [right,white] at (8.5,3) {$0$};
        \node [right,white] at (8.5,4) {$0$};
        \end{tikzpicture}    }
        \end{minipage}
                              
    \caption{Left: a realization of the Moran IPS; time runs forward from left to right; the environment has peaks at times $t_0$ and $t_1$. Right: the ASG that arises from the second and third lines (from bottom to top) in the left picture, with the potential ancestors drawn in black; time runs backward from right to left; backward time $\beta\in[0,T]$ corresponds to forward time $t=s+T-\beta$.}
    \label{particlepicture}
\end{figure}
\subsection*{Reading off ancestries in the Moran model}
{The \emph{ancestral selection graph} (ASG) was introduced by Krone and Neuhauser in \cite{KroNe97} (see also \cite{NeKro97}) to study the genealogical relations in the diffusion limit of the Moran model with mutation and selection. In what follows, we briefly explain how to adapt this construction to the Moran model in deterministic environment.
\smallskip

Consider a realization of the {IPS} associated to the Moran model in the environment $\zeta\coloneqq (t_k, p_k)_{k \in I}$  in the time interval $[s,s+T]$. Fix a sample of $n$ individuals at time $s+T$ and trace backward in time (from right to left in Fig. \ref{particlepicture}) the lines of their potential ancestors (i.e. the lines that are ancestral to the sample for some type-configuration at time $s$), ignoring the effect of mutations; backward time $\beta\in[0,T]$ corresponds to the forward time $t=s+T-\beta$. We do this as follows. When a neutral arrow joins two individuals in the current set of potential ancestors, the two lines coalesce into a single one at the tail of the arrow. When a neutral arrow hits a potential ancestor from outside the current set of potential ancestors, the hit line is replaced by the line at the tail of the arrow. When a selective arrow hits the current set of potential ancestors, the individual that is hit has two possible parents, the \textit{incoming branch} at the tail and the \textit{continuing branch} at the tip. The true parent depends on the type of the incoming branch, but for the moment we work without types. These unresolved reproduction events can be of two types: a \emph{branching} event if the selective arrow emanates from an individual outside the current set of potential ancestors, and a \emph{collision} event if the selective arrow links two current potential ancestors. Note that at the peak times, multiple lines in the ASG can be hit by selective arrows, and therefore, multiple branching and collision events can occur simultaneously. Mutations are superposed on the lines of the ASG. 
\smallskip

The object arising under this procedure up to time $\beta=T$ is called the \emph{Moran-ASG in $[s,s+T]$ under the environment $\zeta$}. It contains all the lines that are potentially ancestral (ignoring mutation events) to the lines sampled at time $t=s+T$, see Fig. \ref{particlepicture}. Note that, since the number of events occurring in $[s,s+T]$ is almost surely finite ({see \eqref{finitenbofevents}}), the Moran-ASG in $[s,s+T]$ is well-defined.
\smallskip

Given an assignment of types to the lines present in the ASG at time $t=s$, we can extract the true genealogy and determine the types of the sampled individuals at time $t=s+T$. {To this end}, we propagate types forward in time along the lines of the ASG taking into account mutations and reproductions, with the rule that if a line is hit by a selective arrow, the incoming line is the ancestor if and only if it is of type~$0$, see Figure~\ref{fig:peckingorder}. This rule is called the \emph{pecking order}. Proceeding in this way, the types in~$[s,s+T]$ are determined along with the true genealogy.}
\begin{figure}[h!]
	\begin{minipage}{0.23 \textwidth}
		\centering
		\scalebox{.8}{
			\begin{tikzpicture}
			\draw[line width=0.5mm] (0,1) -- (2,1);
			\draw[color=black] (0,0) -- (1,0);
			\draw[-{open triangle 45[scale=2.5]},color=black] (1,0) -- (1,1) node[text=black, pos=.6, xshift=7pt]{};
			\node[above] at (1.8,1) {\tiny D};
			\node[above] at (0.3,1) {\tiny C};
			\node[above] at (0.3,0) {\tiny I};
			\node[left] at (0,1) {\tiny $1$};
			\node[left] at (0,0) {\tiny $1$};
			\node[right] at (2,1) {\tiny $1$};
			\end{tikzpicture}}
	\end{minipage}\hfill
	\begin{minipage}{0.23 \textwidth}
		\centering
		\scalebox{.8}{
			\begin{tikzpicture}
			\draw[line width=0.5mm] (0,1) -- (2,1);
			\draw[color=black] (0,0) -- (1,0);
			\draw[-{open triangle 45[scale=2.5]},color=black] (1,0) -- (1,1) node[text=black, pos=.6, xshift=7pt]{};
			\node[above] at (1.8,1) {\tiny D};
			\node[above] at (0.3,1) {\tiny C};
			\node[above] at (0.3,0) {\tiny I};
			\node[left] at (0,1) {\tiny $0$};
			\node[left] at (0,0) {\tiny $1$};
			\node[right] at (2,1) {\tiny $0$};
			\end{tikzpicture}}
	\end{minipage}\hfill
	\begin{minipage}{0.23 \textwidth}
		\centering
		\scalebox{.8}{
			\begin{tikzpicture}
			\draw[] (0,1) -- (2,1);
			\draw[line width=0.5mm] (0,0) -- (1,0);
			\draw[line width=0.5mm] (1,1) -- (2,1);
			\draw[-{open triangle 45[scale=2.5]},color=black,line width=0.5mm] (1,-0.025) -- (1,1) node[text=black, pos=.6, xshift=7pt]{};
			\node[above] at (1.8,1) {\tiny D};
			\node[above] at (0.3,1) {\tiny C};
			\node[above] at (0.3,0) {\tiny I};
			\node[left] at (0,1) {\tiny $1$};
			\node[left] at (0,0) {\tiny $0$};
			\node[right] at (2,1) {\tiny $0$};
			\end{tikzpicture}}
	\end{minipage}\hfill
	\begin{minipage}{0.23 \textwidth}
		\centering
		\scalebox{.8}{
			\begin{tikzpicture}
			\draw[] (0,1) -- (2,1);
			\draw[line width=0.5mm] (0,0) -- (1,0);
			\draw[line width=0.5mm] (1,1) -- (2,1);
			\draw[-{open triangle 45[scale=2.5]},color=black,line width=0.5mm] (1,-0.025) -- (1,1) node[text=black, pos=.6, xshift=7pt]{};
			\node[above] at (1.8,1) {\tiny D};
			\node[above] at (0.3,1) {\tiny C};
			\node[above] at (0.3,0) {\tiny I};
			\node[left] at (0,1) {\tiny $0$};
			\node[left] at (0,0) {\tiny $0$};
			\node[right] at (2,1) {\tiny $0$};
			\end{tikzpicture}}
	\end{minipage}
	\caption{The descendant line (D) splits into the continuing line (C) and the incoming line (I). The incoming line is ancestral if and only if it is of type~$0$. The true ancestral line is drawn in bold.}
	\label{fig:peckingorder}
\end{figure} 
\vspace{-1cm}
\subsection*{Evolution of the type composition}%\label{s213}
{Consider the set $\Db^\star$ of non-decreasing functions $\om\in\Db$ satisfying
\begin{itemize}
 \item[(i)] for all $s<t\in\Rb$, $\Delta \om(t)\coloneqq \om(t)-\om(t-)\in[0,1)$ and $\sum_{u\in[s,t]}\Delta \om(u)<\infty$,
 \item[(ii)] $\om$ is pure-jump, i.e. for all $s<t$, $\om(t)=\om(s)+\sum_{u\in(s,t]}\Delta \om(u)$.
\end{itemize} 
Note that the set of environments $\zeta\coloneqq (t_k, p_k)_{k \in I}$ satisfying \eqref{summableassumption} can be identified with the set of functions $\om\in \Db^\star$ with $\om(0)=0$. Indeed, for any $\om\in\Db^\star$, the collection of points $\{(t,\Delta \om(t)):\, \Delta\om(t)>0)\}$ is countable and satisfy \eqref{summableassumption}.\footnote{The same collection is obtained if we add a constant to $\om$; this is why we need to fix the value of $\om(0)$.} Conversely, for any $\zeta\coloneqq (t_k, p_k)_{k \in I}$, the function $\om:\Rb\to\Rb$ defined via
\[\om(t)\defeq\sum_{t_k\in(0,t]}p_k\quad\textrm{for $t\geq 0$, and}\quad {\om(t)\defeq-\sum_{t_k\in(t,0]}p_k}\quad\textrm{for $t<0$},\]
belongs to $\Db^\star$.} For this reason, {we often abuse notation} and refer to the elements of $\Db^\star$ as environments. In addition, an environment $\om\in\Db^\star$ is said to be \textit{simple} if $\om$ has only a finite number of jumps in any compact time interval. We denote by $\textbf{0}$ the environment corresponding to $\zeta=\emptyset$ and refer to it as the \emph{null environment}.
\smallskip

We denote by $Z_N^\om(t)$ the number of fit individuals at time $t$ {in a Moran} population of size $N$ subject to environment $\om\in\Db^\star$. We refer to $Z_N^\om\defeq(Z_N^\om(t))_{t\geq 0}$ as \emph{the quenched fit-counting process}. 
%$denote by $\Pb_N^\om$ {the law of the process} $X_N(\om,\cdot)$, and we refer to it as the quenched probability measure.
In particular, $Z_N^\textbf{0}$ is just the continuous-time Markov chain on $[N]_0$ with infinitesimal generator
\[\As_N^0 f(n)\defeq \left[(1+\sigma_N)\frac{n(N-n)}{N}+\theta_N\nu_0(N-n)\right](f(n+1)-f(n))
+ \left[\frac{n(N-n)}{N}+\theta_N\nu_1 n\right](f(n-1)-f(n)).\]

Note that if $\om$ has a jump at time $t$, the number $n(t)$ of individuals placing an offspring is a binomial random variable with parameters $Z_N^\om(t-)$ and $\Delta\om(t)$. Since the $n(t)$ individuals that will be replaced are chosen uniformly at random, the additional number of fit individuals after the reproduction event is a hypergeometric random variable with parameters $N$, $N-Z_N^\om(t-)$ and $n(t)$. Therefore, the dynamic of $Z_N^\om$ is as follows. Recall that in any finite time interval, the number of environmental reproductions is almost surely finite. Thus, we can define $(S_i)_{i\in\Nb}$ as the increasing sequence of times at which environmental reproductions take place. We set $S_0\coloneqq 0$. By construction, $(S_i)_{i\in\Nb}$ is Markovian and its transition probabilities are given by
\[\mathbb{P}(S_{i+1} > t \mid S_i = s ) = \prod_{u \in (s, t]} (1- \Delta \om(u))^N,\quad i\in\Nb_0, 0\leq s\leq t.\] If $Z_N^\om(0)=n\in [N]_0$, then $Z_N^\om$ evolves in $[0,S_1)$ as $Z_N^\textbf{0}$ started at $n$. For $i\in\Nb$, if $Z_N^\om(S_{i}-)=k$, then $Z_N^\om(S_{i})=k+H(N,N-k,\tilde B_i(k))$, where the random variables $H(N,N-k,b)\sim \textrm{Hyp}(N,N-k,b)$, $b\in[k]_0$, and $\tilde B_i(k)$ are independent, and $\tilde B_i(k)$ is a binomial random variable with parameters  $k$ and $\Delta \om(S_i)$ conditioned to be positive. Then, $Z_N^\om$ evolves in $[S_i,S_{i+1})$ as $Z_N^\textbf{0}$ {started at $Z_N^\om(S_{i})$.}
\smallskip

Let us fix $T>0$. We end this section with our first main result, which provides the continuity in $[0,T]$ of the \emph{fit-counting} process with respect to the environment. Note that the restriction of the environment to $[0,T]$ can be identified with an element of
\begin{equation}\label{dbs}
 \Db_T^\star\coloneqq \{\om\in\Db_{0,T}: \om(0)=0,\, \Delta \om(t)\in[0,1)\textrm{ for all $t\in[0,T]$, $\om$ is non-decreasing and pure-jump}\}.
\end{equation}
Moreover, we equip $\Db_T^\star$ with the metric $d_T^\star$ defined in Appendix \ref{A1} {(see \eqref{defdtstar})}.
\begin{theorem}[Continuity]\label{thm2.1}
Let $\om\in\Db_T^\star$ and let $\{\om_k\}_{k\in\Nb}\subset\Db_T^\star$ be such that $d_T^\star(\om_k,\om)\to 0$ as $k\to\infty$. If $Z_N^{\om_k}(0)=Z_N^\om(0)$ for all $k\in\Nb$, then  
\[(Z_N^{\om_k}(t))_{t\in[0,T]}\xRightarrow[k\to \infty]{(d)}(Z_N^\om(t))_{t\in[0,T]}.\]
\end{theorem}
Theorem \ref{thm2.1} is proved in Section \ref{s31}.
\subsection{Moran models in an environment driven by a subordinator} \label{s22} 
In contrast to Section \ref{s21}, we consider here a random environment given by a Poisson point process $(t_i, p_i)_{i \in I}$ on $(-\infty,\infty) \times (0,1)$ with intensity measure $\dd t \times \mu$, where $\dd t$ stands for the Lebesgue measure and $\mu$ is a measure on $(0,1)$ satisfying 
\begin{equation}\label{intmu}
\int_{(0,1)} x \mu(\dd x) < \infty.
\end{equation}
{The latter} implies that {$(t_i, p_i)_{i \in I}$ almost surely satisfies Assumption \eqref{summableassumption}.} In particular, setting $J(t) \coloneqq  \sum_{t_i\in(0,t]} p_i$, for $t\geq 0$, and {$J(t)\defeq -\sum_{t_i\in(t,0]}p_i$}, for $t<0$, we have $J\in\Db^\star$ almost surely. Moreover, by the L\'evy-Ito decomposition, $(J(t))_{t\in\Rb}$ is a pure-jump subordinator with L\'evy measure $\mu$. If the measure $\mu$ is finite, $J$ is a compound Poisson process, and thus, the environment $J$ is almost surely simple. 
\smallskip

We will see in Section \ref{s31} that, using the graphical representation, one can simultaneously construct Moran models for any $\om\in\Db^\star$. Now, consider an independent pure-jump subordinator $(J(t))_{t\geq 0}$, with L\'evy measure $\mu$ on $(0,1)$ satisfying \eqref{intmu}. Thanks to Theorem \ref{thm2.1} the process $Z_N^J\defeq (Z_N^J(t))_{t\geq 0}$ is well defined. We refer to $Z_N^J$ as the \emph{annealed fit-counting process}. By definition, we have \[\Pb(Z_N^J \in \cdot)=\int \Pb(Z_N^\om \in \cdot) \Pb(J \in \dd \om).\] In other words, $\Pb(Z_N^\om \in \cdot)$ is the law of $Z_N^J$ conditionally on a realization $\om$ of the environment (i.e. $\Pb(Z_N^\om \in \cdot) = \Pb(Z_N^J \in \cdot \mid J=\om)$) and is classically referred to as the \textit{quenched measure} while $\Pb(Z_N^J \in \cdot)$ integrates the effect of the random environment and is classically referred to as the \textit{annealed measure}. 
%if $\Pb_N^\om$, $\Pb_N$ and $P$ denote the law of $Z_N^\om$, $Z_N^J$ and $J$, respectively, then 
%\begin{equation}\label{avsqN}
% \Pb_N(\cdot)=\int \Pb_N^\om(\cdot) P(\dd \om).
%\end{equation}
The process $Z_N^J$ is a continuous-time Markov chain on $[N]_0$ with infinitesimal generator

\[\As_N f(n)\coloneqq \As_N^0 f(n) + \int_{(0,1)}\left(\Eb\left[f\left(n+\Hs(N,N-n,B_n(u))\right)\right]-f(n)\right){\mu(\dd u)},\quad n\in [N_0], \]
where $B_n(u)\sim \textrm{Bin}(n,u)$, and for any $i\in[n]_0$, $\Hs(N,N-n,i)\sim\textrm{Hyp}(N,N-n,i)$ are independent. 
\smallskip

The dynamic of the graphical representation is as follows: For each $i,j\in [N]$ with $i\neq j$, selective (resp. neutral) arrows from level $i$ to level $j$ appear at rate $\sigma_N/N$ (resp. $1/N$). For each $i\in [N]$, open circles (resp. crosses) appear at level $i$ at rate $\theta_N\nu_0$ (resp. $\theta_N\nu_1$). For each $k \in [N]$, every group of $k$ lines is subject to simultaneous potential reproductions at rate 
\[\sigma_{N,k}\coloneqq \int_{(0,1)}y^k (1-y)^{N-k}\mu(\dd y),\] resulting in the appearance of $k$ selective arrows from the lines of this group (of potential parents) to the lines of a group of size $k$ (the potential descendants) that is chosen uniformly at random among subsets of size $k$ of the $N$ individuals. The $k$ selective arrows are drawn uniformly at random from the $k$ potential parents to the $k$ potential descendants. Recall that only type $0$ propagates through selective arrows, while both types propagates through neutral arrows. The appearance of a selective arrow is therefore silent when the potential parent at its tail is of type $1$. 

\subsection{The Wright--Fisher diffusion in random environment}\label{s23}
In this section we are interested in the Wright--Fisher diffusion in random environment described in the introduction{ as the solution of the SDE \eqref{WFSDE}. In {Section \ref{S3}} we will see that, indeed, for any $x_0\in[0,1]$, this SDE has a pathwise unique strong solution starting at $x_0$ (see Proposition \ref{eandu}).}
\smallskip

Consider a pure-jump subordinator $J=(J(s))_{s\geq 0}$  with L\'evy measure satisfying \eqref{intmu} and an independent standard Brownian motion $B=(B(s))_{s\geq 0}$. For any $T>0$, the solution of \eqref{WFSDE} in $[0,T]$ is a measurable function of $(B(s),J(s))_{s\in[0,T]}$, which we denote by $F(B,J)$. 
A regular version of the conditional law $\Pb(F(B,J) \in \cdot \mid J=\om)$ of $F(B,J)$ given $J$ is classically referred to as the \textit{quenched probability measure}. It is defined for a.e. realization $\om$ of $J$. $\Pb(F(B,J) \in \cdot)$ integrates the effect of the random environment and is classically referred to as the \textit{annealed measure}. As before, quenched and annealed measure are related via \[\Pb(F(B,J) \in \cdot)=\int \Pb(F(B,J) \in \cdot \mid J=\om)\, \Pb(J \in \dd \om).\]
We write $X$ and $X^\om$ for the solution of \eqref{WFSDE} under the annealed and quenched measures, respectively.
For $\omega$ simple, the process $X^\om$ starting at $x_0$ can be alternatively defined as follows. Denote by $t_1<\cdots<t_k$ {the consecutive jump times} of $\om$ in $[0,T]$ and set $t_0\coloneqq 0$ and $X^\om(0)\coloneqq x_0$. In the intervals $[t_i,t_{i+1})$, $X^\om$ evolves as the solution of \eqref{WFD} starting at $X^\om(t_i)$. Moreover, if $X^\om(t_i-)=x$, then $X^\om(t_i)\coloneqq x+x(1-x)\Delta\om(t_i)$; see Fig. \ref{fig:wfpath} for an illustration. 

\begin{figure}[h!]
\scalebox{0.6}{
\begin{tikzpicture}
\pgfmathsetseed{1337}
\draw[dashed, opacity=0.4] (0,-1)--(0,4);
\draw[dashed, opacity=0.4] (2,-1)--(2,4);
\draw[dashed, opacity=0.4] (8,-1)--(8,4);
\draw[dashed, opacity=0.4] (12,-1)--(12,4);
\draw[dashed, opacity=0.4] (14.5,-1)--(14.5,4);
\draw[dashed, opacity=0.4] (15,-1)--(15,4);
\draw[very thick] (0,4)--(15,4);
\draw[ultra thick,opacity=1,red] (0,-.94)--(1.97,-.94) (2.03,-.94)--(7.97,-0.94) (8.03,-.94)--(11.97,-0.94) (12.03,-.94)--(14.47,-0.94) (14.53,-0.94)--(15,-0.94);
\draw[ultra thick,opacity=1,red] (2,-1)--(2,3.38);
\draw[ultra thick,opacity=1,red] (8,-1)--(8,2.24);
\draw[ultra thick,opacity=1,red] (12,-1)--(12,3.13);
\draw[ultra thick,opacity=1,red] (14.5,-1)--(14.5,3.17);
\draw[very thick] (0,-1)--(15,-1);
\node [left] at (-0.2,4) {$1$};
\node [left] at (-0.2,-1) {$0$};
\node [left] at (-0.2,2) {$x_0$};
\node [right] at (7.2,-2) {$t$};
\node [right] at (-0.2,-1.5) {$0$};
\node [right] at (1.8,-1.5) {$t_1$};
\node [right] at (7.8,-1.5) {$t_2$};
\node [right] at (11.8,-1.5) {$t_3$};
\node [right] at (14.3,-1.5) {$t_4$};
\node [right] at (14.8,-1.5) {$T$};
 \draw[-{angle 45[scale=5]}] (6.5,-1.8) -- (8.5,-1.8) node[text=black, pos=.6, xshift=7pt]{};
\Emmett{100}{0.02}{0.2}{black}{black}{black}
\end{tikzpicture}}
	\caption{An illustration of a path of the process $X^\om$ in the interval $[0,T]$. The red lines represent the jump sizes $\Delta \om$ of the environment $\om$; $t_1$, $t_2$, $t_3$ and $t_4$ are the jump times of $\om$.}
	\label{fig:wfpath}
\end{figure} 

The next result provides the convergence of the type frequency process in the Moran model to the Wright-Fisher diffusion in random environment, as population size grows to $\infty$ and time is suitably accelerated.
{\begin{theorem}[Convergence]\label{thm2.2}
Assume that $N \sigma_N\rightarrow \sigma$ and $N\theta_N\rightarrow \theta$ as $N\to\infty$, for some $\sigma, \theta\geq 0$  (weak selection - weak mutation).
\begin{enumerate}
 \item Let $J$ be a pure-jump subordinator with L\'{e}vy measure $\mu$ in $(0,1)$, and set 
$J_N(t)\coloneqq J(t/N)$, $t\geq 0$. Define the process $(X_N(t))_{t\geq 0}$ via $X_N(t)\defeq Z_N^{J_N}(Nt)/N,$ $t\geq 0$. If $X_N(0)\xrightarrow[N\to\infty]{(d)} x_0$, then
 \[ (X_N(t))_{t\geq 0}\xRightarrow[N\to\infty]{(d)} (X(t))_{t\geq 0},\]
 where $X$ is the unique pathwise solution of \eqref{WFSDE} with $X(0)=x_0$.
 \item Let $\om\in\Db^\star$ be a simple environment and set $\om_N(t)\coloneqq\om(t/N)$, $t\geq 0$. Define the process $(X_N^\om(t))_{t\geq 0}$ via $X_N^\om(t)\defeq Z_N^{\om_N}(Nt)/N,$ $t\geq 0$. If $X_N^\om(0)\xrightarrow[N\to\infty]{(d)} x_0$, then
 \[(X_N^\om(t))_{t\geq 0}\xRightarrow[N\to\infty]{(d)} (X^\om(t))_{t\geq 0},\]
with $X^\om$ starting at $x_0$. 
\end{enumerate}
\end{theorem}}
The proof of Theorem \ref{thm2.2} is given in Section \ref{s32}. The reason for using the environment $J_N$ or $\om_N$ is to compensate that time is sped up by a factor of $N$. In this way, $X_N$ and $X$ share the same environment. 
\begin{remark}
The analogous result of Theorem \ref{thm2.2}-(1) in the context of discrete-time Wright--Fisher models without mutations is covered by the fairly general result \cite[Thm. 3.2]{BCM19} (see also \cite[Thm 2.12]{CSW19}). 
\end{remark}

\begin{remark}
If $J$ is a compound Poisson process, then almost every environment is simple. In this case, according to Theorem \ref{thm2.2}-(2) the quenched convergence holds for almost every environment (with respect to the law of $J$). We conjecture that this is true for general $J$. {In Proposition \ref{q-tight} we show that the sequence $(X_N^\om)_{N\geq 1}$ is tight} for any environment $\om$. Hence, it would suffice to prove the continuity of $\om\mapsto X^\om$ to obtain the desired convergence. Unfortunately, since the diffusion term in \eqref{WFSDE} is not Lipschitz, the standard techniques used to prove this type of result fail. Developing new techniques to cover non-Lipschitz diffusion coefficients is beyond the scope of this paper.
\end{remark}
\subsection{The ancestral selection graph in random/deterministic environment}\label{s24}
The aim of this section is to associate an ASG to the Wright--Fisher diffusion in random/deterministic environment. In contrast to the Moran model setting described in Section \ref{s21}, setting up a graphical representation for the forward process is not straightforward. To circumvent this problem, we proceed as follows. We first consider the graphical representation of a Moran model with parameters $\sigma/N$, $\theta/N$, $\nu_0$, $\nu_1$, and environment $\om_N(\cdot)=\om(\cdot/N)$, and we speed up time by $N$. Next, we sample $n$ individuals at time $T$ and we construct the ASG as in Section \ref{s21}.
\smallskip

Now, replace $\om$ by a pure-jump subordinator $J$ with L\'evy measure $\mu$ supported in $(0,1)$. Note that the Moran-ASG in $[0,T]$ evolves according to the time reversal of $J$. The latter is the subordinator $\bar{J}^T\coloneqq(\bar{J}^T(\beta))_{\beta\in[0,T]}$ with $\bar{J}^T(\beta)\coloneqq J(T)-J((T-\beta)-)$, which has the same law as $J$ (its law does not depend on $T$). A simple asymptotic analysis of the rates and probabilities for the possible events leads to the following definition.

\begin{definition}[The annealed{/quenched} ASG] \label{defannealdasg}
The \emph{annealed ancestral selection graph} $\Gs\defeq(\Gs(\beta))_{\beta\geq 0}$ with parameters $\sigma,\theta,\nu_0,\nu_1$, and environment driven by a pure-jump subordinator with L\'evy measure $\mu$, associated to a sample of size $n$ of the population at time $T$ is the branching-coalescing particle system  starting with $n$ lines and with the following dynamic.
\begin{itemize}
 \item[(i)] {Each} line independently splits at rate $\sigma$ into two, an incoming line and a continuing line.
 \item[(ii)] {Every} given pair of lines independently {coalesces} into a single one at rate $2$.
 \item[(iii)] If $m$ is the current number of lines in the ASG, every group of $k$ lines independently experiences a \emph{simultaneous branching} at rate
 {\begin{equation}\label{smk}
  \sigma_{m,k}\coloneqq \int_{(0,1)}y^k (1-y)^{m-k}\mu(\dd y),
 \end{equation}}
 i.e each line in the group splits into two lines, an incoming line and a continuing line.
 \item[(iv)] {Each} line is independently decorated by a beneficial mutation at rate $\theta \nu_0$.
 \item[(v)] {Each} line is independently decorated by a deleterious mutation at rate $\theta \nu_1$.
\end{itemize}
{Let $\om:\Rb\to\Rb$ be a fixed environment. The \emph{quenched ancestral selection graph} with parameters $\sigma,\theta,\nu_0,\nu_1$, and environment $\om$ of a sample of size $n$ at time $T$ is a branching-coalescing particle system $\Gs_T^\om\defeq(\Gs_T^\om(\beta))_{\beta\geq 0}$ starting at $\beta=0-$ with $n$ lines and evolving as the annealed ASG but with (iii) replaced by
\begin{itemize}
 \item[(iii')] If at time $\beta$, we have $\Delta \om(T-\beta)>0$, then any line splits into two lines, an incoming line and a continuing line, with probability $\Delta \om(T-\beta)$, independently from the other lines.
\end{itemize}}
 See Fig. \ref{fig:backfor} for an illustration of the type-frequency process $X^\om$ and the killed ASG $\Gs_T^\om$.
\end{definition} 
The branching-coalescing system $\Gs_T^\om$ is clearly well-defined for $\omega$ simple. The justification of the previous definition for general environments is more involved and will be given in Section \ref{s41}.
\begin{figure}[t!]
\scalebox{0.7}{
\begin{tikzpicture}
%%%%%forward
\pgfmathsetseed{1337}
\draw[dashed, opacity=0.5] (0,-1)--(0,4);
\draw[dashed, opacity=0.5] (2,-1)--(2,4);
\draw[dashed, opacity=0.5] (8,-1)--(8,4);
\draw[dashed, opacity=0.5] (12,-1)--(12,4);
\draw[dashed, opacity=0.5] (15,-1)--(15,4);
\draw[opacity=0.5] (0,4)--(15,4);
\draw[opacity=0.5] (0,-1)--(15,-1);
\node [left,opacity=0.5] at (-0.2,4) {$1$};
\node [left,opacity=0.5] at (-0.2,-1) {$0$};
\node [left, opacity=0.5] at (-0.2,2) {$x$};
\node [right, opacity=0.5] at (7.2,-2) {$t$};
\node [right, opacity=0.5] at (1.8,-1.5) {$t_0$};
\node [right, opacity=0.5] at (7.8,-1.5) {$t_1$};
\node [right, opacity=0.5] at (11.8,-1.5) {$t_2$};
\node [right, opacity=0.5] at (14.8,-1.5) {$T$};
\node [right, opacity=0.5] at (-0.2,-1.5) {$0$};
 \draw[-{angle 45[scale=5]}, opacity=0.5] (6.5,-1.8) -- (8.5,-1.8) node[text=black, pos=.6, xshift=7pt]{};
\bafo{100}{0.02}{0.2}{black}{black}{black}
%%%%backward

\node [right, opacity=1] at (1.8,4.5) {$T-t_0$};
\node [right, opacity=1] at (7.8,4.5) {$T-t_1$};
\node [right, opacity=1] at (11.8,4.5) {$T-t_2$};
\node [right, opacity=1] at (14.8,4.5) {$0$};
\node [right, opacity=1] at (-0.2,4.5) {$T$};
\node [right, opacity=1] at (7.2,5) {$\beta$};
 \draw[-{angle 45[scale=5]}, opacity=1] (8.5,4.8) -- (6.5,4.8) node[text=black, pos=.6, xshift=7pt]{};
 
%%%horizontal 
\draw[thick] (1,-0.5)--(15,-0.5);
\draw[thick] (3,0.5)--(14.5,0.5);
\draw[thick] (10,1.5)--(13,1.5);
\draw[very thick] (0,2.5)--(12,2.5);
\draw[very thick] (6,0)--(12,0);
\draw[very thick] (0,1.5)--(9,1.5);
\draw[very thick] (5,2)--(8,2);
\draw[very thick] (4,3)--(8,3);
\draw[very thick] (7,1)--(8,1);
\draw[very thick] (0,0)--(2,0);
\draw[very thick] (0,3)--(2,3);
\draw[very thick] (14.5,-0.5)--(14.5,0.5);
\draw[very thick] (13,1.5)--(13,0.5);
\draw[very thick] (12,-0.5)--(12,0);
\draw[very thick] (12,1.5)--(12,2.5);  
\draw[very thick] (9,0.5)--(9,1.5);     
\draw[very thick] (8,0.5)--(8,1);
\draw[very thick] (8,1.5)--(8,2);  
\draw[very thick] (8,2.5)--(8,3);  
\draw[very thick] (7,0.5)--(7,1);  
\draw[very thick] (5,1.5)--(5,2);  
\draw[very thick] (4,0.5)--(4,3);
\draw[very thick] (2,-0.5)--(2,0);
\draw[very thick] (2,2.5)--(2,3);      
  
%%%%mutations              
\node[ultra thick] at (10,1.5) {$\bigtimes$} ;    
\node[ultra thick] at (6,0) {$\bigtimes$} ;  
\node[ultra thick] at (3,0.5) {$\bigtimes$} ; 
\node[ultra thick] at (1,-0.5) {$\bigtimes$} ;        
%%%%backward
\end{tikzpicture}}
	\caption{Illustration of a realization of the type-frequency process $X^\om$ (grey path) and the killed ASG $\Gs_T^\om$ (black lines) embedded in the same picture. Forward time $t$ runs from left to right; backward time $\beta\coloneqq T-t$ runs from right to left. The environment $\om$ jumps at forward times $t_0$, $t_1$ and $t_2$.}
	\label{fig:backfor}
\end{figure} 
\begin{remark}
 {Note that in the Moran model, a neutral arrow appears from line $A$ to line $B$ at rate $1/N$ and from line $B$ to line $A$ at the same rate. Two lines are thus connected by a neutral arrow at rate $2/N$, which explains the rate of coalescence events in (ii).}
\end{remark}

\subsection{Type frequency via the killed ASG}\label{s25}
The aim of this section is to relate the type-$0$ frequency process $X$ to the ASG. To this end, {assume} that the proportion of fit individuals at time $0$ is equal to $x\in[0,1]$. Conditionally on $X(T)$, the probability of sampling independently $n$ unfit individuals at time $T$ equals $(1-X(T))^n$. Now, consider the annealed ASG associated to the $n$ sampled individuals in $[0,T]$ and assign randomly types independently to each line in the ASG at time $\beta=T$ according to the initial distribution $(x,1-x)$. In the absence of mutations, the $n$ sampled individuals are unfit if and only if all the lines in the ASG at time $\beta=T$ are assigned the unfit type (because at any selective event a fit individual can only be replaced by another fit individual). Therefore, if $R(T)$ denotes the number of lines present in the ASG at time $\beta=T$, then conditionally on $R(T)$, the probability that the $n$ sampled individuals are unfit is $(1-x)^{R(T)}$. We would then expect to have \[\Eb[(1-X(T))^n\mid X(0)=x]=\Eb[(1-x)^{R(T)}\mid R(0)=n].\]
Mutations determine the types of some of the lines in the ASG even before we assign types to the lines at time $\beta=T$. Hence, we can prune away from the ASG all the sub-ASGs arising from a mutation event. If in the pruned ASG there is a line ending in a beneficial mutation, we can infer that at least one of the sampled individuals is fit. If all the lines end up in a deleterious mutation, we can infer directly that all the sampled individuals are unfit. In the remaining case, the sampled individuals are all unfit if and only if all the lines present at time $\beta=T$ in the pruned ASG are assigned the unfit type. We use this idea in Section \ref{s42} to construct for a given sample of the population at time $t=T$, branching-coalescing systems $\bar{\Gs}\defeq (\bar{\Gs}(\beta))_{\beta\geq 0}$ and $\bar{\Gs}_T^\om\defeq (\bar{\Gs}_T^\om(\beta))_{\beta\geq 0}$ in the annealed and quenched setting, respectively. Both processes have a cemetery state $\dagger$. The main feature of $\bar{\Gs}$ (resp. $\bar{\Gs}^\om_T$) is that for any $\beta\geq0$, the individuals in the sample are all unfit if and only if $\bar{\Gs}\neq \dagger$ (resp. $\bar{\Gs}^\om_T\neq \dagger$) and all the lines present at time $\beta$ in $\bar{\Gs}$ (resp. $\bar{\Gs}^\om_T$) are unfit. We refer to $\bar{\Gs}$ and $\bar{\Gs}^\om_T$ as the annealed and quenched killed ASG (k-ASG), respectively. 

\subsection*{Moment dualities}
In this section we establish a duality relation between the process $X$ and the line-counting process of the k-ASG.
\smallskip

For each $\beta\geq 0$, we denote by $R(\beta)$ the number of lines present in the {annealed} k-ASG at time $\beta$, with the convention that $R(\beta)=\dagger$ if $\Gsb(\beta)=\dagger$. The process $R\coloneqq (R(\beta))_{\beta\geq 0}$, called the {\textit{annealed}} line-counting process of the k-ASG, is a continuous-time Markov chain with values on $\Nb_0^\dagger\coloneqq \mathbb{N}_0\cup\{\dagger\}$ and infinitesimal generator matrix $Q^\mu_\dagger\coloneqq (q^\mu_\dagger(i,j))_{i,j\in\Nb_0^\dagger}$ defined via
\begin{equation}\label{krates}
 q^\mu_\dagger(i,j)\coloneqq \left\{\begin{array}{ll}
            i(i-1)+i \theta\nu_1 &\text{if $j=i-1$},\\
            \binom{i}{k}(\sigma_{i,k}+\sigma 1_{\{k=1\}}) &\text{if $j=i+k,\, i\geq k\geq 1$},\\
            i\theta\nu_0&\textrm{if $j=\dagger$},\\
            -i(i-1+\theta+\sigma)-\int_{(0,1)}(1-(1-y)^i)\mu(\dd y)&\textrm{if $j=i\in\Nb_0$},
            \end{array}\right.
\end{equation}
{where the coefficients $\sigma_{m,k}$ are defined} {in Eq. \eqref{smk}. All other entries are zero}. 
\smallskip

{Similarly, for $T \in \mathbb{R}$ and {a fixed} environment $\omega\in\Db^\star$, we denote by {$R_T^\om \coloneqq (R_T^\om(\beta))_{ \beta \geq 0}$} the line-counting process associated to the quenched k-ASG $\Gsbq$. The process $R_T^\om$, {called the \textit{quenched} line-counting process of the k-ASG,} is a continuous-time (inhomogeneous) Markov process with values in $\Nb_0^\dagger$. It jumps from $i\in\Nb$ to $j\in\Nb_0^\dagger \setminus \{i\}$ at rate $q^0_\dagger(i,j)$, where $q^0_\dagger$ is the matrix defined in \eqref{krates} with $\mu=0$. In addition, at each time $\beta \geq 0$ with $\Delta\omega(T-\beta)>0$, conditionally on $\{R_T^\om(\beta-)=i\}$, $i\in\Nb$, we have $R_T^\om(\beta) \sim i +  \bindist{i}{\Delta \omega(T-\beta)}$.
If $\theta >0$ and $\nu_0 \in (0,1)$, the states $0$ and $\dagger$ are absorbing for $R$ and $R_T^\om$.}
\smallskip

Let $J$ be a pure-jump subordinator with L\'evy measure $\mu$ supported in $(0,1)$. We write here $X^J$ {instead} of $X$ to stress the dependency of the (strong) solution of \eqref{WFSDE} on the environment $J$. Similarly, we write 
$X^\omega$ for its quenched version (as introduced in Section \ref{s23}). Since in the annealed case, backward and forward environments have the same law, we can construct the line-counting process of the k-ASG as the strong solution of an SDE involving $J$ {and} other four independent Poisson processes encoding the non-environmental events. We denote it by $(R^J(\beta))_{\beta\geq 0}$. The next result establishes a formal relation between $X^J$ and $R^J$: a \emph{reinforced moment duality}, which allows us to derive moment dualities in the annealed and quenched setting (see Fig. \ref{fig:backfor} to visualize forward and backward processes in the same picture).
\begin{theorem}[{Reinforced, annealed and quenched moment dualities}] \label{thm2.3}
For all $x\in[0,1]$, $n\in\mathbb{N}$ and $T\geq 0$, and any function $f\in\Cs^2([0,\infty))$ with compact support,
\begin{equation}\label{rmd}
\mathbb{E}[(1-X^J(T))^n f(J(T))\mid X^J(0)=x]=\mathbb{E}[(1-x)^{R^J(T)} f(J(T))\mid R^J(0)=n],
 \end{equation}
with the convention $(1-x)^\dagger=0$ for all $x\in[0,1]$. In particular, if $f=1$ we recover the moment duality\footnote{We will often drop the superscript $J$ when using this relation, unless we want to emphasize the dependency on $J$.}
 \begin{equation}\label{md}
\mathbb{E}[(1-X^J(T))^n \mid X^J(0)=x]=\mathbb{E}[(1-x)^{R^J(T)} \mid R^J(0)=n].
 \end{equation}
For almost every (with respect to the law of $J$) environment $\om\in\Db^\star$, 
\begin{equation}
\mathbb{E} \left [ (1-X^\omega(T))^n\mid X^\omega(0)=x \right ]= \mathbb{E} \left [ (1-x)^{R_T^\omega(T-)}\mid R_T^\omega(0-)=n \right ].  \label{quenchedual}
\end{equation}
\end{theorem}
We prove \eqref{rmd} and \eqref{md} in Section \ref{s51}. The proof of the quenched duality \eqref{quenchedual} is given in Section \ref{s61}. Moreover, Theorem \ref{thmf1} extends  \eqref{quenchedual} to any simple environments. 
\begin{remark}
For $\theta=0$, \eqref{md} is a particular case of \cite[Lemma 2.14]{CSW19}. 
\end{remark}
\subsection*{Asymptotic type composition}
\begin{figure}[t!]
\scalebox{0.65}{
\begin{tikzpicture}
\pgfmathsetseed{1337}
\draw[dashed, opacity=0.4] (5,-1)--(5,4);
\draw[dashed, opacity=0.4] (15,-1)--(15,4);
\draw[dashed, opacity=0.4] (0,-1)--(0,4);
\draw[very thick] (0,4)--(15,4);
\draw[opacity=0.4] (0,2)--(15,2);
\draw[ultra thick,opacity=1,red] (0,-.94)--(1.97,-.94) (2.03,-.94)--(7.97,-0.94) (8.03,-.94)--(11.97,-0.94) (12.03,-.94)--(15,-0.94) ;
\draw[ultra thick,opacity=1,red] (2,-1)--(2,3.38);
\draw[ultra thick,opacity=1,red] (8,-1)--(8,2.24);
\draw[ultra thick,opacity=1,red] (12,-1)--(12,3.13);
\draw[very thick] (0,-1)--(15,-1);
\node [left] at (-0.2,4) {$1$};
\node [left] at (-0.2,-1) {$0$};
\node [right] at (-0.9,-1.5) {$-(\tau+h)$};
\node [left] at (-0.2,2) {$x$};
\node [right] at (7.2,-2) {$t$};
\node [right] at (4.5,-1.5) {$-\tau$};
\node [right] at (14.8,-1.5) {$0$};
\draw[-{angle 45[scale=5]}] (6.5,-1.8) -- (8.5,-1.8) node[text=black, pos=.6, xshift=7pt]{};
\frompast{100}{0.02}{0.2}{black}{black}{black}
\end{tikzpicture}}
	\caption{The black (resp. grey) path represents a realization of $X^\om$ in the interval $[-(\tau+h),0]$ (resp. $[-\tau,0]$ ) starting at $X^\om(-(\tau+h))=x$ (resp. $X^\om(-\tau)=x$). The jump sizes of the environment $\om$ are depicted in red.}
	\label{fig:wfpast}
\end{figure} 

Assume now that $\theta >0$ and $\nu_0, \nu_1 \in (0,1)$. In particular, the processes $X$ and $X^\omega$ are not absorbed in $\{0,1\}$. We will describe the asymptotic behavior of these processes using Theorem \ref{thm2.3}. The quenched case is particularly delicate, because for a given environment $\om$, $X^\omega(t)$ strongly depends on the environment in the recent past, and only weakly on the environment in the distant past (see Fig. \ref{fig:wfpath}). Hence, unless $\om$ is constant after some fixed time $t_0$ (i.e $\om$ has no jumps after $t_0$), $X^\omega(t)$ will not converge as $t\to\infty$ (see Remark \ref{periodic} for the case of periodic environments). In contrast, for a given environment $\om$ in $(-\infty,0]$, we will see that $X^\om(0)$, conditionally on $X^\om(-\tau)=x$, converges in distribution as $\tau\to\infty$, and we will characterize its law; the setting is illustrated in Fig. \ref{fig:wfpast} (compare with Fig. \ref{fig:wfpath}). To this end, define for $n\in\Nb_0$,
\begin{align}
\pi_n\coloneqq \mathbb{P}(\exists \beta\geq 0: R(\beta)=0 \mid R(0)=n),\quad
\Pi_n(\omega)\coloneqq \mathbb{P}(\exists \beta\geq 0 : R_0^\omega(\beta)=0 \mid R_{0}^\omega(0-)=n), \label{defpin}
\end{align}
and set $\pi_\dagger\coloneqq 0$ and $\Pi_\dagger(\om)\coloneqq 0$. Clearly, $\pi_0=1$ and $\Pi_0(\omega)=1$. 
\begin{theorem}[Asymptotic type frequency] \label{thm2.4}
Assume that $\theta >0$ and $\nu_0, \nu_1 \in (0,1)$. 
\begin{enumerate}
 \item The diffusion $X$ has a unique stationary distribution $\eta_X\in\Ms_1([0,1])$ and {$X(t)\xrightarrow[]{(d)}X(\infty)$} as $t\to\infty$, where $X(\infty)$ denotes a random variable distributed according to $\eta_X$. Moreover, for all $n\in\mathbb{N}$,
\begin{equation}
\mathbb{E}\left[ (1-X(\infty))^n \right]=\pi_n, \label{cvmomentsannealed}
\end{equation}
and the absorption probabilities $(\pi_n)_{n\geq 0}$ satisfy 
\begin{equation}
(\sigma+\theta+n-1)\pi_n=\sigma \pi_{n+1}+(\theta\nu_1+ n-1)\pi_{n-1}+ \frac{1}{n}\sum\limits_{k=1}^n\binom{n}{k} \sigma_{n,k}(\pi_{n+k}-\pi_n),\quad n\in\Nb, \label{recwn}
\end{equation}
where the coefficients $\sigma_{n,k}$, $k\in[n]$, $n\in\Nb$, are defined in Eq. \eqref{smk}.
 \item For almost every (with respect to the law of $J$) environment $\om$ and for any $x \in (0,1)$, the distribution of $X^\omega(0)$ conditionally on $\{ X^\omega(-\tau) = x \}$ has a limit distribution $\Ls^\omega$ as $\tau\to\infty$, which does not depend on $x$. Moreover,
\begin{eqnarray}
\ \int_0^1 (1-y)^n \mathcal{L}^{\omega}(\dd y) = \Pi_n(\omega),\quad n\in\Nb, \label{dualimite}
\end{eqnarray}
and the convergence of moments is exponential, i.e. 
\begin{eqnarray} 
 \ \left | \mathbb{E}\left [ (1-X^\omega(0))^n \mid X^\omega(-\tau)=x \right ] - \Pi_n(\omega) \right | \leq e^{-\theta\nu_0 \tau},\quad n\in\Nb. \label{approxwn}
\end{eqnarray}
\end{enumerate}

\end{theorem}
The setting of Theorem \ref{thm2.4}-(1) is illustrated in Fig. \ref{fig:wfpath} and its proof is given in Section \ref{s51}; the setting of part (2) is illustrated in Fig. \ref{fig:wfpast} and its proof is provided in Section \ref{s61}. Moreover, {Theorem \ref{thmf2}} extends {Theorem} \ref{thm2.4}-(2) to any simple environment. {A refinement of Theorem \ref{thm2.4}-(2) is given in Theorem \ref{thmf3} for simple environments under additional conditions.}

\begin{remark}\label{simpsonindexannealed}
\emph{Simpson's index} is a popular tool for describing population diversity. It represents the probability that two individuals chosen uniformly at random, from the population will be of the same type. In our case it is given by $\Sim(t) \coloneqq  X(t)^2 + (1-X(t))^2$. If the types represent different species, it gives a measure of bio-diversity. If the types represent two alleles of a gene for a given species, it measures \emph{homozygosity}. As a consequence of Theorem \ref{thm2.4}, one can express the moments of $\Sim(\infty)$ {in terms} of the coefficients $(\pi_n)_{n\geq 0}$. In particular, we have 
\[ \mathbb{E}[\Sim(\infty)]=\mathbb{E}[X(\infty)^2 + (1-X(\infty))^2] = 1 - 2 \pi_1 + 2\pi_2. \]
\end{remark}
\begin{remark}\label{periodic}
If $\omega$ is a periodic environment on $[0,\infty)$ with period $T_p > 0$, the proof of Theorem \ref{thm2.4}-(2) yields that, for any $x \in (0,1)$ and $r \in [0,T_p)$, the distribution of $X^\omega(nT_p + r)$, conditionally on $\{ X^\omega(0) = x \}$, has a limit distribution $\mathcal{L}_r^{\omega}$, when $n$ goes to infinity, which is a function of $\om$ and $r$, and does not depend on $x$. Furthermore, $\mathcal{L}_r^{\omega}$ satisfies $\int_0^1 (1-y)^n \mathcal{L}_r^{\omega}(\dd y) = \Pi_n(\omega_r)$, where $\omega_r$ is the periodic environment in $(-\infty,0]$ defined by $\omega_r(t) \coloneqq  \omega(r+t+(\lfloor -t/T_p \rfloor + 1) T_p)$ for any $t \in (-\infty,0]$. The convergence of moments is exponential as in \eqref{approxwn}. 
\end{remark}

\subsection*{Mixed environments} 
\begin{figure}[t!]
\scalebox{0.7}{
\begin{tikzpicture}
\pgfmathsetseed{1337}
\draw[dashed, opacity=0.4] (0,0)--(0,6);
\draw[thick, opacity=1] (9,0)--(9,6);
\draw[dashed, opacity=0.4] (15,0)--(15,6);

%%%deterministic environment
\draw[ultra thick, opacity=1,red] (12,0)--(12,4.8);
\draw[ultra thick, opacity=1,red] (10,0)--(10,3);
\draw[ultra thick, opacity=0.4,red] (2,0)--(2,4.5);
\draw[ultra thick, dashed, opacity=0.4,red] (1,0)--(1,4.8);
\draw[ultra thick, dashed, opacity=0.4,red] (3,0)--(3,3.6);
\draw[ultra thick, dashed, opacity=0.4,red] (6,0)--(6,5.4);

\draw[ultra thick, opacity=0.4,red] (0,0.02)--(0.97,0.02) (1.03,0.02)--(2.97,0.02) (2.03,0.02)--(5.97,0.02) (6.03,0.02)--(8.99,0.02);
\draw[ultra thick, opacity=1,red] (9.01,0.02)--(9.97,0.02) (10.03,0.02)--(11.97,0.02) (12.03,0.02)--(15,0.02);
\draw[] (0,6)--(15,6);
\draw[] (0,0)--(15,0);
\node [left] at (-0.2,6) {$1$};
\node [left] at (-0.2,0) {$0$};
\node [left] at (-0.2,3) {$x$};
\node [right] at (8.5,-0.5) {$-{\tau_\star^{}}$};
\node [right] at (14.8,-0.5) {$0$};
\node [right] at (-.4,-0.5) {$-\tau$};
\various{100}{0.02}{0.2}{0.4}{1}{black}
\pgfmathsetseed{1336}
\variousb{50}{0.02}{0.2}{0.4}{1}{blue}
\end{tikzpicture}}
\caption{Two realizations (blue and black) of the type-frequency process $X^{J\otimes_{\tau_\star^{}}^{}\om}$ in $[-\tau,0]$ starting at $x$. Both realizations share the same deterministic environment $\om$ in $[-{\tau_\star^{}},0]$; the solid red lines in $(-{\tau_\star^{}},0]$ represent the jump sizes of $\om$. The corresponding environments in $[-\tau,-\tau_\star^{}]$ are random; the jump sizes $\Delta J$ associated to the black (resp. blue) realization in $[-\tau,-{\tau_\star^{}})$ are depicted by solid (resp. dashed) red lines.}
\label{fig:mixed}
\end{figure}

We now present an application illustrating the advantage of studying both quenched and annealed {settings}. We consider a population evolving from the distant past in a (stationary) random environment {and analyze} the effect in the type composition at present of a recent perturbation of the environment. To this end, we assume we only know the distribution of the environment before the perturbation. 
\smallskip

In the absence of perturbations, the environment is given by a pure jump subordinator $J$ in $(-\infty,0]$ with L\'evy measure $\mu$ satisfying \eqref{intmu}. The perturbation occurs in $(-\tau_\star^{},0]$ (for some $\tau_\star^{} > 0$) and is given by a deterministic environment $\om$. Let $X^{J\otimes_{\tau_\star^{}}^{} \om}$ be the solution of \eqref{WFSDE} under the environment $J\otimes_{\tau_\star^{}}^{} \om$, which coincides with $J$ and $\om$ in $(-\infty,-{\tau_\star^{}}]$ and $(-{\tau_\star^{}},0]$, respectively; see Fig~\ref{fig:mixed} for an illustration. Recall that $ (R_0^\om(\beta))_{\beta \in [0,{\tau_\star^{}})}$ is the line-counting process associated to the quenched k-ASG (see Section \ref{s25}). We are interested in the distribution of $X^{J\otimes_{\tau_\star^{}}^{} \om}(0)$. The next result provides the moments of this random variable.
\begin{proposition}\label{prop2.5}
Assume that $\theta>0$ and $\nu_0,\nu_1\in(0,1)$. For any $\tau_\star^{}>0$, $n\in\Nb$, $x\in[0,1]$, and almost every (with respect to the law of $J$) $\om\in\Db^\star$, we have
\begin{equation}\label{mix1}
\lim_{\tau\to\infty}\Eb\left[(1-X^{J\otimes_{\tau_\star^{}}^{} \om}(0))^n\mid X^{J\otimes_{\tau_\star^{}}^{} \om}(-\tau)=x\right]=\E\left[\pi_{R_0^{\om}({\tau_\star^{}}-)}\mid R_0^\om(0-)=n\right].
\end{equation}
\end{proposition}
Proposition \ref{prop2.5} is proved in Section \ref{s61}; a refinement of this result is given for simple environments under the additional condition $\sigma=0$ in Proposition  \ref{mixf4}.

\subsection{Ancestral type via the pruned lookdown ASG}\label{s26}
In this section we are interested in the type distribution at present of the individuals that will be successful in the long run. This distribution may substantially differ from the type composition at present and show a bias towards the fit type. 
\smallskip

Consider a sample of $n$ individuals at some time $T$ in the future and trace their ancestral lines using the ASG. We will see in Section \ref{s52} that the number of lines in the ASG is positive recurrent (see Lemma \ref{pr}). Hence, the ASG has bottlenecks and, if $T$ is sufficiently large, the $n$ individuals share a common ancestor at time $0$. Assigning types to the lines in the ASG at time $0$ and propagating the types along using the pecking order, we determine the types in the sample as well as the true genealogy. In particular, we obtain the type of the common ancestor of the sample. What it means for $T$ to be sufficiently large depends on $n$ and on the realization of the ASG, but this dependency vanishes as $T\to \infty$. Because we are interested in the type of the individual that is successful in the long run, we can work under this limit consideration. In what follows we formalize this idea.
\smallskip
 
Consider a realization $\Gs_{[0,T]}^{}\coloneqq (\Gs(\beta))_{\beta\in[0,T]}$ of the annealed ASG in $[0,T]$ started with one line, representing an individual sampled at forward time $T$. If $t$ denotes forward time, we {set} $\beta=T-t$ to denote the backward time (see Fig. \ref{fig:backfor}). For $\beta\in[0,T]$, let $V_\beta$ be the set of lines present at time $\beta$ in $\Gs_{[0,T]}^{}$. Consider a function $c:V_T\to\{0,1\}$ representing an assignment of types to the lines in $V_T$. Given $\Gs_{[0,T]}^{}$ and $c$, we propagate types (forward in time) along the lines of $\Gs_{[0,T]}$ keeping track, at any time $\beta\in[0,T]$, of the true ancestor in $V_T$ of each line in $V_\beta$. We denote by $a_c(\Gs_{[0,T]}^{})$ the type of the ancestor in $V_T$ of the single line in $V_0$. Assume that, under $\Pb_x$, $c$ assigns independently to each line type $0$ with probability $x$ and type $1$ with probability $1-x$. {The \emph{annealed ancestral type distribution at time $T$} is 
\[h_T(x)\coloneqq \Pb_x(a_c(\Gs_{[0,T]})=0),\quad x\in[0,1].\]} 
{In the quenched setting, we proceed in the same way, but using $\Gs_{[0,T]}^\om\coloneqq (\Gs_T^\om(\beta))_{\beta\in[0,T]}$, the quenched ASG in $[0,T]$ in the environment $\om$ of an individual sampled at time $T$, instead of $\Gs_{[0,T]}^{}$. The \emph{quenched ancestral type distribution at time $T$} is 
\[h_T^\om(x)\coloneqq \Pb_x(a_c(\Gs_{[0,T]}^\om)=0),\quad x\in[0,1],\]
where under $\Pb_x$, $c$ assigns independently to each line present in $\Gs_{[0,T]}^\om$ at time $\beta=T$ type $0$ with probability $x$ and type $1$ with probability $1-x$.}
 \smallskip

{In the absence of mutations, the ancestor of the sampled individual is fit if and only if there is at least one fit line in the ASG having type $0$ at time $\beta=T$. In the presence of mutations, determining the type of the ancestor is more involved. In \cite{LKBW15} the ancestral type distribution was obtained for the null environment using the line-counting process of a pruned version of the ASG, called \emph{the pruned lookdown ASG} (pLD-ASG). In Section \ref{s43} we 
generalize this construction to incorporate the effect of the environment. The main feature of the pLD-ASG is that the type of the ancestor at time $t=0$ of the sampled individual at time $t=T$ is $0$ if and only if there is at least one line in the pLD-ASG at time $\beta=T$ that has type $0$ (see Lemma \ref{pruningdonnehtx}). Hence,  $h_T(x)$ and $h_T^\om(x)$ can be represented via the corresponding line-counting processes (which can be easily inferred from the description of the pLD-ASG given in Section \ref{s52}).}
\smallskip

The line-counting process of the annealed pLD-ASG, denoted by $L\coloneqq (L(\beta))_{\beta \geq 0}$, is a continuous-time Markov chain with values on $\mathbb{N}$ and generator matrix $Q^\mu\coloneqq (q^\mu(i,j))_{i,j\in\Nb}$ given by
\begin{equation}\label{kratespldasg}
q^\mu(i,j)\coloneqq \left\{\begin{array}{ll}
            i(i-1)+(i-1) \theta\nu_1 + \theta\nu_0 &\text{if $j=i-1$},\\
            i(\sigma + \sigma_{i,1}) &\text{if $j=i+1$},\\
            \binom{i}{k}\sigma_{i,k} &\text{if $j=i+k,\, i\geq k\geq 2$},\\
            \theta\nu_0&\textrm{if $1 \leq j < i-1$},\\
            -(i-1)(i+\theta)-i\sigma-\int_{(0,1)}(1-(1-y)^i)\mu(\dd y)& \textrm{if $j=i$}.\\
            \end{array}\right.
\end{equation}
where $\sigma_{m,k}$  is defined in \eqref{smk}; all other entries are $0$. 
\smallskip

The pLD-ASG associated to $\om\in\Db^\star$ is well-defined and contains almost surely finitely many lines at any time; we show this in Section \ref{s52}. The corresponding line-counting process $(L_T^\om(\beta))_{ \beta \geq 0}$ started at time $T$, is a continuous-time (inhomogeneous) Markov process with values in $\mathbb{N}$. It jumps from $i\in\Nb$ to $j\in\Nb \setminus \{i\}$ at rate $q^0(i,j)$, where $q^0$ is the matrix defined in \eqref{kratespldasg} with $\mu=0$, and in addition, at each time $\beta \geq 0$ such that $\Delta\omega(T-\beta)>0$, conditionally on $\{L_T^\omega(\beta-)=i\}$, $i\in\Nb$, we have $L_T^\omega(\beta) \sim i+\bindist{i}{\Delta \omega(T-\beta)}$. 
\smallskip

Now, we state the main result of this section describing the asymptotic behavior of $h_T(x)$ and $h_T^\om(x)$.
\begin{theorem}[Ancestral type distribution]\label{thm2.6} The following assertions hold:
 \begin{enumerate}
  \item The process $L$ admits a unique stationary distribution $\eta_L$. Moreover, if $L(\infty)$ is a random variable distributed {according to} $\eta_L$, then $L(T)\xrightarrow[]{(d)}L(\infty)$ as $T\to\infty$. In particular, $h(x)\coloneqq \lim_{T\to\infty} h_T(x)$ is well-defined, and
\begin{eqnarray}
h(x)  = \sum_{n \geq 0} x(1-x)^{n} a_n, \label{represh(x)tailpldasg}
\end{eqnarray}
where the coefficients $a_n\defeq \mathbb{P}(L(\infty) > n)$, $n\in\Nb_0$, satisfy the following recursion\footnote{ in the case $\mu=0$, the recursion is known as Fernhead's recursion.}
 \begin{equation}\label{fr}
(\sigma +\theta +n+1 )\, a_n=  \sigma\, a_{n-1}+(\theta\nu_1+n+1)\,a_{n+1} + \frac{1}{n}\sum\limits_{j=1}^{n} \gamma_{n+1,j}\, (a_{j-1}-a_{j}),\quad n\in\Nb,
\end{equation}
where $\gamma_{i,j}\coloneqq \sum_{k=i-j}^{j}\binom{j}{k}\sigma_{j,k}$ {if $1 \leq j<i\leq 2j$ and $\gamma_{i,j}\coloneqq 0$ otherwise.}
\item Assume that $\theta\nu_0 > 0$. For any $n \in \Nb$, the distribution of $L_T^\om(T-)$ conditionally on $\{ L_T^\omega(0-) = n \}$ has a limit distribution $\mu^\om\in\Ms_1(\Nb)$ as $T\to\infty$,  which does not depend on $n$. In particular, $h^{\omega}(x)\coloneqq \lim_{T\to\infty}h_T^\om(x)$ is well-defined and 
\begin{equation}\label{pldasgatdtpsinftyq}
h^{\omega}(x)= 1- \sum_{n=1}^\infty \mu^\om(\{n\})(1-x)^{n}.
\end{equation}
Moreover, for any $x \in [0,1]$ and $t > 0$, 
\begin{eqnarray}
\left | h^{\omega}(x) - h^{\omega}_T(x) \right | \leq 2e^{-\theta\nu_0 T}. \label{approxh(x)4}
\end{eqnarray}
 \end{enumerate}

\end{theorem}
The proof of Theorem \ref{thm2.6}-(1) is given in Section \ref{s52}; part (2) is proved in Section \ref{s62}. 
Theorem \ref{thmfa} extends part (2) to the case $\theta\nu_0=0$ for simple environments under additional conditions. A refinement of Theorem \ref{thm2.6}-(2) is given in Theorem \ref{thmf4} for simple environments under additional conditions.
\smallskip

In the case $\theta=0$, Theorem \ref{thm2.6} yields the following result about the boundary behavior of $X$.
\begin{corollary}[Accessibility of the boundaries]\label{cor2.7}
If $\theta=0$, then for any $T>0$ and $x\in[0,1]$, 
\[h_T(x)=\mathbb{E} [X(T) \mid X(0) = x].\]
Moreover, {conditionally} on $\{X(0)=x\}$, $X(T)$ converges almost surely as $T\to\infty$ to a Bernoulli random variable with parameter $h(x)$. In particular, the absorbing states $0$ and $1$ are both accessible from any $x\in(0,1)$. %The analogous result holds in the quenched case.
\end{corollary}
\begin{remark}
Corollary \ref{cor2.7} is not a direct consequence of \cite[Thm. 3.2]{CSW19}, whose statement does not cover SDEs with a diffusion term {(the term $\sqrt{2X(t)(1-X(t))}\dd B(t)$)}.   
\end{remark}
We close this section with an application of our results to the comparison of the (isolated) effects of the environment and of (genic) selection. To this end, we fix a non-zero measure $\mu$ on $(0,1)$ satisfying \eqref{intmu} and we consider two models, both without mutations. The first model has selection parameter 
\begin{equation}\label{shape}
\sigma=\sigma_\mu\coloneqq \int_{(0,1)}y\mu(\dd y),
\end{equation}
and no environment (i.e. in \eqref{WFSDE} we take $S(t)\coloneqq \sigma_\mu t$). The second one has selection parameter $\sigma=0$ and an environment given by a subordinator with L\'evy measure $\mu$ (i.e. in \eqref{WFSDE} we take $S(t)\coloneqq J(t)$). We will use the superscript "sel" (resp. "env") to refer to the first (resp. second) model. 

For $n\in\Nb$, set $\rho_n^{\rm{env}}\coloneqq \Pb(L^{\rm{env}}(\infty)=n)$ and $\rho_n^{\rm{sel}}\coloneqq \Pb(L^{\rm{sel}}(\infty)=n)$. Consider the probability generating functions 
\[p^{\rm{env}}(z)\coloneqq \sum_{n=1}^\infty \rho_n^{\rm{env}}z^{n}\quad\textrm{and}\quad p^{\rm{sel}}(z)\coloneqq \sum_{n=1}^\infty \rho_n^{\rm{sel}}z^{n},\quad z\in[0,1].\]
Note that $p^{\rm{env}}(z)=1-h^{\rm{env}}(1-z)$ and $p^{\rm{sel}}(z)=1-h^{\rm{sel}}(1-z)$.
\begin{proposition}\label{comparison}
For any non-zero measure $\mu$ on $(0,1)$ satisfying \eqref{intmu} we have
\[\rho_1^{\rm{env}}>\rho_1^{\rm{sel}}=\frac{\sigma_\mu}{e^{\sigma_\mu}-1}\quad\textrm{and}\quad p^{\rm{env}}(z)\leq \frac{\rho_1^{\rm{env}}}{\rho_1^{\rm{sel}}}\, p^{\rm{sel}}(z)=\rho_1^{\rm{env}}\,\left(\frac{e^{\sigma_\mu z}-1}{\sigma_\mu}\right),\quad z\in[0,1].\]
In particular, there is $x_c\in(0,1)$ such that, for $x\in[x_c,1)$,
\[h^{\rm{env}}(x)=\Pb\left(\lim_{t\to\infty}X^{\rm{env}}(t)=1 \mid X^{\rm{env}}(0)=x\right)<\Pb\left(\lim_{t\to\infty}X^{\rm{sel}}(t)=1 \mid X^{\rm{sel}}(0)=x\right)=h^{\rm{sel}}(x).\]
\end{proposition}
\begin{remark}
As a consequence of Proposition~\ref{comparison} one recovers the classical result of Kimura \cite{Ki62}
\[h^{\rm{sel}}(x)=\frac{1-e^{-\sigma_\mu z}}{1-e^{-\sigma_\mu}},\quad x\in[0,1].\] 
\end{remark}
\begin{remark}
Consider a Wright--Fisher diffusion with no mutations and selection parameter $\sigma$, evolving in an environment with L\'evy measure $\mu$. The quantity $\sigma_\mu$ in \eqref{shape} corresponds to the quantity $\alpha_{\mathfrak{s}}$ in \cite{CSW19}. As shown there, $\alpha_{\mathfrak{s}}$ is not sufficient to fully describe the strength of the environment; one also needs to know the \emph{shape of rare selection}, which is defined as $\alpha^*\coloneqq\int_{(0,1)}\log(1+y)\mu(\dd y)/\alpha_{\mathfrak{s}}$. The joint action of weak selection and the environment is then described by the quantity $\alpha_{\rm{eff}}\coloneqq \sigma+\alpha_{\mathfrak{s}}\alpha^*$, which is called the \emph{effective strength of selection}. The main result in \cite{CSW19} establishes that both boundaries are accessible if and only if $\alpha_{\rm{eff}}$ is smaller than a quantity $\beta^*$ coding for neutral reproductions ($\beta^*=\infty$ in our case). 
\end{remark}
The proof of Corollary \ref{cor2.7} and Proposition~\ref{comparison} are given in Section \ref{s52}.
%%%%%%%%%%%%%%%%%%%%%%%%%%%%%%%%%%%%%%%%%%%%%%%%%%%%%%%%%%%%%%%%%%%%%%%%%%%%%%%%%%%%%%%%%%%%%%%%
%%%%%%%%%%%%%%%%%%%%%%%%%%%%%%%%%%%%%%%%%%%%%%%%%%%%%%%%%%%%%%%%%%%%%%%%%%%%%%%%%%%%%%%%%%%%%%%%
%%%%%%%%%%%%%%%%%%%%%%%%%%%%%%%Section 3%%%%%%%%%%%%%%%%%%%%%%%%%%%%%%%%%%%%%%%%%%%%%%%%%%%%%%%%
%%%%%%%%%%%%%%%%%%%%%%%%%%%%%%%%%%%%%%%%%%%%%%%%%%%%%%%%%%%%%%%%%%%%%%%%%%%%%%%%%%%%%%%%%%%%%%%%
%%%%%%%%%%%%%%%%%%%%%%%%%%%%%%%%%%%%%%%%%%%%%%%%%%%%%%%%%%%%%%%%%%%%%%%%%%%%%%%%%%%%%%%%%%%%%%%%
\section{Moran models and Wright--Fisher processes}\label{S3}
This section is devoted to the proofs of Theorem \ref{thm2.1} and Theorem \ref{thm2.2} and other related results. 
\subsection{Results related to Section \ref{s21}:  continuity with respect to the environment}\label{s31}
\subsection*{Graphical representation}
We start by making more precise the description of the graphical representation of the Moran model as an IPS. This will allow us to decouple the randomness of the model coming from the initial type configuration, the one coming from mutations and reproductions, and the one coming from the environment. Non-environmental events are as usual encoded via a family of independent Poisson processes
\[\Lambda\coloneqq \{\lambda_{i}^{0},\lambda_{i}^{1},\{\lambda_{i,j}^{\vartriangle},\lambda_{i,j}^{\blacktriangle}\}_{j\in{[N] \setminus \{i\}}} \}_{i\in[N]},\]
where: (a) for each $i,j\in [N]$ with $i\neq j$, $(\lambda_{i,j}^{\vartriangle}(t))_{t\in\Rb}$ and $(\lambda_{i,j}^{\blacktriangle}(t))_{t\in\Rb}$ are Poisson processes with rates $\sigma_N/N$ and $1/N$, respectively, and (b) for each $i\in [N]$, $(\lambda_{i}^{0}(t))_{t\in\Rb}$ and $(\lambda_{i}^{1}(t))_{t\in\Rb}$ are Poisson processes with rates $\theta_N\nu_0$ and $\theta_N\nu_1$, respectively. We call $\Lambda$ the \textit{basic background}. The environment introduces a new independent source of randomness into the model, that we describe via the collection
\[\Sigma\coloneqq  \{(U_i(t))_{i\in[N], t\in\Rb}, (\tau_A^{}(t,\cdot))_{A\subset[N], t\in\Rb}\},\]
where: (c) $(U_i(t))_{i\in[N], t\in\Rb}$ is a $[N] \times \Rb-$indexed family of i.i.d. random variables with $U_i(t)$ being uniformly distributed on $[0,1]$, and (d) $(\tau_A(t,\cdot))_{A\subset[N], t\in\Rb}$ is a family of independent random variables with $\tau_A(t,\cdot)$ being uniformly distributed on the set of injections from $A$ to $[N]$. We call $\Sigma $ the \textit{environmental background}.  We assume that basic and environmental backgrounds are independent and we call $(\Lambda,\Sigma)$ the \textit{background}.
\smallskip

Recall that in the graphical representation individuals are represented by horizontal lines at levels $i\in [N]$ {(see Fig. \ref{particlepicture})}. The random appearance of selective and neutral arrows, circles and crosses is prescribed by the background as follows. At the arrival times of $\lambda_{i,j}^{\vartriangle}$ (resp. $\lambda_{i,j}^{\blacktriangle}$), we draw open (resp. filled) head arrows from level $i$ to level $j$. At the arrival times of $\lambda_{i}^{0}$ (resp. $\lambda_{i}^{1})$, we draw an open circle (resp. a cross) at level $i$. Now, given an environment $\zeta\coloneqq (t_k, p_k)_{k \in I}$ satisfying \eqref{summableassumption}, we define, for each $k\in I$,
\[I_{\zeta}(k)\coloneqq \{i\in[N]:U_i(t_k)\leq p_k\}\quad\textrm{and}\quad n_{\zeta}(k)\coloneqq |I_{\zeta}(k)|,\]
and we draw, at time $t_k$, for each $i\in I_{\zeta}(k)$ an open head arrow from level $i$ to level $\tau_{I_{\zeta}(k)}^{}(t,i)$. 
\subsection*{Continuity with respect to the environment}
Now, we embark on the proof of Theorem \ref{thm2.1}, which states the continuity of the type composition in a Moran model with respect to the environment. The paths of the fit-counting process are considered as elements of $\Db_{0,T}$, which is endowed with the $J_1$-Skorokhod topology, i.e. the topology induced by the Billingsley {metric} $d_T^0$ defined in {\eqref{defdt0}}. Recall also that the restriction of an environment to $[0,T]$ is described by means of a function in $\Db_T^\star$ (see \eqref{dbs}), which is endowed with the topology induced by the metric $d_T^\star$ defined in {\eqref{defdtstar}}.
\smallskip

Let us denote by $\mu_N(\om)$ the law of $(Z_N^\om(t))_{t\in[0,T]}$ (recall that $Z_N^\om(t)$ is the number of fit individuals at time $t$ in a Moran population of size $N$ subject to environment $\om$). Theorem \ref{thm2.1} states the continuity of the mapping $\om\mapsto \mu_N(\om)$, where the set of probability measures on $\Db_{0,T}$ is equipped with the topology of weak convergence of measures. We will use that the topology of weak convergence of probability measures on a complete and separable metric space $(E,d)$ is induced by the metric $\varrho_E$ {defined in \eqref{defblm}.} 
\smallskip

First, we get rid of the small jumps of the environment. To this end, we introduce the following notation. For $\delta>0$ and $\om\in\Db_T^\star$, we define $\om^\delta, \om_\delta\in \Db_T^\star$ via
\begin{align}
\om^\delta(t)\coloneqq \sum_{u\in[0,t]:\Delta \om(u)\geq \delta} \Delta \om(u)\quad\textrm{and}\quad \om_\delta(t)\coloneqq \sum_{u\in[0,t]:\Delta \om(u)< \delta} \Delta \om(u).\label{defomdelta}
\end{align}
Clearly, $\om^\delta$ is simple and $\om=\om^\delta+\om_\delta$. Moreover, $\om_\delta \to 0$ pointwise as $\delta\to 0$, and hence for any $t\in[0,T]$, 
\[d_t^\star(\om,\om^\delta)\leq \sum_{u\in[0,T]}\lvert \Delta \om(u)-\Delta \om^\delta(u)\rvert= \om_\delta(T) \xrightarrow[\delta\to 0]{}0.\]
In addition, for $\om\in\Db_T^\star$, $n\in\Nb$ and $\vec{r}\coloneqq (r_i)_{i\in[n]}\in[0,T]^n$, we denote by $\mu_N^{\vec{r}}(\om)$ the law of $(Z_N^{\om}(r_i))_{i\in[n]}$, where $[0,N]^n$ is equipped with the distance $d_1$ defined via \begin{equation}\label{defd1} 
d_1( (x_i)_{i\in[n]}, (y_i)_{i\in[n]}) \coloneqq \sum_{i\in[n]}|x_i - y_i|.
\end{equation}
\begin{proposition}\label{sjumps}
Let $\om\in\Db_T^\star$. Assume that, for any $\delta>0$, we have $Z_N^{\om^\delta}(0)=Z_N^{\om}(0)$, then
\begin{equation}\label{onedimbound}
\varrho_{[0,T]^n}^{}(\mu_N^{\vec{r}}(\om^\delta),\mu_N^{\vec{r}}(\om))\leq nN\,\om_\delta(r_*)e^{\sigma_N r_*+ \om(r_*)}, \ \forall \vec{r}\in[0,T]^n, n\in\Nb,
\end{equation}
where $r_*\coloneqq \max_{i\in[n]}r_i$. Moreover, 
\begin{equation}\label{unibound}
\varrho_{\Db_{0,T}}^{}(\mu_N(\om^\delta),\mu_N(\om))\leq N\om_\delta(T)e^{(1+\sigma_N) T+\om(T)}.
\end{equation}
In particular,
\[(Z_N^{\om^\delta}(t))_{t\in[0,T]}\xrightarrow[\delta\to 0]{(d)}(Z_N^{\om}( t))_{t\in[0,T]}.\]
\end{proposition}
\begin{proof}
For $\delta>0$, we couple in $[0,T]$ a Moran model with parameters $(\sigma_N,\theta_N,\nu_0,\nu_1)$ and environment $\om$ to a Moran model with parameters $(\sigma_N,\theta_N,\nu_0,\nu_1)$ and environment $\om^\delta$ (both of size $N$) by using: (1) the same initial type configuration, (2) the same basic background, and (3) the same environmental background. 
\smallskip

{For any $t\in[0,T]$ and $a,b\in\{0,1\}$, we denote by $Y_N^{a,b}(t)$ the number of individuals that at time $t$ meet the following two requirements: 1) having type $a$ under the environment $\om$, 2) having type $b$ under the environment $\om^\delta$.} Clearly, we have
\[\lvert Z_N^{\om^{\delta}}(t)-Z_N^{\om}(t)\rvert=\lvert Y_N^{1,0}(t) - Y_N^{0,1}(t) \rvert\leq Y_N^{1,0}(t) + Y_N^{0,1}(t)\coloneqq Y_N^{\neq}(t).\]
Note that $Y_N^{\neq}(t)$ is the number of individuals that have a different type at time $t$ under $\om$ and $\om^\delta$. {In particular, we have $Y_N^{\neq}(t) \leq N$ a.s.} Let us assume that at time $t$ a graphical element arises in the basic background, {i.e. $t$ is an arrival time of one of the Poisson processes in the family $\Lambda$.} If the graphical element corresponds to a mutation event, then $Y_N^{\neq}(t)\leq Y_N^{\neq}(t-)$. If the graphical element is a neutral reproduction, we have
\[\Eb\left[Y_N^{\neq}(t)\mid Y_N^{\neq}(t-)\right]=Y_N^{\neq}(t-)+\frac{1}{N}Y_N^{\neq}(t-)(N-Y_N^{\neq}(t-))-\frac{1}{N}(N-Y_N^{\neq}(t-))Y_N^{\neq}(t-)=Y_N^{\neq}(t-).\]
{If the graphical element corresponds to a selective event, then $Y_N^{\neq}(t)$ can increase by $1$ only if the individual at the tail of the arrow has a different type at time $t$ under $\om$ and $\om^\delta$.} We thus have \[\Eb\left[Y_N^{\neq}(t)\mid Y_N^{\neq}(t-)\right]\leq\left(1+\frac1N\right)Y_N^{\neq}(t-).\]
 
Now, let $0\leq s<t\leq T$ and assume that there are {neither jumps of $\om^\delta$ nor} selective events in $(s,t)$. In particular, in $(s,t)$ only the population driven by $\om$ is affected by the environment. Moreover, since {neutral reproduction and mutation} events do not increase the expected value of $Y_N^{\neq}$, we obtain \[\Eb\left[Y_N^{\neq}(t-)\mid Y_N^{\neq}(s)\right]\leq Y_N^{\neq}(s) +N\sum_{u\in(s,t)}\Delta \om(u).\]

In addition, if at time $t$ the environment $\om^\delta$ jumps (there are only finitely many of these jumps), then \[\Eb\left[Y_N^{\neq}(t)\mid Y_N^{\neq}(t-)\right]\leq Y_N^{\neq}(t-)(1+\Delta \om (t)).\]

Let {$0\leq t_1<\cdots<t_m\leq T$} be the jump times of $\om^\delta$. From the previous discussion, we obtain
\begin{equation}\label{auxbound}
 \Eb\left[Y_N^{\neq}(t_{i+1})\mid Y_N^{\neq}(t_i)\right]\leq \Eb\left[\left(1+\frac1N\right)^{K_i} \right]\left(Y_N^{\neq}(t_i)+N\epsilon_i(\delta)\right)(1+\Delta \om(t_{i+1})),
 \end{equation}
where $\epsilon_i(\delta)\coloneqq \sum_{u\in(t_i,t_{i+1})}\Delta \om(u)$ and $K_i$ is the number of selective events in $(t_i,t_{i+1})$. Note that $K_i$ has a Poisson distribution with parameter $N\sigma_N(t_{i+1}-t_i)$. Hence,
 \[\Eb\left[Y_N^{\neq}({t_{i+1}})\mid Y_N^{\neq}(t_i)\right]\leq e^{\sigma_N(t_{i+1}-t_i)}\left(Y_N^{\neq}(t_i)+N\epsilon_i(\delta)\right)(1+\Delta \om(t_{i+1})).\]
Iterating this formula and using that $Y_N^{\neq}(0)=0$ yields 
\begin{equation}\label{boundMt}
 \Eb\left[Y_N^{\neq}(t)\right]\leq e^{\sigma_N t} N\om_\delta(t)\prod_{t_i\leq t}(1+\Delta \om(t_i))\leq N\,\om_\delta(t)\,e^{\sigma_N t+\sum_{u\in[0,t]}\Delta \om(u)}.
\end{equation}
Recall the definition of the space $\textrm{BL}(E)$ from Appendix \ref{A2} {and that $[0,N]^n$ is equipped with the distance $d_1$ defined in \eqref{defd1}.} For any $n \geq 1$ and $F\in \textrm{BL}([0,N]^n)$, we have 
\[ \left\lvert \int F \dd\mu_N^{\vec{r}}(\om^\delta)- \int F \dd\mu_N^{\vec{r}}(\om)\right\rvert = \left\lvert \mathbb{E} \left [ F((Z_N^{\om^\delta}(r_j))_{j\in[n]}) \right ] - \mathbb{E} \left [ F((Z_N^\om(r_j))_{j\in[n]}) \right ] \right\rvert. \]
Hence, if $\lVert F\rVert_{\textrm{BL}}\leq 1$ {(see \eqref{defnormbl} for the definition of $\lVert \cdot\rVert_{\textrm{BL}}$)} and we couple {$Z_N^{\om^{\delta}}(t)$} and $Z_N^\om(t)$ as before, we get that 
\begin{align*}
\left\lvert \mathbb{E} \left [ F((Z_N^{\om^\delta}(r_j))_{j\in[n]}) \right ] - \mathbb{E} \left [ F((Z_N^\om(r_j))_{j\in[n]}) \right ] \right\rvert & \leq \left\lvert \mathbb{E} \left [ d_1((Z_N^{\om^\delta}(r_j))_{j\in[n]}, (Z_N^\om(r_j))_{j\in[n]}) \right ] \right\rvert \\ 
& = \mathbb{E} \left [ \sum_{j\in[n]}|Y_N^{\neq}(r_j)| \right ] \leq \sum_{j\in[n]} N\om_\delta(r_j)e^{\sigma_N r_j+\sum_{u\in[0,r_j]}\Delta \om(u)}, 
\end{align*}
where the last bound comes from \eqref{boundMt} applied at each $r_j$. Taking the supremum over all $F\in \textrm{BL}([0,N]^n)$ with $\lVert F\rVert_{\textrm{BL}}\leq 1$ and using the definition of the distance $\varrho_{[0,N]^n}$ in \eqref{defblm} we get \eqref{onedimbound}.
Now, define $Y_N^*(t)\coloneqq \sup_{u\in[0,t]}Y_N^{\neq}(u)$. If at time $t$ a neutral event occurs, then 
\[\Eb[Y_N^*(t)\mid Y_N^*(t-)]\leq\left(1+\frac1N\right)Y_N^*(t-).\]
Other events can be treated as before, leading to \eqref{auxbound} with $K_i$ being this time the number of selective and neutral events in $(t_i,t_{i+1})$. Hence, Eq. \eqref{unibound} follows similarly as \eqref{onedimbound}. The convergence of $Z_N^{\om^\delta}$ towards $Z_N^\om$ is a direct consequence of \eqref{unibound}. 
\end{proof}
\begin{proposition}\label{fjumps}
 Let $\om\in\Db_T^\star$ and $\{\om_k\}_{k\in\Nb}\subset \Db_T^\star$ be such that $d_T^\star(\om_k,\om)\to 0$ as $k\to\infty$. If $\om$ is simple and, for any $k\in\Nb$, $Z_N^\om(0)=Z_N^{\om_k}(0)$, then
 \begin{equation}\label{convfinitejumps}
(Z_N^{\om_k}(t))_{t\in[0,T]}\xrightarrow[k\to \infty]{(d)}(Z_N^\om(t))_{t\in[0,T]}.
\end{equation}
\end{proposition}
\begin{proof}
The proof consists of two parts. In the first part, we construct a time deformation $\lambda_k\in\Cs_T^\uparrow$ with suitable properties. In the second part, we compare $Z_N^{\om_k}\circ\lambda_k$ and $Z_N^\om$ under an appropriate coupling of the underlying Moran models.
 \smallskip
 
 \textbf{Part 1:}
We assume, without loss of generality, that $d_T^\star(\om_k,\om)>0$, for all $k\in\Nb$. Set $\epsilon_k\coloneqq 2d_T^\star(\om_k,\om)$, so that $d_T^\star(\om_k,\om)< \epsilon_k$. By definition of the metric $d_T^\star$ in \eqref{defdtstar}, there is $\phi_k\in \Cs_T^\uparrow$ such that 
\[\lVert \phi_{k}\rVert_T^0\leq \epsilon_k\quad\textrm{and}\quad\sum_{u\in[0,T]}|\Delta \om(u)-\Delta (\om_k\circ\phi_{k})(u)|\leq \epsilon_k,\]

where $\lVert \cdot \rVert_T^0$ is defined in \eqref{normbij}. Denote by $r_1<\cdots<r_n$ {the consecutive jump times} of $\om$ in $[0,T]$. We assume without loss of generality that $0<r_1\leq r_n<T$. The case where $\om$ jumps at $T$ can be reduced to the previous case, by extending $\omega_k$, $k\in\Nb$, and $\omega$ to $[0,T+\varepsilon]$ as constants in $[T,T+\varepsilon]$. Set $\gamma_k\coloneqq T\,\sqrt{e^{\epsilon_k}-1}$. In the remainder of the proof we assume that $k$ is sufficiently large, so that $\gamma_k\leq \min_{i\in[n]_0}(r_{i+1}-r_{i})/3$, where $r_0\coloneqq 0$ and $r_{n+1}\coloneqq T$. This condition ensures that the intervals $I_i^k\coloneqq [r_i-\gamma_k,r_i+\gamma_k]$, $i\in[n]$, are disjoint and contained in $[0,T]$. Now, define $\lambda_k:[0,T]\to[0,T]$ via
\begin{itemize}
 \item[(i)] For $u\notin \cup_{i=1}^nI_i^k$: $\lambda_k(u)\coloneqq u$.
 \item[(ii)] For $u\in [r_i-\gamma_k,r_i]:$ $\lambda_k(u)\coloneqq \phi_k(r_i)+m_i(u-r_i)$, where $m_i\coloneqq (\phi_k(r_i)-r_i+\gamma_k)/\gamma_k.$
 \item[(iii)] For $u\in (r_i,r_i+\gamma_k]:$ $\lambda_k(u)\coloneqq \phi_k(r_i)+\bar{m}_i(u-r_i)$, where $\bar{m}_i\coloneqq (r_i+\gamma_k-\phi_k(r_i))/\gamma_k.$
\end{itemize}
For $k$ sufficiently large, so that $\epsilon_k < \log 2$, we can infer from $\lVert \phi_{k}\rVert_T^0\leq \epsilon_k$ and from $\gamma_k=T\,\sqrt{e^{\epsilon_k}-1}$ that $m_i$ and $\bar{m}_i$ are positive. It is then straightforward to check that $\lambda_k\in\Cs_T^\uparrow$, 
$\lambda_k(I_i^k)=I_i^k$, $i\in[n]$, and that
\[\sum_{u\in[0,T]}|\Delta \om(u)-\Delta \omb_k(u)|\leq \epsilon_k,\]
where {$\omb_k\coloneqq \om_k\circ\lambda_k$.} Moreover, since $\lVert \phi_{k}\rVert_T^0\leq \epsilon_k$, we infer that $\phi_k(r_i)\in[e^{-\epsilon_k}r_i,e^{\epsilon_k}r_i]$. It follows that, for $k$ sufficiently large, 
\[1-2\sqrt{e^{\epsilon_k}-1} \leq m_i \leq 1+2\sqrt{e^{\epsilon_k}-1},\quad i\in[n],\]
and the same holds for $\bar{m}_i$. Note that we can write $\lambda_k(t) = \int_0^t p_k(u) \dd u$, with $p_k:[0,T] \mapsto \mathbb{R}$ taking only the values $(m_i)_{i\in[n]}$, $(\bar{m}_i)_{i\in[n]}$, and $1$. In particular, we have $|p_k(u)-1|\leq 2\sqrt{e^{\epsilon_k}-1}$ for all $u\in[0,T]$. Thus, for any $s,t\in[0,T]$ with $s\neq t$, the slope $(\lambda_k(s)-\lambda_k(t))/(s-t)$ belongs to $[1-2\sqrt{e^{\epsilon_k}-1}, 1+2\sqrt{e^{\epsilon_k}-1}]$. Therefore, for $k$ sufficiently large, we have 
\[\frac{\lambda_k(s)-\lambda_k(t)}{s-t}, \,\frac{s-t}{\lambda_k(s)-\lambda_k(t)}\leq 1+3\sqrt{e^{\epsilon_k}-1},\quad i\in[n].\]
Hence, using that $\log(1+x)\leq x$ for $x>-1$, {we obtain for $k$ sufficiently large 
\begin{equation}\label{lki}
\lVert \lambda_k\rVert_T^0\leq 3\sqrt{e^{\epsilon_k}-1}.
\end{equation}}
\textbf{Part 2:} 
 For $\delta>0$, we couple in $[0,T]$ a Moran model with parameters $(\sigma_N,\theta_N,\nu_0,\nu_1)$ and environment $\om$ to a Moran model with parameters $(\sigma_N,\theta_N,\nu_0,\nu_1)$ and environment $\om_k$ (both of size $N$) by using: (1) the same initial type configuration, (2) the same basic background, and (3) using in the second population the environmental background of the first one time-changed by $\lambda_k^{-1}$. {Under this coupling and by construction of the function $\lambda_k$,  the Moran model associated to $\om$ and the Moran model associated to $\om_k$ (time-changed by $\lambda_k$) experience the same basic events out of the time-intervals $I_i^k$. Moreover, at the times $r_i$,} the success of simultaneous environmental reproductions is decided according to the same uniform random variables.
\smallskip
  
For $t\in[0,T]$, let $Y_N^{\neq}(t)$ be the number of individuals that have a different type at time $t$ for $\om$ and at time $\lambda_k(t)$ for $\om_k$, and set $Y_N^*(t)\coloneqq \sup_{u\in[0,t]}Y_N^{\neq}(u)$. 
\smallskip

Consider the event $E_{k}\coloneqq \{\textrm{there are no basic events in $\cup_{i\in[n]}I_i^k$}\}$, and note that
\begin{equation}\label{ekcc}
\Pb(E_{k}^c)\leq n\left(1-e^{-2N(1+\sigma_N+\theta_N)\gamma_k}\right).
\end{equation}
 Moreover, on the event $E_{k}$, only the population driven by $\om_k$ can change in $(r_i,r_i+\gamma_k]$, and this can only be due to environmental events. Hence,
\begin{equation}\label{eqc1}
 \Eb[Y_N^*(r_i+\gamma_k)1_{E_k}]\leq \Eb[Y_N^*(r_i)1_{E_{k}}] +N\sum_{u\in(r_i,r_i+\gamma_k]}\Delta\omb_k(u).
\end{equation}
 A similar argument yields
 \begin{equation}\label{eqc2}
 \Eb[Y_N^*(r_{i+1}-)\,1_{E_{k}}]\leq \Eb[Y_N^*(r_{i+1}-\gamma_k)1_{E_{k}}]+ N\sum_{u\in(r_{i+1}-\gamma_k,r_{i+1})}\Delta\omb_k(u).
\end{equation}
 Since in the interval $(r_i+\gamma_k,r_{i+1}-\gamma_k]$ there are no simultaneous jumps of the two environments, we can proceed as in the proof of Proposition \ref{sjumps} to obtain
 \begin{equation}\label{eqc3}
 \Eb\left[Y_N^*(r_{i+1}-\gamma_k)1_{E_{k}}\right]\leq e^{(1+\sigma_N)(r_{i+1}-r_{i})}\left[\Eb[Y_N^*(r_i+\gamma_k)1_{E_{k}}]+N\sum_{u\in (r_i+\gamma_k,r_{i+1}-\gamma_k]}\Delta\omb_k(u)\right].
\end{equation}
 Moreover, at time $r_{i+1}$, there are two possible contributions to take into account: (i) the contribution of selective arrows arising simultaneously in both environments, and (ii) the contribution of selective arrows arising only on the environment with the biggest jump. This leads to
 \begin{equation}\label{eqc4}
 \Eb\left[Y_N^*(r_{i+1})\,1_{E_{k}}\right]\leq \Eb[Y_N^*(r_{i+1}-)\,1_{E_{k}}](1+\Delta \om(r_{i+1})\wedge \Delta \omb_k(r_{i+1}))+ N|\Delta \om(r_{i+1})-\Delta \omb_k(r_{i+1})|.
\end{equation}
Using \eqref{eqc1}, \eqref{eqc2}, \eqref{eqc3} and \eqref{eqc4}, we obtain
\[\Eb[Y_N^*(r_{i+1})\,1_{E_k}]\leq  e^{(1+\sigma_N)(r_{i+1}-r_{i})}\left[\Eb[Y_N^*(r_{i})\,1_{E_k}]+N\sum_{u\in 
(r_i,r_{i+1}]}|\Delta \om(u)-\Delta \omb_k(u)|\right](1+\Delta \om(r_{i+1})).\]
Iterating this inequality, using that $Y_N^*(0)=0$, and adding the contribution of the interval $(r_n+\gamma_k,T]$, we obtain
\begin{equation}\label{ekc}
\Eb\left[Y_N^*(T)\,1_{E_{k}}\right]\leq N\epsilon_k e^{(1+\sigma_N)T+\sum_{u\in(0,T]}\Delta \om(u)}.
\end{equation}
Using \eqref{ekcc}, \eqref{ekc}, \eqref{lki}, and the definition of $d_T^0$ in \eqref{defdt0} we obtain for $k$ sufficiently large
\[\Eb\left[d_T^0(Z_N^\om,Z_N^{\om_k})\right] \leq 2nN\left(1-e^{-2N(1+\sigma_N+\theta_N)\gamma_k}\right)+ {3} \sqrt{e^{\epsilon_k}-1}\vee\left(N\epsilon_k\, e^{(1+\sigma_N)T+\sum_{u\in(0,T]}\Delta \om(u)} \right).\]
{Since $\gamma_k\to 0$ and $\epsilon_k\to 0$ as $k\to\infty$, the result follows.}
\end{proof}
\begin{proof}[Proof of Theorem \ref{thm2.1} (Continuity)]
 If $\om$ has a finite number of jumps, the result follows directly from Proposition \ref{fjumps}. In the general case, note that for any $\delta>0$, we have
 \begin{equation}\label{trii}
  \Bl{\mu_N(\om_k)}{\mu_N(\om)}\leq \Bl{\mu_N(\om_k)}{\mu_N(\om_{k}^{\delta})}+ \Bl{\mu_N(\om_{k}^{\delta})}{\mu_N(\om^\delta)}+\Bl{\mu_N(\om^{\delta})}{\mu_N(\om)},
 \end{equation}
 where $\om^{\delta}$ is as in \eqref{defomdelta} and, similarly, $\om_{k}^{\delta}(t)\coloneqq \sum_{u\in[0,t]: \Delta\om_k(u)\geq \delta}\Delta \om_k(u)$. Recall the definition of $d_T^\star$ in \eqref{defdtstar}. We claim that, for any $\delta\in A_\om\coloneqq \{d>0:\Delta \om(u)\neq d \textrm{ for any $u\in[0,T]$}\}$, we have
{\begin{equation}\tag{Claim 1}\label{claim-cont}
d_T^\star(\om_{k}^{\delta},\om^\delta)\xrightarrow[k\to\infty]{}0.
\end{equation}}
Assume {that \eqref{claim-cont}} is true and let $\delta\in A_\om$. Note that for any $\lambda\in\Cs^\uparrow_T$, we have
\begin{align*}
 \om_{k,\delta}(T)&\coloneqq \sum_{u\in[0,T]: \Delta\om_k(u)< \delta}\Delta \om_k(u)=\om_k(T)-\om_{k}^{\delta}(T)=\om_k(\lambda(T))-\om_{k}^{\delta}(\lambda(T))\\
 &\leq |\om_k(\lambda(T))-\om(T)|+|\om(T)-\om^\delta(T)|+|\om^\delta(T)-\om_{k}^{\delta}(\lambda(T))|\\
 &\leq  d_T^0(\om,\om_k)+ \om_\delta(T)+d_T^0(\om_{k}^{\delta},\om^\delta)\leq  d_T^\star(\om,\om_k)+ \om_\delta(T)+d_T^\star(\om_{k}^{\delta},\om^\delta),
 \end{align*}
where we used the definition of $d_T^0$ in \eqref{defdt0} and then Lemma \ref{zero-star}. Combining this {with \eqref{claim-cont}} and Proposition \ref{sjumps}, we obtain
\begin{equation*}
 \limsup_{k\to\infty}\Bl{\mu_N(\om_k)}{\mu_N(\om_{k}^{\delta})}\leq N\om_\delta(T) e^{(1+\sigma_N)T+\om(T)}.
\end{equation*}
In addition, Proposition \ref{fjumps} together with {\eqref{claim-cont}} implies that
$\limsup_{k\to\infty}\Bl{\mu_N(\om_{k}^{\delta})}{\mu_N(\om^{\delta})}=0$.
Hence, letting $k\to\infty$ in \eqref{trii} and using Proposition \ref{sjumps}, we obtain
\[\limsup_{k\to\infty}\Bl{\mu_N(\om_k)}{\mu_N(\om)}\leq 2N\om_\delta(T) e^{(1+\sigma_N)T+\om(T)}.\]
The previous inequality holds for any $\delta\in A_\om$. It is plain to see that $\inf A_\om=0$. Hence, letting $\delta\to 0$ with $\delta\in A_\om$ in the previous inequality yields the result. 

It remains to prove {\eqref{claim-cont}}. Let $\delta\in A_\om$. Since $d_T^\star(\om_k,\om)$ converges to $0$ as $k\to\infty$, then we see from the definition of $d_T^\star$ in \eqref{defdtstar} that there is a sequence $(\lambda_k)_{k\in\Nb}$ with $\lambda_k\in\Cs_T^\uparrow$ such that
\[\lVert \lambda_k\rVert_T^0\xrightarrow[k\to\infty]{}0\quad\textrm{ and }\quad \epsilon_k\coloneqq  \sum_{u\in[0,T]}|\Delta(\om_k\circ\lambda_k)(u)-\Delta \om(u)| \xrightarrow[k\to\infty]{}0.\]
Set {$\omb_k\coloneqq \om_k\circ\lambda_k$.} Clearly, 
$\Delta\omb_k(u)\leq \epsilon_k+\Delta \om(u)$ and $\Delta\om (u)\leq \epsilon_k+\Delta \omb_k(u)$, $u\in[0,T].$
Therefore,
\begin{align*}
\om_{k}^{\delta}(\lambda_k(t))-\om^\delta(t)&=\sum_{u\in[0,t]: \Delta\omb_k(u)\geq \delta}\Delta \omb_k(u)-\sum_{u\in[0,t]: \Delta\om(u)\geq \delta}\Delta \om(u)\\
&\leq \sum_{u\in[0,t]: \Delta\om(u)\geq \delta-\epsilon_k}\Delta \omb_k(u)-\sum_{u\in[0,t]: \Delta\om(u)\geq \delta}\Delta \om(u)\\
&=\sum_{u\in[0,t]: \Delta\om(u)\geq \delta-\epsilon_k}(\Delta \omb_k(u)-\Delta \om(u))+\sum_{u\in[0,t]: \Delta\om(u)\in[\delta-\epsilon_k,\delta)}\Delta \om(u)\\
&\leq d_T^\star(\om_k,\om)+\sum_{u\in[0,T]: \Delta\om(u)\in[\delta-\epsilon_k,\delta)}\Delta \om(u).
\end{align*}
Similarly, we obtain
\begin{align*}
\om^\delta(t)-\om_{k}^{\delta}(\lambda_k(t))&=\sum_{u\in[0,t]: \Delta\om(u)\geq \delta}\Delta \om(u)-\sum_{u\in[0,t]: \Delta\omb_k(u)\geq \delta}\Delta \omb_k(u)\\
&\leq \sum_{u\in[0,t]: \Delta\om(u)\in[\delta,\delta+\epsilon_k)}\Delta \om(u)+\sum_{u\in[0,t]: \Delta\omb_k(u)\geq \delta}(\Delta \om(u)-\Delta \omb_k(u))\\
&\leq \sum_{u\in[0,T]: \Delta\om(u)\in[\delta,\delta+\epsilon_k)}\Delta \om(u)+d_T^\star(\om_k,\om).
\end{align*}
Thus, using the definition of $d_T^0$ in \eqref{defdt0}, we deduce that
\[d_T^0(\om_k^\delta,\om^\delta)\leq d_T^\star(\om_k,\om)+\sum_{u\in[0,T]: \Delta\om(u)\in(\delta-\epsilon_k,\delta+\epsilon_k)}\Delta \om(u).\]
Since $\delta\in A_\om$, letting $k\to\infty$ in the previous inequality yields $\lim_{k\to\infty }d_T^0(\om_k^\delta,\om^\delta)=0$. Recall that $\om^\delta$ has a finite number of jumps, and hence, {\eqref{claim-cont}} follows using Lemma \ref{zero-star}.
\end{proof}
%%%%%%%%%%%%%%%%%%%%%%%%%%%%%%%%%%%%%%%%%%%%%%%%%%%%%%%%%%%%%%%%%%%%%%%%%%%%%%%%%%%%%%%%%%%%%%%%
%%%%%%%%%%%%%%%%%%%%%%%%%%%%%%%%%%%%%%%%%%%%%%%%%%%%%%%%%%%%%%%%%%%%%%%%%%%%%%%%%%%%%%%%%%%%%%%%
%%%%%%%%%%%%%%%%%%%%%%%%%%%%%%%Section 4%%%%%%%%%%%%%%%%%%%%%%%%%%%%%%%%%%%%%%%%%%%%%%%%%%%%%%%%
%%%%%%%%%%%%%%%%%%%%%%%%%%%%%%%%%%%%%%%%%%%%%%%%%%%%%%%%%%%%%%%%%%%%%%%%%%%%%%%%%%%%%%%%%%%%%%%%
%%%%%%%%%%%%%%%%%%%%%%%%%%%%%%%%%%%%%%%%%%%%%%%%%%%%%%%%%%%%%%%%%%%%%%%%%%%%%%%%%%%%%%%%%%%%%%%%
\subsection{Results related to Section \ref{s23}: the Wright--Fisher process as a large population limit}\label{s32}
We start this section proving that the SDE \eqref{WFSDE} is well-posed.
\begin{proposition}[Existence and uniqueness]\label{eandu}
Let $\sigma,\theta\geq 0$, $\nu_0,\nu_1\in[0,1]$ with $\nu_0+\nu_1=1$. Let $J$ be a pure-jump subordinator with L\'{e}vy measure $\mu$ supported in $(0,1)$ and let $B$ be a standard Brownian motion independent of $J$. Then, for any $x_0\in[0,1]$, there is a pathwise unique strong solution $(X(t))_{t\geq 0}$ to the SDE \eqref{WFSDE} such that $X(0)=x_0$. Moreover, $X(t)\in[0,1]$ for all $t\geq 0$.
\end{proposition}
\begin{remark}
The Wright--Fisher diffusion defined via the SDE \eqref{WFSDE} with $\theta=0$ corresponds to \cite[Eq. 10]{CSW19} with $K_y$, $y\in(0,1)$, being a random variable that takes the value $1$ with probability $1-y$ and the value $2$ with probability $y$.
\end{remark}
\begin{proof}
In order to prove the existence and the pathwise uniqueness of strong solutions of \eqref{WFSDE} we use \cite[Thms. 3.2, 5.1]{LiPu12}. To this purpose, we first extend \eqref{WFSDE} to an SDE on $\Rb$ and we write it in the form of \cite[Eq. 2.1]{LiPu12}. Note that by L\'{e}vy-It\^{o} decomposition, the pure-jump subordinator $J$ can be expressed as
$J(t)=\int_{(0,1)} x\, N(t,\dd x),\quad t\geq 0,$
where $N(\dd s,\dd x)$ is a Poisson random measure with intensity measure $\mu$. Hence, defining the functions $a,b:\Rb\to\Rb$ and $g:\Rb\times(0,1)\to\Rb$ by setting
\[a(x)\coloneqq \sqrt{2x(1-x)},\quad b(x)\coloneqq\sigma x(1-x)+\theta\nu_0(1-x)-\theta\nu_1 x, \quad g(x,u)\coloneqq x(1-x)u,\]
for $x\in[0,1]$ and $u\in(0,1)$; $a(x)\coloneqq 0$, $g(x,u)\coloneqq 0$ for $x\notin[0,1]$; $b(x)\coloneqq \theta\nu_0$ for $x<0$ and $b(x)\coloneqq -\theta\nu_1$ for $x>1$, Eq. \eqref{WFSDE} can be extended to the following SDE on $\Rb$
\begin{equation}\label{ASDE}
 X(t)=X(0)+\int_0^t a(X(s))\dd B(s)+\int_0^t\int_{(0,1)}g(X(s-),u)N(\dd s,\dd u)+\int_0^t b(X(s))\dd s,\quad t\geq 0.
\end{equation}
{Thus,} any solution $(X(t))_{t\geq 0}$ of \eqref{ASDE} such that $X(t)\in[0,1]$ for any $t\geq 0$ is a solution of \eqref{WFSDE} and vice-versa. Note that the functions $a, b,g$ are continuous. Moreover, $b=b_1- b_2$, where 
\[\begin{array}{llll}
 b_1(x)\coloneqq \theta\nu_0+\sigma x\quad\textrm{for $x\in[0,1]$,} & b_1(x)\coloneqq \theta\nu_0\quad\textrm{for}\quad x\leq0,&\textrm{and}\quad b_1(x)\coloneqq \theta\nu_0+\sigma\quad\textrm{for}\quad x\geq1,\\
 b_2(x)\coloneqq \theta x +\sigma x^2\quad\textrm{for $x\in[0,1]$,}& b_2(x)\coloneqq 0\quad\,\,\,\,\,\,\textrm{for}\quad x\leq0,&\textrm{and}\quad b_2(x)\coloneqq \theta+\sigma\quad\,\,\,\,\,\,\textrm{for}\quad x\geq1.
\end{array}\]
In addition, $b_2$ is non-decreasing. Thus, in order to apply \cite[Thms. 3.2, 5.1]{LiPu12}, we only need to verify the sufficient conditions (3.a), (3.b) and (5.a) therein. Condition (3.a) in our case amounts to prove that $x\mapsto b_1(x)+\int_{(0,1)}g(x,u)\mu(\dd u)$ is Lipschitz continuous. In fact, a straightforward calculation shows that
\[|b_1(x)-b_1(y)|+\int_{(0,1)}|g(x,u)-g(y,u)|\mu(\dd u)\leq \left(\sigma+\int_{(0,1)}u\mu(\dd u)\right)|x-y|,\quad x,y\in\Rb,\]
 and hence (3.a) follows. Condition (3.b) amounts to prove that $x\mapsto a(x)$ is $1/2$-H\"older, which was already shown in the proof of \cite[Lemma 3.6]{GS18}.
% \begin{equation}\label{lipsigma}
%  |a(x)-a(y)|^2\leq 4|x-y|,\quad x,y\in[0,1].
% \end{equation}
% One can easily check that {Claim \eqref{lipsigma}} holds whenever $x\notin (0,1)$ or $y\notin (0,1)$. Now, assume that $x,y\in(0,1)$. If $x+y<1$, we have
% \begin{align*}
%  |a(x)-a(y)|&=\frac{2|x-y|(1-(x+y))}{\sqrt{2x(1-x)}+\sqrt{2y(1-y)}}\leq \frac{\sqrt{2}|x-y|(1-(x+y))}{\sqrt{x(1-x)+y(1-y)}} \\
%  &=\frac{\sqrt{2}|x-y|(1-(x+y))}{\sqrt{(x+y)(1-(x+y))+2xy}}\leq \frac{\sqrt{2}|x-y|(1-(x+y))}{\sqrt{(x+y)(1-(x+y))}}\\
%  &\leq \sqrt{2|x-y|(1-(x+y))} {\leq \sqrt{2|x-y|}.}
% \end{align*} 
% Since $a(z)=a(1-z)$ for all $z\in\Rb$, the same inequality {$|a(x)-a(y)|\leq \sqrt{2|x-y|}$} holds for $x,y\in(0,1)$ such that $x+y> 1$, and the case $x+y=1$ is trivial. Hence, {Claim \eqref{lipsigma}} is true and condition (3.b) follows. 
Therefore, \cite[Thm. 3.2]{LiPu12} yields the pathwise uniqueness for \eqref{ASDE}. Condition (5.a) follows from the fact that the functions $a, b$, $x\mapsto \int_{(0,1)} g(x,u)^2\mu(\dd u)$ and $x\mapsto \int_{(0,1)} g(x,u)\mu(\dd u)$ are bounded on $\Rb$. Hence, \cite[Thm. 5.1]{LiPu12} ensures the existence of a strong solution to \eqref{ASDE}. It remains to show that any solution of \eqref{ASDE} with $X(0)\in[0,1]$ is such that $X(t)\in[0,1]$ for any $t\in[0,1]$. Sufficient conditions implying such a result are provided in \cite[Prop. 2.1]{FuLi10}. The conditions on the diffusion and drift coefficients are satisfied, namely, $a$ is $0$ outside $[0,1]$ and $b(x)$ is positive for $x\leq 0$ and negative for $x\geq 1$. However, the condition on the jump coefficient, $x+g(x,u)\in[0,1]$ for every $x\in\Rb$, is not fulfilled. Nevertheless, the proof of \cite[Prop. 2.1]{FuLi10} works without modifications under the alternative condition $x+g(x,u)\in[0,1]$ for $x\in[0,1]$ and $g(x,u)=0$ for $x\notin[0,1]$, which is in turn satisfied. This ends the proof.
\end{proof}

\begin{lemma}\label{core} The solution of the SDE \eqref{WFSDE} is a Feller process with generator $\As$ satisfying for all $f\in \mathcal{C}^2([0,1],\Rb)$
 \begin{align*}
\As f(x)&=x(1-x)f''(x)+ (\sigma x(1-x)+\theta\nu_0(1-x)-\theta\nu_1 x) f{'}(x)+\int_{(0,1)}\left(f\left(x+ x(1-x)u\right)-f(x)\right)\mu(\dd u).
\end{align*}
Moreover, $\mathcal{C}^\infty([0,1],\Rb)$ is an operator core for $A$.
 \end{lemma}
 \begin{proof}
Since pathwise uniqueness implies weak uniqueness (see \citep[Thm.~1]{BLP15}), we infer from \citep[Cor.~2.16]{Ku11} that the martingale problem associated to $A$ in $\Cs^{2}([0,1])$ is well-posed. Moreover, an inspection of the proof shows that this is also true in $\Cs^{\infty}([0,1])$. Using \cite[Prop. 2.2]{vanC92}, we deduce that $X$ is Feller. The fact that $\Cs^\infty([0,1])$ is a core follows then from \cite[Thm.~2.5]{vanC92}.
 \end{proof}
Now, we proceed to prove the first part of the main result of Section \ref{s23}, i.e. the annealed convergence of a sequence Moran models towards the solution of the SDE \eqref{WFSDE}.
\begin{proof}[Proof of Theorem \ref{thm2.2}-(1) (Annealed convergence)]
Let $\As_N^*$  and $\As$ be the infinitesimal generators of the processes {$(X_N(t))_{t\geq 0}$} and $(X(t))_{t\geq 0}$, respectively. {Note that $(X_N(t))_{t\geq 0}$ has state space
{\begin{equation} \label{defen}
E_N\coloneqq \{k/N:k\in[N]_0\}. 
\end{equation}}} We will prove that, for all $f\in\Cs^\infty([0,1],\Rb)$, 
{\begin{equation}\tag{Claim 2}\label{claimcvgenerator}
\sup\limits_{x\in E_N}|\As_N^* f(x)-\As f(x)|\xrightarrow[N\to\infty]{} 0.
\end{equation}
Provided \eqref{claimcvgenerator}} is true, since $X$ is Feller and $\Cs^\infty([0,1],\Rb)$ is an operator core for $A$ (see Lemma \ref{core}), the result follows applying \cite[Theorem 19.25]{Kall02}. Thus, it remains to prove {\eqref{claimcvgenerator}}. To this end, we decompose the generator $\As$ as $\As^{1}+\As^{2}+\As^3+\As^4$, where 
\begin{align*}
\As^1f(x)\coloneqq x(1-x)f{''}(x),\quad\qquad\qquad\qquad\qquad\qquad &\As^2f(x)\coloneqq (\sigma x(1-x)+\theta\nu_0(1-x)-\theta\nu_1 x) f{'}(x),\\
\As^3f(x)\coloneqq \int_{(0,\varepsilon_N)}\!\!\left(f\left(x+ x(1-x)u\right)-f(x)\right)\mu(\dd u),\quad &
\As^4f(x)\coloneqq \int_{(\varepsilon_N,1)}\!\!\left(f\left(x+ x(1-x)u\right)-f(x)\right)\mu(\dd u). 
\end{align*}
We will choose $\varepsilon_N>0$ later in an appropriate way. We also write $\As_N^*=\As_N^{1}+\As_N^{2}+\As_N^3+\As_N^4$, where 
\begin{align*}
\As_N^1f(x)&\coloneqq N^2x(1-x)\left[\Delta_{\frac1N}f(x)+\Delta_{-\frac1N}f(x)\right],\\
\As_N^2f(x)&\coloneqq N^2(\sigma_N x(1-x)+\theta_N\nu_0(1-x))\left[\Delta_{\frac1N}f(x)\right]+N^2\theta_N\nu_1x\left[\Delta_{-\frac1N}f(x)\right],\\
\As_N^3f(x)&\coloneqq \int_{(0,\varepsilon_N^{})}\!\!\!\!\!\!\!\Eb\left[f\left(x+\xi_N(x,u)\right)\right]-f(x)\mu(\dd u),\quad
\As_N^4f(x)\coloneqq \int_{(\varepsilon_N^{},1)}\!\!\!\!\!\!\!\Eb\left[f\left(x+\xi_N(x,u)\right)\right]-f(x)\mu(\dd u),
\end{align*}
where $\Delta_hf(x)\defeq f(x+h)-f(x)$, and $\xi_N(x,u)\coloneqq \Hs(N,N(1-x),B_{Nx}(u))/N$, with
$\Hs(N,N(1-x),k)\sim\hypdist{N}{N(1-x)}{k}$, and $B_{Nx}(u)\sim\bindist{Nx}{u}$ being independent. 
\smallskip

Let $f\in\Cs^\infty([0,1],\Rb)$. Note that
\begin{equation}\label{t0}
 \sup_{x\in E_N} |\As_N^* f(x)-\As f(x)|\leq \sum_{i=1}^4 \sup_{x\in E_N} |\As_N^i f(x)-\As^i f(x)|.
\end{equation}
Using Taylor expansions of order three around $x$ for $f(x+1/N)$ and $f(x-1/N)$, we get 
\begin{equation}\label{t1}
 \sup\limits_{x\in E_N}|\As_N^1 f(x)-\As^1 f(x)|\leq \frac{\lVert f'''\rVert_\infty}{3N}. 
\end{equation}
Similarly, using the triangular inequality and appropriate Taylor expansions of order two yields
\begin{equation}\label{t2}
 \sup\limits_{x\in E_N}|\As_N^2 f(x)-\As^2 f(x)|\leq {\frac{(N\sigma_N+N\theta_N) \lVert f''\rVert_\infty}{2N}}+ (|\sigma-N\sigma_N|+|\theta-N\theta_N|)\lVert f'\rVert_{\infty}.
\end{equation}
{Since $N\sigma_N\to\sigma$ and $N\theta_N\to\theta$, the right-hand side in \eqref{t2} converges to $0$ as $N\to\infty$}. In addition, since $\Eb[\xi_N(x,u)]=x(1-x)u$, we have
\[|\As^3_Nf(x)|\leq \lVert f'\rVert_\infty\int_{(0,\varepsilon_N)} u\mu(\dd u),\quad x\in[0,1],\]
and hence,
\begin{equation}\label{t3}
 \sup\limits_{x\in E_N}|\As_N^3 f(x)-\As^3 f(x)|\leq 2 ||f'||_\infty\int_{(0,\varepsilon_N^{})} u\mu(\dd u).
\end{equation}
For the last term, we first note that
\begin{align*}
 \left|\Eb\left[f\left(x+\xi_N(x,u)\right)-f(x+x(1-x)u)\right]\right|&
 \leq \lVert f'\rVert_{\infty}\Eb\left[\left|\xi_N(x,u)-x(1-x)u\right|\right]\\
 &\leq \lVert f'\rVert_{\infty}\sqrt{\Eb\left[\left(\xi_N(x,u)-x(1-x)u\right)^2\right]}\leq\lVert f'\rVert_{\infty} \sqrt{\frac{u}{N}}.
\end{align*}
In the last inequality we used that
\begin{equation}\label{hybimix}
 \Eb\left[\left(\xi_N(x,u)-x(1-x)u\right)^2\right]=\frac{x(1-x)^2u(1-u)}{N}+\frac{Nx^2(1-x)u^2}{N^2(N-1)}\leq \frac{u}{N},
\end{equation}
which is obtained from standard properties of the hypergeometric and binomial distributions. Hence, 
\begin{equation}\label{t4}
 \sup\limits_{x\in E_N}|\As_N^4 f(x)-\As^4 f(x)|\leq ||f'||_\infty\int_{(\varepsilon_N^{},1)} \sqrt{\frac{u}{N}}\,\mu(du)\leq \frac{||f'||_\infty}{\sqrt{N\varepsilon_N}}\int_{(\varepsilon_N^{},1)} u\,\mu(\dd u) .
\end{equation}
Now, choose $\varepsilon_N\coloneqq 1/\sqrt{N}$. Since $\int_{(0,1)} u\mu(\dd u)<\infty$, {\eqref{claimcvgenerator}} follows by plugging \eqref{t1}, \eqref{t2}, \eqref{t3} and \eqref{t4} in \eqref{t0} and letting $N\to\infty$. 
\end{proof}
Before embarking on the proof of the second part of Theorem \ref{thm2.2}, we prove the following estimates for the Moran model with null environment.
\begin{lemma}\label{nullest}
 For any $x\in E_N$ (see \eqref{defen} for its definition) and $t\geq 0$, we have
 \[\Eb\left[\left(X_N^{\bz}(t)-x\right)^2\mid X_N^{\bz}(0)=x\right]\leq \left(\frac12 + N(\sigma_N+3\theta_N)\right)t,\]
 and
 \[-N\theta_N\nu_1 t\leq \Eb\left[X_N^\bz(t)-x\mid X_N^{\bz}(0)=x\right]\leq N(\sigma_N+\theta_N\nu_0) t.\]
\end{lemma}
\begin{proof}
 Fix $x\in E_N$ and consider the functions $f_x,g_x:E_N\to[0,1]$ defined via $f_x(z)\coloneqq (z-x)^2$ and $g_x(z)\coloneqq z-x$, $z\in E_N$. The process $X_N^\bz$ is a Markov chain with generator $\As_N^{\star,0}\coloneqq \As_N^1+\As_N^2$, where $\As_N^1$ and $\As_N^2$ are defined in the proof of Theorem \ref{thm2.2}. Moreover, for every $z\in E_N$, we have
 \begin{align*}
  \As_N^{\star,0} f_x(z)&{=2z(1-z)+N\left[(\sigma_N z+\theta_N \nu_0)(1-z)\left(2(z-x)+\frac1N\right)+\theta_N\nu_1 z\left(2(x-z)+\frac1N\right)\right]}\\
  &{\leq \frac12+N\left[3\left(\frac{\sigma_N}{4}+\theta_N \nu_0\right)+3\theta_N\nu_1\right]}\leq \frac12 + N(\sigma_N+3\theta_N),
 \end{align*}
  and 
 \begin{align*}
  \As_N^{\star,0} g_x(z)&=N\left[(\sigma_N z+\theta_N \nu_0)(1-z)-\theta_N\nu_1 z\right]\in[-N\theta_N\nu_1,N(\sigma_N+\theta_N\nu_0)].
 \end{align*}
 Hence, Dynkin's formula applied to $X_N^0$ with $f_x$ leads to
 \[\Eb\left[\left(X_N^\bz(t)-x\right)^2\mid X_N^{\bz}(0)=x\right]=\int_0^t \Eb\left[\As_N^{\star,0} f_x(X_N^\bz(s))\mid X_N^{\bz}(0)=x\right]\dd s\leq \left(\frac12 + N(\sigma_N+3\theta_N)\right)t.\]
 Similarly, applying Dynkin's formula to $X_N^0$ with $g_x$, we obtain
 \[\Eb\left[X_N^\bz(t)-x\mid X_N^{\bz}(0)=x\right]=\int_0^t \Eb\left[\As_N^{\star,0} g_x(X_N^\bz(s))\mid X_N^{\bz}(0)=x\right]\dd s\in[-N\theta_N\nu_1 t,N(\sigma_N+\theta_N\nu_0)t],\]
which ends the proof.
\end{proof}

\begin{proposition}[Quenched tightness]\label{q-tight}
Assume that $N \sigma_N\rightarrow \sigma$ and $N\theta_N\rightarrow \theta$ as $N\to\infty$. For any $\om\in\Db^\star$, the sequence {$(X_N^\om)_{N\geq 1}$} is tight.
\end{proposition}
\begin{proof}
Let $(\Fs_s^N)_{s\geq 0}$ denote the natural filtration associated to the process $X_N^\om$. To prove the tightness of the sequence {$(X_N^\om)_{N\geq 1}$}, we use \cite[Thm. 1]{BKS16}. For this we need to show that the following conditions hold:
\begin{itemize}
\item[A1)] For each $T,\varepsilon>0$, there exists a compact set $K\subset\Rb$ such that
\[\liminf_{N\to\infty} \Pb\left(X_N^\om(t)\in K,\, \forall t\leq T\right)\geq 1-\varepsilon.\] 
\item[A2)] There exist $\alpha>0$ and non-decreasing, c\`adl\`ag processes $F_N$, $F$ such that $F_N$ is $\Fs_0$-measurable,
$F_N\xRightarrow[]{(d)}F$ and for any $N\geq 1$ and every $0\leq s\leq t$
\[\Eb\left[1\wedge |X_N^\om(t)-X_N^\om(s)|^{\alpha} \right]\leq F_N(t)-F_N(s).\]
\end{itemize}
Since for all $t\geq 0$ and $N\geq 1$, $X_N^\om(t)\in E_N \subset [0,1]$ (see \eqref{defen} for the definition of $E_N$), condition A1) is satisfied. Now, we claim that there are constants $c,C>0$, independent of $N$, such that
{\begin{equation}\tag{Claim 3}\label{claiml2cont}
\Eb\left[(X_N^\om(t)-X_N^\om(s))^2\mid \Fs_s^N\right]\leq c\sum_{u\in[s,t]}\Delta \om(u)+ C (t-s),\quad\textrm{for all $0\leq s\leq t$}.
\end{equation}}
 If {\eqref{claiml2cont}} is true, then condition A2) is satisfied with $\alpha=2$ and $F_N(t)=F(t)=c\sum_{u\in[0,t]}\Delta \om(u)+Ct$, and the result follows from \cite[Thm. 1]{BKS16}. The rest of the proof is devoted to prove {\eqref{claiml2cont}}. 
\smallskip

For $x\in E_N$ and $t\geq 0$, we set $\psi_x(\om,t)\coloneqq \Eb_x[(X_N^\om(t)-x)^2].$
For $s\geq 0$, we set $\om_s(\cdot)\coloneqq \om(s+\cdot)$. From the definition of $X_N^\om$, it follows that, for any $0\leq s<t$,
\begin{equation}\label{InMP}
 \Eb\left[(X_N^\om(t)-X_N^\om(s))^2\mid \Fs_s^N\right]=\psi_{X_N^\om(s)}(\om_s,t-s).
\end{equation}
Let $0\leq s<t$. We split the proof of {\eqref{claiml2cont}} in three cases.
\smallskip

\textbf{Case 1: $\om$ has no jumps in $(s,t]$.} In particular, $\om_s$ has no jumps in $(0,t-s]$. Hence, restricted to $[0,t-s]$, $X_N^{\om_s}$ has the same distribution as $X_N^\bz$. Using Lemma \ref{nullest} with $x=X_N^\om(s)$ and plugging the result in \eqref{InMP}, we infer that {\eqref{claiml2cont}} holds for any $c\geq 1$ and $C\geq C_1\coloneqq 1/2+\sup_{N\in\Nb}(N(\sigma_N+3\theta_N))$.
\smallskip

\textbf{Case 2: $\om$ has $n$ jumps in $(s,t]$.} Let $t_1,\ldots,t_n\in(s,t]$ be the jump times of $\om$ in $(s,t]$ in increasing order. We set $t_0\coloneqq s$ and $t_{n+1}=t$. For any $i\in[n+1]$ and any $r\in(t_{i-1},t_{i})$, $\om$ has no jumps in $(t_{i-1},r]$. In particular, $(t_{i-1},r]$ falls into Case 1. Using {\eqref{claiml2cont}} in $(t_{i-1},r]$ and letting $r\to t_i$, we obtain 
\[\Eb\left[(X_N^\om(t_i -)-X_N^\om(t_{i-1}))^2\mid \Fs_{t_{i-1}}^N\right]\leq C_1(t_i-t_{i-1}).\]
Moreover,  
\[\Eb\left[(X_N^\om(t_i)-X_N^\om(t_{i}-))^2\mid \Fs_{t_{i}-}^N\right]\leq \Eb\left[\left(\frac{B_{N}(\Delta \om(t_i))}{N}\right)^2\right]\leq \Delta \om(t_i),\]
{where $B_{N}(\Delta \om(t_i)) \sim \bindist{N}{\Delta \om(t_i)}$.} Using the two previous inequalities and the tower property of the conditional expectation, we get
\begin{equation}\label{oneint}
 \Eb\left[(X_N^\om(t_i)-X_N^\om(t_{i-1}))^2\mid \Fs_{s}^N\right]\leq 2C_1(t_i-t_{i-1})+2\Delta\om(t_i).
\end{equation}
Now, note that
\begin{align*}
 (X_N^\om(t)-X_N^\om(s))^2&=\left(\sum_{i=0}^n( X_N^\om(t_{i+1})-X_N^\om(t_i))\right)^2\\
 &=\sum_{i=0}^n( X_N^\om(t_{i+1})-X_N^\om(t_i))^2+2\sum_{i=0}^n( X_N^\om(t_{i+1})-X_N^\om(t_i))(X_N^\om(t_i)-X_N^\om(s)).
\end{align*}
Using Eq. \eqref{oneint}, we see that
\[\Eb\left[\sum_{i=0}^n(X_N^\om(t_{i+1})-X_N^\om(t_{i}))^2\mid \Fs_{s}^N\right]\leq 2C_1(t-s)+2\sum_{i=1}^n\Delta\om(t_i).\]
Moreover, we have 
\begin{align*}
 \Eb\left[(X_N^\om(t_{i+1})-X_N^\om(t_{i}))(X_N^\om(t_{i})-X_N^\om(s))\mid \Fs_{t_i}^N\right]=\phi_{X_N^\om(s),X_N^\om(t_i)}(\om_{t_i},t_{i+1}-t_{i}),
\end{align*}
where {for $x,y\in E_N$ and $t\geq 0$ we set} $\phi_{x,y}(\om,t)\coloneqq (y-x){\Eb[X_N^\om(t)-y \mid X_N^\om(0)=y]}$. Since, for any $r\in(t_i,t_{i+1})$, {$\om_{t_i}$ has no jumps in $(0,r-t_i]$,} we can use Lemma \ref{nullest} to infer that {for any $x,y\in E_N$,} 
\[\phi_{x,y}(\om_{t_i}, r-t_i)\leq N((\sigma_N+\theta_N\nu_0)\vee\theta_N\nu_1)(r-t_i).\]
Note that $(y-x)\Eb_y[X_N^{\om_{t_i}}(t_{i+1}-t_i)-X_N^{\om_{t_i}}((t_{i+1}-t_i)-)]\leq \Delta \om(t_{i+1})$. Hence, letting $r\to t_{i+1}$, we get
\[\phi_{x,y}(\om_{t_i}, t_{i+1}-t_i)\leq N((\sigma_N+\theta_N\nu_0)\vee\theta_N\nu_1)(t_{i+1}-t_i)+\Delta\om(t_{i+1}),\]
{for any $x,y\in E_N$, hence in particular for $x=X_N^\om(s)$ and $y=X_N^\om(t_i)$.} Altogether, we obtain
\[\Eb\left[(X_N^\om(t)-X_N^\om(s))^2\mid \Fs_{s}^N\right]\leq C_2(t-s)+3\sum_{i=1}^n\Delta\om(t_i),\]
where $C_2\coloneqq 2C_1+\sup_{N\in\Nb}N((\sigma_N+\theta_N\nu_0)\vee\theta_N\nu_1)$. Hence, {\eqref{claiml2cont}} holds for any $C\geq C_2$ and $c\geq 3$.
\smallskip

\textbf{Case 3: $\om$ has infinitely many jumps in $(s,t]$}. For any $\delta$, we consider $\om^\delta$ as in \eqref{defomdelta} and we couple the processes $X_N^\om$ and $X_N^{\om^\delta}$ as in the proof of Proposition \ref{sjumps}. Note that $\om^\delta$ has only a finite number of jumps in any compact interval, thus $\om^\delta$ falls into case 2. Moreover, we have
\begin{align*}
 \psi_x(\om,t)&\leq 2\Eb_x[(X_N^\om(t)-X_N^{\om^\delta}(t))^2]+2\Eb_x[(X_N^{\om^\delta}(t)-x)^2]\\
 &\leq 2\Eb_x[|X_N^\om(t)-X_N^{\om^\delta}(t)|]+2\Eb_x[(X_N^{\om^\delta}(t)-x)^2]\\
 &\leq 2e^{N\sigma_N t+\om(t)} \sum_{u\in[0,t]:\Delta \om(u)<\delta}\Delta \om(u)+2\Eb_x[(X_N^{\om^\delta}(t)-x)^2],
\end{align*}
where in the last inequality we use Proposition \ref{sjumps}. Now, using {\eqref{claiml2cont}} for $X_N^{\om^\delta}$  and the previous inequality, we obtain
\[\Eb\left[(X_N^\om(t)-X_N^\om(s))^2\mid \Fs_{s}^N\right]\leq e^{N\sigma_N (t-s)+\om(t-s)}\sum_{u\in[s,t]:\Delta \om(u)<\delta}\Delta \om(u)+2C_2(t-s)+6\sum_{u\in[s,t]}\Delta\om(u).\]
We let $\delta\to 0$ and conclude that {\eqref{claiml2cont}} holds for any $C\geq 2C_2$ and $c\geq 6$. 
\end{proof}
Now, we proceed to prove the quenched convergence of the sequence of Moran models to the Wright--Fisher diffusion, under the assumption that the environment is simple.
\begin{proof}[Proof of Theorem \ref{thm2.2}-(2) (Quenched convergence)]
Let $B\coloneqq (B(t))_{t\geq 0}$ be a standard Brownian motion. We denote by $X^\bz$ the unique strong solution of \eqref{WFSDE} associated to $B$ and the null environment. Theorem \ref{thm2.2}-(1) implies in particular that $X_N^\bz$ converges to $X^\bz$ as $N\to\infty$.
\smallskip

Now, assume that $\om\neq\bz$ is simple. We denote by $T_\om$ the {(discrete but possibly infinite)} set of jump times of $\om$ in $(0,\infty)$. Moreover, for $0<i< |T_\om|+1$, we denote by $t_i\coloneqq t_i(\om)\in T_\om$ the time of the $i$-th jump of $\om$. We set $t_0 \coloneqq 0$ and $t_{|T_\om|+1} \coloneqq \infty$.
Therefore, we need to prove that
\[(X_N^\om(t))_{t\geq 0}\xrightarrow[N\to\infty]{(d)} (X^\om(t))_{t\geq 0},\]
where {we recall that} the process $X^\om$ {starting at $x_0$} is defined as follows. 
\begin{itemize}
 \item[(i)] {$X^\om(0)=x_0$}.
\item[(ii)] For $i\in\Nb$ with $i\leq |T_\om|+1$, the restriction of $X^{\om}$ to the interval $(t_{i-1},t_i)$  is given by a version of $X^\bz$ started at $X^\om(t_{i-1})$.
\item[(iii)] For $0<i< |T_\om|+1$, conditionally on $X^\om(t_i-)$, \[X^\om(t_i)=X^\om(t_i-)+X^\om(t_i-)(1-X^\om(t_i-))\Delta\om(t_i).\]
 \end{itemize}
Since the sequence $(X_N^\om)_{N\in\Nb}$ is tight {(see Proposition \ref{q-tight}),} it is enough to prove the convergence at the level of the finite dimensional distributions. More precisely, we will prove by induction on $i\in\Nb$ with $i\leq |T_\om|+1$ that for any finite set $I\subset[0,t_i)$, we have \[((X_N^\om(t))_{t\in I},X_N^\om(t_i-))\xrightarrow[N\to\infty]{(d)} ((X^\om(t))_{t\in I},X^\om(t_i-)).\]
For $i=|T_\om|+1<\infty$ we remove the components $X_N^\om(t_i-)$ and $X^\om(t_i-)$ since they do not make sense. Since $X_N^\om(t_1-)=X_N^\bz(t_1)$ and $X^\om(t_1-)=X^\bz(t_1)$ almost surely, the result for $i=1$ follows from Theorem \ref{thm2.2}-(1). Now, assume that the result is true for some $i<|T_\om|+1$ and let $I\subset(0,t_{i+1})$. Without loss of generality we assume that $I=\{s_1,\ldots,s_k, t_i,r_1,\ldots, r_m\}$, with $s_1<\cdots<s_k<t_i<r_1<\cdots<r_m$. We also assume that $i<|T_\om|$, the other case, i.e. $i=|T_\om|<\infty$, follows using an analogous argument.
\smallskip

Let $F:[0,1]^{k+1}\to\Rb$ be a Lipschitz function with $\lVert F\rVert_{\textrm{BL}}\leq 1$ {(see \eqref{defnormbl} for the definition of $\lVert \cdot\rVert_{\textrm{BL}}$)}. Note that
\begin{align*}
 \Eb\left[F((X_N^\om(s_j))_{j=1}^k,X_N^\om(t_i))\right]&=\Eb\left[F((X_N^\om(s_j))_{j=1}^k,X_N^\om(t_i-)+\xi_N(X_N^\om(t_i-),\Delta\om(t_i)))\right],
\end{align*}
where for $x\in E_N$ (see \eqref{defen} for the definition of $E_N$) and $u\in(0,1)$, $\xi_N(x,u)\coloneqq \Hs(N,N(1-x),B_{Nx}(u))/N$ with $\Hs(N,N(1-x),k)\sim\hypdist{N}{N(1-x)}{k}$, $k\in[N]_0$, and $B_{Nx}(u)\sim\bindist{Nx}{u}$ being independent between them and independent of $X_N^\om$. Now, set
\begin{align*}
 D_N&\coloneqq \Eb\left[F((X_N^\om(s_j))_{j=1}^k,X_N^\om(t_i-)+\xi_N(X_N^\om(t_i-),\Delta\om(t_i)))\right]\\
 &\,\,\,-\Eb\left[F((X_N^\om(s_j))_{j=1}^k,X_N^\om(t_i-)+X_N^\om(t_i-)(1-X_N^\om(t_i-))\Delta\om(t_i))\right].
\end{align*}
Using that $\lVert F\rVert_{\textrm{BL}}\leq 1$ and \eqref{hybimix}, we see that
$|D_N| \leq\sqrt{\Delta\om(t_i)/N}\to 0$ as $N\to\infty$.
Therefore, the induction hypothesis yields
\begin{align*}
 \Eb\left[F((X_N^\om(s_j))_{j=1}^k,X_N^\om(t_i))\right]&=D_N+\Eb\left[F((X_N^\om(s_j))_{j=1}^k,X_N^\om(t_i-)+X_N^\om(t_i-)(1-X_N^\om(t_i-))\Delta\om(t_i))\right]\\
 &\xrightarrow[N\to\infty]{}\,\Eb\left[F((X^\om(s_j))_{j=1}^k,X^\om(t_i-)+X^\om(t_i-)(1-X^\om(t_i-))\Delta\om(t_i))\right]\\
 &=\Eb\left[F((X^\om(s_j))_{j=1}^k,X^\om(t_i))\right].
\end{align*}
Therefore,
\begin{equation}\label{uptoti}
((X_N^\om(s_j))_{j=1}^k,X_N^\om(t_i))\xrightarrow[N\to\infty]{(d)} ((X^\om(s_j))_{j=1}^k,X^\om(t_i)).
\end{equation}
Let $G:[0,1]^{k+m+2}\to\Rb$ be a Lipschitz function with $\lVert G\rVert_{\textrm{BL}}\leq 1$. For $x\in E_N$, define
\[H_N(z,x)\coloneqq \Eb_x[G(z,x,(X_N^\bz(r_j-t_i))_{j=1}^m,X_N^\bz(t_{i+1}-t_i))], \ \forall z \in\Rb^k.\]
Note that
\begin{equation}\label{mpr1}
 \Eb[G((X_N^\om(s_j))_{j=1}^k,X_N^\om(t_i),(X_N^\om(r_j))_{j=1}^m,X_N^\om(t_{i+1}-))]=\Eb\left[H_N((X_N^\om(s_j))_{j=1}^k,X_N^\om(t_i))\right].
\end{equation}
Similarly, for $x\in[0,1]$, define
\[H(z,x)\coloneqq \Eb_x[G(z,x,(X^\bz(r_j-t_i))_{j=1}^m,X^\bz(t_{i+1}-t_i))],\quad z\in\Rb^k,\]
and note that
\begin{equation}\label{mpr2}
 \Eb[G((X^\om(s_j))_{j=1}^k,X^\om(t_i),(X^\om(r_j))_{j=1}^m,X^\om(t_{i+1}-))]=\Eb\left[H((X^\om(s_j))_{j=1}^k,X^\om(t_i))\right].
\end{equation}
Using \eqref{uptoti} and the Skorokhod representation theorem, we can assume that $((X_N^\om(s_j))_{j=1}^k,X_N^\om(t_i))_{N\geq1}$ and $((X^\om(s_j))_{j=1}^k,X^\om(t_i))$ {are defined on} the same probability space and that the convergence holds almost surely. In particular, we can write
\begin{equation}\label{er1r2}
 |\Eb\left[H_N((X_N^\om(s_j))_{j=1}^k,X_N^\om(t_i))\right]-\Eb\left[H((X^\om(s_j))_{j=1}^k,X^\om(t_i))\right]|\leq R_N^1+R_N^2,
\end{equation}
where
\begin{align*}
R_N^1&\coloneqq |\Eb\left[H_N((X_N^\om(s_j))_{j=1}^k,X_N^\om(t_i))\right]-\Eb\left[H_N((X^\om(s_j))_{j=1}^k,X_N^\om(t_i))\right]|,\\
R_N^2&\coloneqq |\Eb\left[H_N((X^\om(s_j))_{j=1}^k,X_N^\om(t_i))\right]-\Eb\left[H((X^\om(s_j))_{j=1}^k,X^\om(t_i))\right]|.
\end{align*}
Using that $\lVert G\rVert_{\textrm{BL}}\leq 1$, we obtain
\begin{equation}\label{er1}
 R_N^1\leq \sum_{j=1}^k\Eb[|X_N^\om(s_j)-X^\om(s_j)|]\xrightarrow[N\to\infty]{}0.
\end{equation}
Moreover, since $X_N^\om(t_i)$ converges to $X^\om(t_i)$ almost surely, we conclude using Theorem \ref{thm2.2} that, for any $z\in[0,1]^k$, $H_N(z,X_N^\om(t_i))$ converges to $H(z,X^\om(t_i))$ almost surely. Therefore, using dominated convergence theorem, we conclude that 
\begin{equation}\label{er2}
 R_N^2\xrightarrow[N\to\infty]{}0.
\end{equation}
Plugging \eqref{er1} and \eqref{er2} in \eqref{er1r2} and using \eqref{mpr1} and \eqref{mpr2} completes the proof. 
\end{proof}

%%%%%%%%%%%%%%%%%%%%%%%%%%%%%%%%%%%%%%%%%%%%%%%%%%%%%%%%%%%%%%%%%%%%%%%%%%%%%%%%%%%%%%%%%%%%%%%%
%%%%%%%%%%%%%%%%%%%%%%%%%%%%%%%%%%%%%%%%%%%%%%%%%%%%%%%%%%%%%%%%%%%%%%%%%%%%%%%%%%%%%%%%%%%%%%%%
%%%%%%%%%%%%%%%%%%%%%%%%%%%%%%%Section 5%%%%%%%%%%%%%%%%%%%%%%%%%%%%%%%%%%%%%%%%%%%%%%%%%%%%%%%%
%%%%%%%%%%%%%%%%%%%%%%%%%%%%%%%%%%%%%%%%%%%%%%%%%%%%%%%%%%%%%%%%%%%%%%%%%%%%%%%%%%%%%%%%%%%%%%%%
%%%%%%%%%%%%%%%%%%%%%%%%%%%%%%%%%%%%%%%%%%%%%%%%%%%%%%%%%%%%%%%%%%%%%%%%%%%%%%%%%%%%%%%%%%%%%%%%
\section{The ASG and its relatives}\label{S4}
In this section we formalize the definition of the quenched ASG and we provide definitions for the k-ASG and the pLD-ASG both in the annealed and in the quenched setting.
\subsection{Results related to Section \ref{s24}: the quenched ASG}\label{s41}
The aim of this section is to prove the following result.
\begin{proposition}[Existence of the quenched ASG]\label{eqasg}
Let $\om\in\Db^\star$. For any $n \in\Nb$ and $T > 0$, there exists a branching-coalescing particle system $(\Gs_T^\om(\beta))_{\beta\geq 0}$ starting at $\beta=0-$ with $n$ lines, that almost surely {consists of} finitely many lines at each time $\beta \in [0,T]$, and that satisfies the requirements {(i), (ii), (iii'), (iv) and (v)} of Definition \ref{defannealdasg}. 
\end{proposition}
\begin{proof}
We will explicitly construct a branching-coalescing particle system $(\Gs_T^\om(\beta))_{\beta\geq 0}$ with the desired properties. The main difficulty is that the environment $\om$ {may have} infinitely many jumps on each compact interval. Fix $T > 0$ and $n \in \Nb$ (sampling size) and define 
\[\Lambda_{\textrm{mut}}\coloneqq\{\lambda_{i}^{0},\lambda_{i}^{1} \}_{i \geq 1}, \ \Lambda_{\textrm{\textrm{sel}}}\coloneqq\{\lambda_{i}^{\vartriangle}\}_{i \geq 1}, \ \Lambda_{\textrm{coal}}\coloneqq\{\lambda_{i,j}^{\blacktriangle} \}_{i,j \geq 1, i \neq j},\]
where $\lambda_{i}^{0}, \lambda_{i}^{1}, \lambda_{i}^{\vartriangle}$ and $\lambda_{i,j}^{\blacktriangle}$ are independent Poisson processes on $[0,T]$ with {parameters $\theta \nu_0, \theta \nu_1, \sigma$ and $1$, respectively}. For $\beta \in [0,T]$, let $\tilde \omega (\beta) := \omega (T) - \omega ((T-\beta)-)$ and  $I_{\tilde \omega} := \{ \beta \in [0,T]: \Delta \tilde \omega (\beta) > 0 \}$; $I_{\tilde \omega}$ is the countable set of jump times of $\tilde \omega$. {Let $\mathcal{U}_{\tilde \omega} := \{ U_i(\beta) \}_{i \geq 1, \beta \in I_{\tilde \omega}}$ be a i.i.d. family of uniform random variables on $(0,1)$}. Assume, without loss of generality, that the arrival times of $\lambda_i^0,\lambda_i^1, \lambda_i^{\vartriangle}$ and $\lambda_{i,j}^{\blacktriangle}$, $i,j\in\Nb$, $i\neq j$, are countable, distinct between them and distinct from the jump times of $\tilde \omega$. Let $I_{\textrm{coal}}$ (resp. $I_{\textrm{\textrm{sel}}}$) be the set of arrival times of $\Lambda_{\textrm{coal}}$ (resp. $\Lambda_{\textrm{\textrm{sel}}}$). 
\smallskip

We first construct a set {$\Vs^\om\subset \Nb \times [0,T]$} {of \emph{virtual lines},} representing the set of lines that would be part of the ASG if there were no coalescence events. In particular, once a line enters this set, it will remain there. The set $\Vs^\om$ is constructed on the basis of the set of potential branching times $I_{\textrm{bran}}\coloneqq I_{\tilde{\om}}\cup I_{\textrm{sel}}$ as follows. Consider the (countable) set \[S_{\textrm{bran}}\coloneqq\{(\beta_1,\ldots,\beta_k):k\in\Nb,\, 0\leq \beta_1<\cdots<\beta_k, \beta_i\in I_{\textrm{bran}}, i\in[k]\},\]
and fix an injective function $i_\star:[n]\times S_{\textrm{bran}}\to\Nb\setminus[n]$. The set $\Vs^\om$ is {determined} as follows.
{\begin{enumerate}
 \item For any $i\in[n]$ (i.e. in the initial sample) and $\beta\in[0,T]$: $(i,\beta)\in\Vs^\om$.
 \item For any $(\beta_1,\ldots,\beta_k)\in S_{\textrm{bran}}$, $j\in[n]$, and $\beta\in[\beta_k,T]$: $(i_\star(j,\beta_1,\ldots,\beta_k),\beta)\in\Vs^\om$ if and only: 
 \begin{itemize}
 \item for any $\ell\in[k]$ with $\beta_\ell\in I_{\tilde{\om}}$, $U_{i_\star(j,\beta_1,\ldots,\beta_{\ell-1})}(\beta_\ell)\leq \Delta \tilde \omega(\beta_\ell)$ (or $U_{j}(\beta_1) \leq \Delta \tilde \omega(\beta_1)$ if $\ell=1$),
 \item for any $\ell\in[k]$ such that $\beta_\ell\in I_{\textrm{sel}}$, $\beta_\ell$ is a jump time of $\lambda_{i_\star(j,\beta_1,\ldots,\beta_{\ell-1})}^{\vartriangle}$ (of $\lambda_{j}^{\vartriangle}$ if $\ell=1$),
 \end{itemize}
\end{enumerate}}
and these are all possible virtual lines; see Fig. \ref{fig:asgw}.
\smallskip

\begin{figure}[t!]
\scalebox{0.6}{
\begin{tikzpicture}
%%%%%forward
\pgfmathsetseed{1337}
\draw[dashed, opacity=0.5] (0,-1)--(0,4.5);
\draw[dashed, red, opacity=0.5] (2,-1)--(2,4.5);
\draw[dashed, red, opacity=0.5] (8,-1)--(8,4.5);
\draw[dashed, red, opacity=0.5] (12,-1)--(12,4.5);
\draw[dashed, red, opacity=0.5] (15,-1)--(15,4.5);
\draw[dashed, opacity=0.5] (14,-1)--(14,4.5);
\draw[dashed, opacity=0.5] (10,-1)--(10,4.5);

\node [right, opacity=0.5] at (7.2,-2) {$t$};
\node [right, opacity=0.5] at (1.8,-1.5) {$t_0$};
\node [right, opacity=0.5] at (7.8,-1.5) {$t_1$};
\node [right, opacity=0.5] at (11.8,-1.5) {$t_2$};
\node [right, opacity=0.5] at (14.8,-1.5) {$T$};
\node [right, opacity=0.5] at (-0.2,-1.5) {$0$};
 \draw[-{angle 45[scale=5]}, opacity=0.5] (6.5,-1.8) -- (8.5,-1.8) node[text=black, pos=.6, xshift=7pt]{};

%%%%backward

\node [right, opacity=1] at (1.8,5) {$\beta_5$};
\node [right, opacity=1] at (7.8,5) {$\beta_4$};
\node [right, opacity=1] at (9.8,5) {$\beta_3$};
\node [right, opacity=1] at (11.8,5) {$\beta_2$};
\node [right, opacity=1] at (13.8,5) {$\beta_1$};
\node [right, opacity=1] at (14.8,5) {$0$};
\node [right, opacity=1] at (-0.2,5) {$T$};
\node [right, opacity=1] at (7.2,5.5) {$\beta$};
 \draw[-{angle 45[scale=5]}, opacity=1] (8.5,5.3) -- (6.5,5.3) node[text=black, pos=.6, xshift=7pt]{};
 
%%%virtual 
\draw[very thick, opacity=1] (0,-1)--(15,-1);
\draw[opacity=0.4] (0,-0.5)--(15,-0.5);
\draw[very thick,opacity=1] (13,-0.5)--(15,-0.5);

\draw[very thick,opacity=1] (0,0)--(15,0);

\draw[opacity=0.4] (0,0.5)--(14,0.5);
\draw[very thick,opacity=1] (7,0.5)--(14,0.5);

\draw[opacity=0.4] (0,1)--(12,1);
\draw[very thick,opacity=1] (0,1.5)--(12,1.5);   

\draw[very thick,opacity=1] (0,2)--(10,2);

\draw[opacity=0.4] (0,2.5)--(8,2.5);
\draw[very thick,opacity=1] (0,3)--(8,3);

\draw[opacity=0.4] (0,3.5)--(2,3.5);
\draw[opacity=0.4] (0,4)--(2,4);
\draw[very thick,opacity=1] (1,4)--(2,4);
\draw[very thick,opacity=1] (0,4.5)--(2,4.5);
  
%%%%mutations   

%%%%connections
\draw[very thick,opacity=1]   (14,0) to[out=30,in=-30] (14,0.5);

\draw[opacity=0.4]     (12,-0.5) to[out=30,in=-30] (12,1);
\draw[very thick,opacity=1]    (12,0.5) to[out=30,in=-30] (12,1.5);

\draw[very thick,opacity=1]    (10,-1) to[out=30,in=-30] (10,2);

\draw[opacity=0.4]     (8,1) to[out=30,in=-30] (8,2.5);
\draw[very thick,opacity=1]   (8,1.5) to[out=30,in=-30] (8,3);

\draw[opacity=0.4]     (2,0.5) to[out=30,in=-30] (2,3.5);
\draw[very thick,opacity=1]    (2,1.5) to[out=30,in=-30] (2,4);
\draw[very thick,opacity=1]    (2,2) to[out=30,in=-30] (2,4.5);

\draw[very thick,opacity=1]   (7,1.5)--(7,0.5);
\draw[very thick,opacity=1]   (1,3)--(1,4);
\draw[very thick,opacity=1]   (13,0)--(13,-0.5);

\node[ultra thick] at (9.5,1.5) {$\bigtimes$} ;    
\node[ultra thick] at (6,0) {$\bigtimes$} ;  
\draw (3,1.5) circle (1.5mm)  [fill=white!100];  
\node[ultra thick] at (1,-0.5) {$\bigtimes$} ;        
%%%%backward
\end{tikzpicture}}
	\caption{Illustration of the construction of the quenched ASG. The environment $\om$ has jumps at forward times $t_0$, $t_1$, $t_2$; backward times $\beta_1,\ldots,\beta_5$ belong to the set of potential branching times $\tilde I_{\rm{bran}}$. Virtual lines are depicted grey or black; active lines are black. The ASG in $[0,T]$ consists of the set of active lines together with their connections and mutation marks.}
	\label{fig:asgw}
\end{figure} 
Let $V^\om(\beta)\defeq \{i\in\Nb: (i,\beta)\in\Vs^\om\}$. {According to Lemma \ref{nbfinilines}} below, $V^\om(\beta)$ is almost surely finite for all $\beta \in [0,T]$. Now, for $\beta \in I_{\textrm{coal}}$, let $(a_{\beta},b_{\beta})$ be the pair $(i,j)$ such that $\beta$ is an arrival time of $\lambda_{i,j}^{\blacktriangle}$. Since the Poisson processes $\lambda_{i,j}^{\blacktriangle}$, $i\neq j$, have distinct jump times, $(a_{\beta},b_{\beta})$ is uniquely defined. Let 
\[\tilde I_{\textrm{coal}}\coloneqq \{ \beta \in I_{\textrm{coal}} : \ a_{\beta},b_{\beta} \in V^\om(\beta) \}\quad\textrm{and}\quad\tilde I_{\textrm{bran}}\coloneqq \{ \beta \in I_{\textrm{bran}} : \ V^\om(\beta-)\subsetneq V^\om(\beta) \}.\] 
Since $V^\om(T)$ is independent of $\Lambda_{\textrm{coal}}$ and almost surely finite, it follows that $\tilde I_{\textrm{coal}}$ and $\tilde I_{\textrm{bran}}$ are almost surely finite. Let $\beta_1 < \cdots < \beta_m$ be the elements of $\tilde I_{\textrm{coal}} \cup \tilde I_{\textrm{bran}}$ (set $\beta_{0} \coloneqq 0$ and $\beta_{m+1} \coloneqq T$ for convenience). We define $V_{\rm{on}}^\om(\beta)\subset V^\om(\beta)$, the set of \emph{active} lines at time $\beta$ as follows (see also Fig. \ref{fig:asgw}). {For $\beta=0$, we set $V_{\rm{on}}^\om(0)\coloneqq V^\om(0)$ and, for $\beta\in(\beta_\ell,\beta_{\ell+1})$, we set $V_{\rm{on}}^\om(\beta)\coloneqq V_{\rm{on}}^\om(\beta_\ell)$. For $\beta=\beta_\ell\in \tilde I_{\textrm{coal}}$, we set {$V_{\rm{on}}^\om(\beta_\ell)\coloneqq  V_{\rm{on}}^\om(\beta_\ell-)\setminus \{ a_{\beta_\ell} \}$ if $\{a_{\beta_\ell},b_{\beta_\ell}\}\subset V_{\rm{on}}^\om(\beta_\ell-)$}, and 
$V_{\rm{on}}^\om(\beta_\ell)\coloneqq V_{\rm{on}}^\om(\beta_\ell-)$ otherwise. Finally, for $\beta=\beta_\ell\in \tilde I_{\textrm{bran}}$, we set $V_{\rm{on}}^\om(\beta_\ell)\coloneqq V_{\rm{on}}^\om(\beta_\ell-) \cup J_\ell$,
where the set $J_\ell$ consists of the integers $i \in V^\om(\beta_\ell) \setminus V^\om(\beta_\ell-)$ such that: $i=i_\star(j,\beta_\ell)$ for some $j\in [n]\cap V_{\rm{on}}^\om(\beta_\ell-)$, or $i=i_\star(j,\hat \beta_1,\ldots,\hat \beta_k,\beta_\ell)$ for some $(j,\hat \beta_{1},\ldots,\hat \beta_{k})\in i_\star^{-1}(V_{\rm{on}}^\om(\beta_\ell-)\setminus [n])$ with $\hat{\beta}_k<\beta_\ell$.}
\smallskip

The {ASG on $[0,T]$} is then the branching-coalescing system starting with $n$ lines at levels in $[n]$, consisting at any time $\beta\in[0,T]$ of the lines in $V_{\rm{on}}^\om(\beta)$,  and where:
\begin{itemize}
\item[(i)] {For} $\beta \in I_{\textrm{bran}}$ such that $V_{\rm{on}}^\om(\beta-)\subsetneq V_{\rm{on}}^\om(\beta)$ and $i\in V_{\rm{on}}^\om(\beta)\setminus V_{\rm{on}}^\om(\beta-)$, either there is $(j,\hat \beta_{1},\ldots,\hat \beta_{k})\in[n]\times S_{\textrm{bran}}$ with $\hat \beta_{k}<\beta$ such that $i=i_\star(j,\hat \beta_1,\ldots,\hat \beta_k,\beta)$, or there is $j\in[n]$ such that $i=i_\star(j,\beta)$.
In the first case, line $i_\star(j,\hat \beta_1,\ldots,\hat \beta_k)$ branches at time $\beta$ into $i_\star(j,\hat \beta_1,\ldots,\hat \beta_k)$ (continuing line) and $i$ (incoming line). In the second case, line $j$ branches at time $\beta$ into $j$ (continuing line) and $i$ (incoming line). 
\item[(ii)] {For} $\beta \in I_{\textrm{coal}}$ such that $V_{\rm{on}}^\om(\beta)\subsetneq V_{\rm{on}}^\om(\beta-)$ and $i\in V_{\rm{on}}^\om(\beta-)\setminus V_{\rm{on}}^\om(\beta)$, $i=a_\beta$ and $b_{\beta}\in V_{\rm{on}}^\om(\beta)$. Thus, lines $i$ an $b_\beta$ merge into $b_\beta$ at time $\beta$.
\item[(iii)] At each $\beta \in [0,T]$ that is an arrival time of $\lambda_{i}^{0}$ (resp. $\lambda_{i}^{1}$) for some $i \in V_{\rm{on}}^\om(\beta)$, we mark line $i$ with a beneficial (resp. deleterious) mutation at time $\beta$. 
\end{itemize}
It is straightforward to see that the so-constructed branching-coalescing particle system satisfies the requirements {(i), (ii), (iii'), (iv) and (v)} of Definition \ref{defannealdasg}. This ends the proof.
\end{proof}
It remains to prove the following lemma.

\begin{lemma} \label{nbfinilines}
The set $V^\omega(\beta)$ is almost surely finite for any $\beta \in [0,T]$. 
\end{lemma}
\begin{proof}
We keep using the notation introduced in the proof of Proposition \ref{eqasg}. For $\delta>0$, we consider the environment $\om^\delta$ defined via \eqref{defomdelta}. We couple the sets of virtual lines $\Vs^\om$ and $\Vs^{\om^\delta}$ associated to $\om$ and $\om^\delta$, respectively, by using the same random {sets $\Lambda_{\rm{sel}}$ and $\mathcal{U}_{\tilde \omega}$ (note that for $\beta \in I_{\tilde \omega}$ with $\Delta \tilde \om(\beta)<\delta$, $\Delta \tilde \om^\delta(\beta)=0<U_i(\beta)$). 
Let $N_T^\omega(\beta)\coloneqq |V^\om(\beta)|$ and $N_T^{\omega^\delta}(\beta)\coloneqq |V^{\om^\delta}(\beta)|$, $\beta\in[0,T]$. Since $\beta\mapsto N_T^\om(\beta)$ is non-decreasing, it is enough to prove that $N_T^\om(T)<\infty$ almost surely. From the construction of the set of virtual lines, it follows that $N_T^{\omega^{\delta}}(\beta)$ increases almost surely to $N_T^{\omega}(\beta)$ as $\delta\to 0$.} By the monotone converge theorem we get for all $\beta \in [0,T]$: 
\begin{eqnarray}
\lim_{\delta \rightarrow 0} \mathbb{E}[N_T^{\omega^{\delta}}(\beta)\mid N_T^{\omega^{\delta}}(0)=n] = \mathbb{E}[N_T^{\omega}(\beta)\mid N_T^{\omega}(0)=n]. \label{monotcv}
\end{eqnarray}
Recall that $\tilde{\om}^\delta(\beta)\defeq \om^\delta(T)-\om^\delta((T-\beta)-)$, $\beta\in[0,T]$, and that $\tilde \omega^{\delta}$ has finitely many jumps in $[0,T]$. Let $T_1 <\cdots < T_N$ be the jump times of $\tilde \omega^{\delta}$. The process $(N_T^{\omega^{\delta}}(\beta))_{\beta \in [0,T]}$ has the following transitions:
\begin{enumerate}
 \item on $(T_i,T_{i+1})$: $N_T^{\omega^{\delta}}$ jumps from $k$ to $k+1$ at rate $k\sigma,$
 \item at time $T_i$: $N_T^{\omega^{\delta}}$ jumps from $k$ to $k+  B_{k}(\Delta \tilde \omega^{\delta}(T_i))$, where $B_{k}(\Delta \tilde \omega^{\delta}(T_i)) \sim \bindist{k}{\Delta \tilde \omega^{\delta}(T_i)}$.
\end{enumerate}
Note that for each $T_i$ we have $\Delta \tilde \omega^{\delta}(T_i)=\Delta \tilde \omega(T_i)$. This yields in particular 
\begin{eqnarray}
\mathbb{E}[N_T^{\omega^{\delta}}(T_i)\mid N_T^{\omega^{\delta}}(T_i-)]=(1+\Delta \tilde \omega(T_i))N_T^{\omega^{\delta}}(T_i-). \label{expectationatjumps}
\end{eqnarray}
Using successively Lemma \ref{expectationbetweenjumps} (see below) and \eqref{expectationatjumps} we get 
\[ \mathbb{E}[N_T^{\omega^{\delta}}(T)\mid N_T^{\omega^{\delta}}(0)=n]\leq ne^{\sigma T} \prod_{\beta \in [0,T]: \Delta \tilde \omega (\beta) \geq \delta}(1+\Delta \tilde \omega(\beta)). \]
In particular, we have 
\begin{eqnarray}
\mathbb{E}[N_T^{\omega^{\delta}}(T)\mid N_T^{\omega^{\delta}}(0)=n]\leq ne^{\sigma T} \prod_{\beta \in I_{\tilde \omega} \cap [0,T]}(1+\Delta \tilde \omega(\beta)) < \infty. \label{expectationatt}
\end{eqnarray}
Letting $\delta$ go to $0$ in \eqref{expectationatt} and using \eqref{monotcv} we get 
\[ \mathbb{E}[N_T^{\omega}(T)\mid N_T^{\omega}(0)=n]\leq ne^{\sigma T} \prod_{\beta \in I_{\tilde \omega} \cap [0,T]}(1+\Delta \tilde \omega(\beta)) < \infty. \]
This concludes the proof. 
\end{proof}
\begin{lemma} \label{expectationbetweenjumps}
Let $0 \leq \beta_1 <\beta_2 \leq T$ be such that $\tilde \omega^{\delta}$ has no jump times on $(\beta_1,\beta_2]$. Then, we have \[\mathbb{E}[N_T^{\omega^\delta}(\beta_2)\mid N_T^{\omega^\delta}(\beta_1)] \leq e^{\sigma (\beta_2-\beta_1)} N_T^{\omega^\delta}(\beta_1).\]
\end{lemma}

\begin{proof}
Since $\tilde \omega^{\delta}$ has no jump times on $(\beta_1,\beta_2]$, on this interval $N_T^{\omega^\delta}$ is the Markov chain on $\mathbb{N}$ with generator 
$\mathcal{G}_N f(n) = \sigma n (f(n+1)-f(n)).$
Let {$f_M(n)\coloneqq n \wedge M$.} Note that, for any $M,n \geq 1$, we have $\mathcal{G}_N f_M(n) \leq \sigma f_M(n)$. Applying Dynkin's formula to $N_T^{\omega^\delta}$ on $(\beta_1,\beta_2]$ with the function $f_M$ we obtain
\begin{align*}
\mathbb{E}\left[f_M(N_T^{\omega^\delta}(\beta_2)) \mid N_T^{\omega^\delta}(\beta_1)\right]&=f_M(N_T^{\omega^\delta}(\beta_1))+\mathbb{E}\left[\int_{\beta_1}^{\beta_2} \mathcal{G}_N f_M(N_T^{\omega^\delta}(\beta))\dd \beta\mid N_T^{\omega^\delta}(\beta_1)\right] \\
&\leq f_M(N_T^{\omega^\delta}(\beta_1))+\sigma \mathbb{E}\left[\int_{\beta_1}^{\beta_2} f_M(N_T^{\omega^\delta}(\beta))\dd \beta\mid N_T^{\omega^\delta}(\beta_1)\right] \\
&= f_M(N_T^{\omega^\delta}(\beta_1))+\sigma \int_{\beta_1}^{\beta_2} \mathbb{E}\left[f_M(N_T^{\omega^\delta}(\beta))\mid N_T^{\omega^\delta}(\beta_1)\right] \dd \beta. 
\end{align*}
By Gronwall's lemma we deduce that $\mathbb{E}[f_M(N_T^{\omega^\delta}(\beta_2)) \mid N_T^{\omega^\delta}(\beta_1)]\leq e^{\sigma (\beta_2-\beta_1)} f_M(N_T^{\omega^\delta}(\beta_1))$. The result follows letting $M\to\infty$ and using the monotone convergence theorem. 
\end{proof}

\subsection{Definitions related to Section \ref{s25}: the killed ASG}\label{s42}
The k-ASG as a branching-coalescing system of particles is defined as follows (see Fig. \ref{fig:backfor}).
\begin{definition}[The annealed{/quenched} k-ASG]\label{defkilledasgannealed}
The \emph{annealed k-ASG} with parameters $\sigma,\theta,\nu_0,\nu_1$, and environment driven by a pure-jump subordinator with L\'evy measure $\mu$, of a sample of size $n$ is the branching-coalescing particle system $\bar{\Gs}\defeq(\bar{\Gs}(\beta))_{\beta\geq 0}$ starting with $n$ lines and with the following dynamic.
\begin{itemize}
 \item[(i)] {Each} line splits into two lines, an incoming line and a continuing line, at rate  $\sigma$.
 \item[(ii)] {Every} given pair of lines {coalesces} into a single line at rate $2$.
 \item[(iii)] {Every} group of $k$ lines is subject to a simultaneous branching at rate $\sigma_{m,k}$ {(defined in Eq. \eqref{smk}),} where $m$ denotes the total number of lines in the ASG before the simultaneous branching event. At the simultaneous branching event, each line in the group involved splits into two lines, an incoming line and a continuing line.
 \item[(iv)] {Each} line is killed at rate $\theta \nu_1$.
 \item[(v)] {Each} line sends the process to the cemetery state $\dagger$ at rate $\theta \nu_0$.
\end{itemize}
{Let $\om\in\Db^\star$. The \emph{quenched k-ASG} with parameters $\sigma,\theta,\nu_0,\nu_1$, and environment $\om$, of a sample of size $n$ at time $T$ is the branching-coalescing particle system $\bar{\Gs}_{T}^\om\defeq(\bar{\Gs}_{T}^{\om}(\beta))_{\beta\geq 0}$ starting at $\beta=0-$ with $n$ lines and evolving according to (i), (ii), (iv) and (v) of the previous definition and replacing (iii) by
\begin{itemize}
 \item[(iii')] {If} at time $\beta$, we have $\Delta \om(T-\beta)>0$, then any line splits into two lines, an incoming line and a continuing line, with probability $\Delta \om(T-\beta)$, independently from the other lines.
\end{itemize}}
\end{definition}
% A similar reasoning leads, in the quenched case, to the following definition. 
% \begin{definition}[The quenched k-ASG]\label{defkilledasgquenched}
% Let $\om:\Rb\to\Rb$ be a fixed environment. The \emph{quenched k-ASG} with parameters $\sigma,\theta,\nu_0,\nu_1$, and environment $\om$, associated to a sample of the population of size $n$ at time $T$ is the branching-coalescing particle system $(\Gs_{\dagger}^T(\om,\beta))_{\beta\geq 0}$ starting at $\beta=0-$ with $n$ lines and with the following dynamic.
% \begin{itemize}
%  \item[(i)] {Each} line splits into two lines, an incoming line and a continuing line, at rate  $\sigma$.
%  \item[(ii)] {Every} given pair of lines {coalesces} into a single line at rate $2$.
%  \item[(iii)] {If} at time $\beta$, we have $\Delta \om(T-\beta)>0$, then any particle lines splits into two lines, an incoming line and a continuing line, with probability $\Delta \om(T-\beta)$, independently from the other lines.
%  \item[(iv)] {Each} line is killed at rate $\theta \nu_1$.
%  \item[(v)] {Each} line send the process to the cemetery state $\dagger$ at rate $\theta \nu_0$.
% \end{itemize}
% \end{definition}
\begin{remark}
The branching-coalescing system underlying the quenched k-ASG is well-defined as it can be constructed on the basis of the quenched ASG.
\end{remark}
\subsection{Definitions related to Section \ref{s26}: the pruned lookdown ASG}\label{s43}
In this section, we give a detailed construction of the pLD-ASG, which incorporates the effect of the environment.
\smallskip

First, we construct the (annealed/quenched) \textit{lookdown ASG} (LD-ASG). The latter is the ASG equipped with a numbering of its lines   encoding the hierarchy given by the pecking order. This is done as follows. Consider a realization of the (annealed/quenched) ASG in $[0,T]$ starting with one line, which is assigned level $1$. When the line at level $i$ coalesces with the line at level $j>i$, the resulting line is assigned level $i$; the level of each line having level $k > j$ before the coalescence is decreased by $1$. When a group of lines with levels $i_1< i_2<\ldots<i_N$ experiences a simultaneous branching, the incoming (resp. continuing) line of the descendant line with level $i_k$ gets level $i_k+k-1$ (resp. $i_k+k$), respectively; a line having level $j$ before the branching, with $i_k<j<i_{k+1}$ gets level $j+k$; a line having level $j>i_N$ before the branching gets level $j+N$. Mutations do not affect the levels. See Fig. \ref{LDASG}(left) for an illustration.
 \smallskip
 \begin{figure}[t!]
    \begin{center}
		\begin{minipage}{0.5 \textwidth}
		\centering
    \scalebox{0.6}{\begin{tikzpicture}
       %%%%%0 and T%%%%%%%%%%%%%%%%%%%
       \draw[dashed,thick,opacity=0.3] (-0.5,-0.5) --(-0.5,4.5);
        \draw[dashed,thick,opacity=0.3] (8.5,-0.5) --(8.5,4.5);
       %%%%%branching coalescences%%%%  
        \draw[dashed,thick,opacity=0.3] (0.8,-0.5) --(0.8,0) (0.8,1)--(0.8,4.5);
        \draw[dashed,thick,opacity=0.3] (2,-0.5) --(2,1) (2,2)--(2,3) (2,4)--(2,4.5);
        \draw[dashed,thick,opacity=0.3] (6.2,-0.5) --(6.2,0) (6.2,1) --(6.2,2) (6.2,4)--(6.2,4.5);
        \draw[dashed,thick,opacity=0.3] (7.5,-0.5) --(7.5,1) (7.5,2) --(7.5,4.5) ;
        \draw[dashed,thick,opacity=0.3] (4.5,-0.5) --(4.5,1) (4.5,2) --(4.5,4.5) ;
   %%%%%time%%%%%%%%%%%%%%%%%%%%%%     
        \node [right] at (-0.8,-0.9) {$T$};
%         \node [right] at (0.4,-0.8) {$\beta_2$};
%         \node [right] at (5.6,-0.8) {$\beta_1$};
        \node [right] at (8.3,-0.9) {$0$};
       \node [right] at (3.5,-1.5) {$\beta$};
       
        \draw[-{triangle 45[scale=5]}] (4.5,-1.2) -- (2.5,-1.2) node[text=black, pos=.6, xshift=7pt]{}; 
        
  %%%%%%%levels%%%%%%%%%%%%%%%%%%%%%%%%%%%
  \node [above] at (8.3,1) {$1$};
  \node [above] at (7.2,2) {$1$};
  \node [above] at (7.2,1) {$2$};
  \node [above] at (5.8,0) {$3$};
  \node [above] at (5.8,1) {$4$};
  \node [above] at (5.8,2) {$2$};
  \node [above] at (5.8,4) {$1$};
  \node [above] at (4.2,0) {$3$};
  \node [above] at (4.2,2) {$2$};
  \node [above] at (4.2,4) {$1$};
  \node [above] at (1.7,0) {$5$};
  \node [above] at (1.7,1) {$3$};
  \node [above] at (1.7,2) {$4$};
  \node [above] at (1.7,3) {$1$};
  \node [above] at (1.7,4) {$2$};
  
  \node [above] at (0.5,1) {$3$};
  \node [above] at (0.5,2) {$4$};
  \node [above] at (0.5,3) {$1$};
  \node [above] at (0.5,4) {$2$};
%%%%%%%%ancestral lines%%%%%%%%%%%%%%%%%     
        \draw[thick] (-0.5,4) -- (6.2,4);
        \draw[thick] (-0.5,2) -- (7.5,2);
        \draw[thick] (4.5,1) -- (8.5,1);
        \draw[thick] (0.8,0) -- (6.2,0);
        \draw[thick] (2,1) -- (-0.5,1);
        \draw[thick] (2,3) -- (-0.5,3);
   %%%%%%%%%%%%% relevant arrows%%%%%%%%%%
        \draw[-{triangle 45[scale=5]},thick] (4.5,2) -- (4.5,1) node[text=black, pos=.6, xshift=7pt]{};
        \draw[-{triangle 45[scale=5]},thick] (0.8,1) -- (0.8,0);
        \draw[-{open triangle 45[scale=5]},thick] (7.5,2) -- (7.5,1);
        
        %\draw[-{open triangle 45[scale=5]},thick] (1.7,2) -- (1.7,4);
        \draw[-{open triangle 45[scale=5]},thick] (2,3) -- (2,4);
        \draw[-{open triangle 45[scale=5]},thick] (2,1) -- (2,2);
        \draw[-{open triangle 45[scale=5]},thick] (6.2,4) -- (6.2,2);
        \draw[-{open triangle 45[scale=5]},thick] (6.2,0) -- (6.2,1);
          %%%%%mutations%%%%%%
        \node[thick] at (0.2,4) {$\bigtimes$} ;
        \node[ultra thick] at (8,1) {$\bigtimes$} ;
        \draw[thick] (3.6,2) circle (1.5mm)  [fill=white!100];  
        \draw[thick] (1.5,1) circle (1.5mm)  [fill=white!100]; 
        \end{tikzpicture}    }
        \end{minipage}\begin{minipage}{0.45 \textwidth}
        \centering
    \scalebox{0.6}{\begin{tikzpicture}   
       %%%%%0 and T%%%%  
        \draw[dashed,thick,opacity=0.3] (-0.5,-0.5) --(-0.5,4.5);
        \draw[dashed,thick,opacity=0.3] (8.5,-0.5) --(8.5,4.5);
       
        %%%%%%%%mtation times%%%%%%%%%%%%%%%%
        \draw[dashed,thick,opacity=0.3] (0.2,-0.5) --(0.2,4.5) ;
        \draw[dashed,thick,opacity=0.3] (1.5,-0.5) --(1.5,4.5) ;
        \draw[dashed,thick,opacity=0.3] (3.6,-0.5) --(3.6,4.5) ;
        \draw[dashed,thick,opacity=0.3] (8,-0.5) --(8,4.5) ;
        
   %%%%%time%%%%%%%%%%%%%%%%%%%%%%     
        \node [right] at (-0.8,-0.9) {$T$};
        \node [right] at (-0.1,-0.9) {$\beta_4$};
        \node [right] at (1.2,-0.9) {$\beta_3$};
        \node [right] at (3.3,-0.9) {$\beta_2$};
        \node [right] at (7.7,-0.9) {$\beta_1$};
        \node [right] at (8.3,-0.9) {$0$};
        \node [right] at (3.5,-1.5) {$\beta$};
       
        \draw[-{triangle 45[scale=5]}] (4.5,-1.2) -- (2.5,-1.2) node[text=black, pos=.6, xshift=7pt]{}; 
  %%%%%%%levels%%%%%%%%%%%%%%%%%%%%%%%%%%%
  \node [above] at (8.3,1) {$1$};
  \node [above] at (7.2,2) {$1$};
  \node [above] at (7.2,1) {$2$};
  \node [above] at (5.8,0) {$3$};
  \node [above] at (5.8,1) {$4$};
  \node [above] at (5.8,2) {$2$};
  \node [above] at (5.8,4) {$1$};
  \node [above] at (4.2,0) {$3$};
  \node [above] at (4.2,2) {$2$};
  \node [above] at (4.2,4) {$1$};
  \node [above] at (1.7,1) {$3$};
  \node [above] at (1.7,2) {$4$};
  \node [above] at (1.7,3) {$1$};
  \node [above] at (1.7,4) {$2$};
  
  \node [above] at (0.5,1) {$3$};
  \node [above] at (0.5,3) {$1$};
  \node [above] at (0.5,4) {$2$};
%%%%%%%%%%%updating levels
  \node [above] at (7.8,1) {$1$};
  
  \node [above] at (3.3,4) {$1$};
  \node [above] at (3.3,2) {$2$};
  
   \node [above] at (1.2,4) {$2$};
   \node [above] at (1.2,3) {$1$};
  \node [above] at (1.2,1) {$3$};
  
  \node [above] at (0,3) {$1$};
  \node [above] at (0,1) {$2$};

%%%%%%%%ancestral lines%%%%%%%%%%%%%%%%%     
        \draw[semithick] (3.6,4) -- (6.2,4);
        \draw[semithick] (0.2,4) -- (6.2,4);
        \draw[line width=2.5pt] (1.5,2) -- (4.5,2);
        \draw[semithick] (4.5,2) -- (7.5,2);
        \draw[line width=2.5pt] (4.5,1) -- (8.5,1);
        \draw[semithick] (3.6,0) -- (6.2,0);
        \draw[semithick] (2,1) -- (1.5,1);
        \draw[line width=2.5pt] (-0.5,1) -- (1.5,1);
        \draw[semithick] (2,3) -- (-0.5,3);
   %%%%%%%%%%%%% relevant arrows%%%%%%%%%%
        \draw[-{triangle 45[scale=5]},thick] (4.5,2) -- (4.5,1) node[text=black, pos=.6, xshift=7pt]{};
        \draw[-{open triangle 45[scale=5]},thick] (7.5,2) -- (7.5,1);
        \draw[-{open triangle 45[scale=5]},thick] (2,3) -- (2,4);
        \draw[-{open triangle 45[scale=5]},thick] (2,1) -- (2,2);
        \draw[-{open triangle 45[scale=5]},thick] (6.2,4) -- (6.2,2);
        \draw[-{open triangle 45[scale=5]},thick] (6.2,0) -- (6.2,1);
          %%%%%mutations%%%%%%
        \node[thick] at (0.2,4) {$\bigtimes$} ;
        \node[ultra thick] at (8,1) {$\bigtimes$} ;
        \draw[thick] (3.6,2) circle (1.5mm)  [fill=white!100];  
        \draw[thick] (1.5,1) circle (1.5mm)  [fill=white!100]; 
        \end{tikzpicture}    }
        \end{minipage}
        \end{center}
    \caption{LD-ASG (left) and its pLD-ASG (right). Backward time $\beta \in [0,T]$ runs from right to left. In the LD-ASG levels remain constant between the dashed lines, in particular, they are not affected by mutation events. In the pLD-ASG, lines are pruned at mutation events, where an additional updating of the levels takes place. The bold line in the pLD-ASG represents the immune line.}
    \label{LDASG}
\end{figure}
The pLD-ASG is obtained via an appropriate pruning of the lines of the LD-ASG. 
Before describing the pruning procedure, we identify a special line in the LD-ASG: \textit{the immune line}. The immune line at time $\beta$ is the line in the ASG present at time $\beta$ that is the ancestor of the starting line if all the lines at time $\beta$ are assigned the unfit type. {In the absence of mutations, the immune line changes only if it is involved in a coalescence or branching event. If it is involved into a coalescence event, the merged line is the new immune line. If it is involved into a branching event, the continuing line is the new immune line.}
\smallskip

In the presence of mutations, the pLD-ASG is constructed simultaneously with the immune line as follows. Let $\beta_1<\cdots<\beta_m$ be the times at which mutations occur in the LD-ASG in $[0,T]$. In the time interval $[0,\beta_1)$ the pLD-ASG coincides with the LD-ASG and the immune line evolves as before. Now, assume that we have constructed the pLD-ASG together with its immune line up to time $\beta_i-$, where the pLD-ASG contains $n$ lines and the immune line has level $k_0\in[n]$. The pLD-ASG is extended up to time $\beta_i$ according to the following rules:
\begin{itemize}
\item[(i)] If at time $\beta_i$, a line with level $k\neq k_0$ at $\beta_i-$, is hit by a deleterious mutation, we stop tracing back this line; all the other lines are extended up to time $\beta_i$; all the lines with level $j>k$ at time $\beta_i-$ decrease their level by $1$ and the others keep their levels unchanged; the immune line continues on the same line (possibly with a different level).
\item[(ii)] If at time $\beta_i-$, the line with level $k_0$ at $\beta_i-$ is hit by a deleterious mutation, we extend all the lines up to time $\beta_i$; the immune line gets level $n$, but remains on the same line; all the lines having at time $\beta_i-$ a level $j>k_0$ decrease their level by $1$, the others keep their levels unchanged.  
\item[(iii)] If at time $\beta_i$, a line with level $k$ is hit by a beneficial mutation, we stop tracing back all the lines with level $j>k$; the remaining lines are extended up to time $\beta_i$, keeping their levels; the line hit by the mutation becomes the immune line.
\end{itemize}
In $[\beta_i,\beta_{i+1})$, $i\in[m-1]$, and $[\beta_m,T]$, the pLD-ASG evolves as the LD-ASG, and the immune line as in the case without mutations. The next result states the main feature of the pLD-ASG.
\begin{lemma}\label{pruningdonnehtx}
If we assign types at (backward) time $T$ in the pLD-ASG, the true ancestor of the single line at (backward) time $0$ is the line of type $0$ with smallest level or, if all lines have type $1$, it is the immune line. 
\end{lemma}
\begin{proof}
 The proof is analogous to the proof of {\cite[Theorem 4]{LKBW15} which covers the null environment case.}
\end{proof}
\section{Annealed results}\label{S5}
\subsection{Annealed results related to Section \ref{s25} }\label{s51}
{We start this section with the proof of the first part of Theorem \ref{thm2.3}, namely Eqs. \eqref{rmd} and \eqref{md}.}
\begin{proof} [Proof of Theorem \ref{thm2.3}(Part I: reinforced and annealed moment duality)]
Let $H:[0,1]\times \mathbb{N}_0^\dagger\times [0,\infty)$ defined via ${H(x,n,j)\defeq}(1-x)^nf(j)$. Let $(P_t)_{t\geq0}$ and $(Q_t)_{t\geq 0} $ denote the semigroups of $(X,J)$ and $(R,J)$, respectively, i.e. 
\[P_t g(x,j)=\E[g(X(t),J(t)+j)\mid X(0)=x]\quad\textrm{and}\quad Q_th(n,j)=\Eb[h(R(t),J(t)+j)\mid R(0)=n].\] 
Let $(\hat{R},\hat{J})$ be a copy of $(R,J)$, which is independent of $(X,J)$. A straightforward calculation shows that
\begin{equation}\label{conmuta}
P_t(Q_s H)(x,n,j)={\mathbb{E}}[(1-X(t))^{\hat{R}(s)}f(J(t)+\hat{J}(s)+j)\mid X(0)=x,\, \hat{R}(0)=n]=Q_s(P_t H)(x,n,j).
\end{equation}
Let $G$ and $G_{\star}$ be the infinitesimal generators of $(X,J)$ and $(R,J)$, respectively. Clearly, for any $x\in[0,1]$, {the function} $(n,j)\mapsto P_t H(\cdot,n,\cdot)(x,j)$ belongs to the domain of $G_{\star}$. Hence, Eq. \eqref{conmuta} yields 
\begin{equation}\label{comgensemi}
 P_t G_{\star} H(x,n,j)= G_{\star} P_t H(x,n,j).
\end{equation}
We claim that 
\begin{equation}\tag{Claim 4}\label{claimgeneduality}
G H(\cdot,n,\cdot)(x,j)=G_{\star} H(x,\cdot,\cdot)(n,j). 
\end{equation}
Assume that \eqref{claimgeneduality} holds. Set ${u(t,x,n,j)\coloneqq} P_tH(\cdot,n,\cdot)(x,j)$ and ${v(t,x,n,j)\coloneqq} Q_tH(x,\cdot,\cdot)(n,j)$. The Kolmogorov forward equation for $Q$ yields
\begin{equation}\label{Kq}
\frac{\dd}{\dd t} v(t,x,n,j)=G_{\star} v(x,\cdot,\cdot)(n,j).
\end{equation}
{Moreover, using the Kolmogorov forward equation for $P$, {\eqref{claimgeneduality}} and \eqref{comgensemi}, we get
\[\frac{\dd}{\dd t} u(t,x,n,j)=P_t G H(\cdot,n,\cdot)(x,j)= P_t G_{\star} H(x,\cdot,\cdot)(n,j)= G_{\star} u(x,\cdot,\cdot)(n,j).\]

Hence, $u$ and $v$} satisfy Eq. \eqref{Kq}. Since $u(0,x,n,j)=(1-x)^nf(j)=v(0,x,n,j)$, {Eq.~\eqref{rmd} follows from the uniqueness of the initial value problem associated with $G_{\star}$ (see \cite[Thm. 1.3]{D65}). Eq.~\eqref{md}  is obtained using $f\equiv 1$ in Eq.~\eqref{rmd}. It remains to prove \eqref{claimgeneduality}.} Note first that
\begin{align}\label{gx}
G H(\cdot,n,\cdot)(x,j)=&\left[n(n-1)x(1-x)^{n-1}-\left(\sigma x(1-x)+ \theta\nu_0(1-x)-\theta\nu_1 x \right)n(1-x)^{n-1}\right]f(j)\nonumber\\
&+(1-x)^n\int_{(0,1]}\left[(1-xz)^nf(j+z)-f(j)\right]\mu(\dd z).
\end{align}
In addition,
\begin{align}\label{gr}
G_\star H(x,\cdot,\cdot)(n,j)=&n((n-1)+\theta\nu_1)[(1-x)^{n-1}\!-(1-x)^n]f(j)\nonumber\\
&+\sigma n[(1-x)^{n+1}\!-(1-x)^n]f(j)- n\theta\nu_0 (1-x)^nf(j)\nonumber\\
&+\sum\limits_{k=0}^n\binom{n}{k}\int_{(0,1)}y^k (1-y)^{n-k}[(1-x)^{n+k}f(j+y)-(1-x)^nf(j)]\mu(\dd y)\nonumber\\
=&\left[n(n-1)x(1-x)^{n-1}-\left(\sigma x(1-x)+ \theta\nu_0(1-x)-\theta\nu_1 x \right)n(1-x)^{n-1}\right]f(j)\nonumber\\
&+(1-x)^n\sum\limits_{k=0}^n\binom{n}{k}\int_{(0,1)}y^k (1-y)^{n-k}[(1-x)^kf(j+y)-f(j)]\mu(\dd y).
\end{align}
Moreover, using Fubini's theorem, we obtain
\[\sum\limits_{k=0}^n\binom{n}{k}\int_{(0,1)}y^k (1-y)^{n-k}[(1-x)^kf(j+y)-f(j)]\mu(\dd y)=\int_{(0,1)}\left[(1-xy)^nf(j+y)-f(j)\right]\mu(\dd y).
\]
{Hence, \eqref{claimgeneduality} follows after comparing \eqref{gr} with \eqref{gx}}. 
\end{proof}
{We now prove Theorem \ref{thm2.4}-(1), which characterizes the asymptotic type frequency in the annealed setting.}
\begin{proof}[Proof of Theorem \ref{thm2.4}(Asymptotic type frequency)-(1)]
We first show that $X(t)$ has a limit in distribution as $t\to \infty$. Since $\theta >0$ and $\nu_0\in (0,1)$, Eq. \eqref{md} in Theorem \ref{thm2.3} implies that, for any $x\in[0,1]$, the limit of $\mathbb{E}[(1-X(t))^n\mid X(0)=x]$ as $t\to\infty$ exists and satisfy
\begin{equation}\label{mome}
 \lim_{t \rightarrow \infty} \mathbb{E}[(1-X(t))^n|X(0)=x]=\pi_n,\quad n\in\Nb_0, 
\end{equation}
where $\pi_n$ in defined in \eqref{defpin}. Recall that probability measures on $[0,1]$ are completely determined by their positive integer moments and that convergence of positive integer moments implies convergence in distribution. Therefore, Eq. \eqref{mome} implies that there is $\eta_X\in\Ms_1([0,1])$ such that, for any $x\in[0,1]$, conditionally on $\{X(0)=x\}$, the law of $X(t)$ converges in distribution to $\eta_X$ as $t\to\infty$ and
\[\pi_n=\int_{[0,1]} (1-z)^n\eta_X(dz), \quad n\in\Nb.\]
Using dominated convergence, the convergence of the law of $X(t)$ towards $\eta_X$ as $t\to\infty$ extends to any initial distribution. As a consequence of this and the Markov property of $X$, it follows that $X$ admits a unique stationary distribution, which is given by $\eta_X$.
\smallskip

Finally, a first step decomposition for the probability of absorption in $0$ of the process $R$ yields 
\[{\left[n(\sigma+\theta+ n-1)+\sum_{k=1}^n \binom{n}{k}\sigma_{n,k}\right]} \pi_n= n\sigma \pi_{n+1}+ n(\theta\nu_1+ n-1)\pi_{n-1}+{\sum\limits_{k=1}^n \binom{n}{k}\sigma_{n,k} \pi_{n+k}}.\]
Dividing both sides in the previous identity by $n$ and rearranging terms yields Eq. \eqref{recwn}.   
\end{proof}
\subsection{Annealed results related to Section \ref{s26}}\label{s52}
{In this section we prove Theorem \ref{thm2.6}-(1) and Corollary \ref{cor2.7}. Before that we prove the following lemma relating the ancestral type distribution at time $T$ to the number $L(T)$ of lines in the pLD-ASG at time $T$. }
 
 \begin{lemma}\label{exprhtannealed}
For all $T \geq 0$ and $x\in[0,1]$, we have
\begin{equation}
 h_T(x)=1-\mathbb{E}[(1-x)^{L(T)}\mid L(0)=1].\label{pldasgatd}
 \end{equation}
\end{lemma}
\begin{proof} {Since types are assigned to the $L(T)$ lines present in the pLD-ASG at {(backward)} time $T$ according to independent Bernoulli random variables with parameter $x$, the result follows from Lemma \ref{pruningdonnehtx}.}
\end{proof}
{The next result is crucial to describe the asymptotic behavior of $h_T(x)$ as $T\to\infty$.}
\begin{lemma}[Positive recurrence]\label{pr}
The process $L$ is positive recurrent. 
\end{lemma}
\begin{proof}
Since {$L$ is irreducible}, it is enough to prove that the state $1$ is positive recurrent. This holds if $\theta\nu_0>0$, because in this case the hitting time of $1$ is upper bounded by an exponential random variable with parameter $\theta\nu_0$. Now, assume that $\theta=0$ (the case $\theta\nu_0=0$ and $\theta\nu_1>0$, can be easily reduced to this case). We proceed in a similar way as in \cite[Proof of Lem. 2.3]{F13}. Define the function $f:\Nb\to\Rb_+$ via
\[f(n)\coloneqq \sum_{i=1}^{n-1} \frac1{i}\ln\left(1+\frac{1}{i}\right),\]
with the convention that an empty sum equals $0$. Note that $f$ is bounded. Note also that, for $n>1$,
\[n(n-1)(f(n-1)-f(n))=-n\ln\left(1+\frac{1}{n-1}\right)\leq -1.\]
This follows using $x=1/n$ in the inequality {$e^{x} < 1/(1-x)$}, which holds for $x<1$. For any $\varepsilon>0$, set ${n_0(\varepsilon)\coloneqq} \lfloor 1/\varepsilon\rfloor +1.$ Note that for $n>n_0(\varepsilon)$
\[n(f(n+i)-f(n))={n\sum_{j=n}^{n+i-1}\frac1{j}\ln\left(1+\frac{1}{j}\right)}\leq n\ln\left(1+\frac1n\right)i\varepsilon\leq i\varepsilon.\]
Hence, for {$n> n_0(\varepsilon)$,}
\begin{align*}
G_L f(n)&\leq -1 +\frac{\varepsilon}{n}\sum_{i=1}^n \binom{n}{i}\sigma_{n,i} \,i+\sigma\varepsilon=-1+\varepsilon\int_{(0,1)}\sum_{i=1}^n\binom{n-1}{i-1}y^i(1-y)^{n-i}\mu(dy)+\sigma \varepsilon\\
&=-1+\varepsilon\left(\int_{(0,1)}y\mu(dy)+\sigma\right),
\end{align*}
where $\sigma_{n,i}$  is defined in \eqref{smk}. Set $m_0\coloneqq n_0(\varepsilon_\star)$, where $\varepsilon_\star\coloneqq 1/\big(2\int_{(0,1)}y\mu(dy)+2\sigma\big)$ (and we set $m_0\coloneqq 1$ in the particular case $\mu=0$ and $\sigma=0$). In particular, for {$n> m_0$}, we have  $G_Lf(n)\leq -1/2.$
  
Define $T_{m_0}\coloneqq \inf\{\beta>0: L(\beta){\leq m_0}\}$. Applying Dynkin's formula to $L$ with the function $f$ and the stopping time $T_{m_0}\wedge k$, $k\in\Nb$, we obtain
\[\Eb\left[f(L(T_{m_0}\wedge k))\mid L(0)=n\right]=f(n)+\Eb\left[\int_0^{T_{m_0}\wedge k}G_Lf(L(\beta))d\beta\mid L(0)=n\right].\]
Therefore, for {$n> m_0$}, we have
\[0\leq \Eb\left[f(L(T_{m_0}\wedge k))\mid L(0)=n\right]\leq f(n)-\frac1{2}\Eb[T_{m_0}\wedge k\mid L(0)=n].\]
Hence,
$\Eb[T_{m_0}\wedge k\mid L(0)=n]\leq 2f(n).$
Letting $k\to\infty$ yields 
$\Eb[T_{m_0}\mid L(0)=n]\leq 2f(n)<\infty.$
Since $L$ is irreducible, the result follows by standard arguments.
\end{proof}

The first part of the proof of Theorem \ref{thm2.6}-(1) builds on the previous two lemmas. The system of equations \eqref{fr} characterizing the tail probabilities $\mathbb{P}(L(\infty) > n)$ is obtained via Siegmund duality. More precisely, consider the continuous-time Markov chain $D\coloneqq (D(\beta))_{\beta \geq 0}$ with values in $\Nb^\dagger\defeq \Nb\cup\{\dagger\}$ with rates
\[q_D(i,j)\coloneqq \left\{\begin{array}{ll}
            (i-1)(\sigma + \sigma_{i-1,1})&\textrm{if $j=i-1$, $i>1$},\\
            (i-1)\theta\nu_1 +i(i-1)   &\textrm{if $j=i+1$, $i>1$},\\
            \gamma_{i,j}-\gamma_{i,j-1} &\textrm{if {$ 1\leq j< i$}, $i>2$},\\
            (i-1)\theta\nu_0&\textrm{if $j=\dagger$, $i>1$,}
            \end{array}\right.
\]
where we recall that $\gamma_{i,j}\coloneqq \sum_{k=i-j}^{j}\binom{j}{k}\sigma_{j,k}$ {if $1 \leq j<i\leq 2j$ and $\gamma_{i,j}\coloneqq 0$ otherwise,} and $\dagger$ is a cemetery point. Note that $1$ and $\dagger$ are absorbing states of $D$. {The next result relates $L$ and $D$ via duality.} 
\begin{lemma}[Siegmund duality]\label{sd}
 The processes $L$ and $D$ are Siegmund dual, i.e. 
\[\mathbb{P}\left(L(\beta) \geq d\mid L(0)=\ell\right)=\mathbb{P}\left(\ell\geq D(\beta) \mid D(0)=d\right),\qquad \textrm{for all }\ell, d\in \mathbb{N}, t\geq0.\]
\end{lemma}
\begin{proof}
 We consider the function $H:\mathbb{N}\times\mathbb{N}\cup\{\dagger\}\rightarrow \{0,1\}$ defined via $H(\ell,d)\coloneqq 1_{\{\ell\geq d\}}$ (i.e. $H(\ell,d)=1$ if $\ell\geq d$ and $H(\ell,d)=0$ otherwise) and $H(\ell,\dagger)\coloneqq 0$, $\ell,d\in\mathbb{N}$. Let $G_L$ and $G_D$ be the infinitesimal generators of $L$ and $D$, respectively. By \cite[Prop. 1.2]{Jaku} we only have to show that $G_L H(\cdot,d)(\ell)=G_D H(\ell,\cdot)(d)$ for all $\ell,d\in\mathbb{N}$. {From \eqref{kratespldasg}}, we have
 \begin{align}\label{gl}
  G_L H(\cdot,d)(\ell)&=\sigma \ell \,1_{\{\ell+1=d\}} - (\ell-1)(\ell+\theta\nu_1)\,1_{\{\ell=d\}} -\theta\nu_0 \sum\limits_{j=1}^{\ell-1} 1_{\{j<d\leq \ell\}}+ \sum\limits_{k=1}^\ell\binom{\ell}{k}\sigma_{\ell,k} \, 1_{\{\ell< d\leq \ell+k\}}\nonumber\\
  &=\sigma \ell \,1_{\{\ell+1=d\}} - (\ell-1)(\ell+\theta\nu_1)\,1_{\{\ell=d\}} -\theta\nu_0 (d-1) 1_{\{d\leq \ell\}}+ \gamma_{d,\ell} \, 1_{\{\ell< d\}}.
 \end{align}
 Similarly, we have
 \begin{align}\label{gd}
  G_D H(\ell,\cdot)(d)&={\sigma (d-1) \,1_{\{d-1=\ell\}}} - (d-1)(d+\theta\nu_1)\,1_{\{\ell=d\}} -\theta\nu_0 (d-1) 1_{\{d\leq \ell\}}\nonumber\\
  & + \sum\limits_{j=1}^{d-1}\left(\gamma_{d,j}-\gamma_{d,j-1}\right) \, 1_{\{j\leq  \ell <d\}}.
 \end{align}
 Summation by parts yields $\sum_{j=1}^{d-1}\left(\gamma_{d,j}-\gamma_{d,j-1}\right) \, 1_{\{j\leq  \ell<d\}}{= \gamma_{d,\ell}1_{\{\ell< d\}}}$. Thus, the result follows comparing \eqref{gd} with \eqref{gl}.
\end{proof}
{Now, we have all the ingredients to prove Theorem \ref{thm2.6}-(1).} 
\begin{proof}[Proof of Theorem \ref{thm2.6}(Ancestral type distribution)-(1)]
Since $L$ is positive recurrent, $L(T)$ converges in distribution as $T\to\infty$ towards the stationary distribution $\eta_L$. In particular, we infer from Eq. \eqref{pldasgatd} that the limit $h(x)$ of $h_T(x)$ as $T\to\infty$ exists and satisfies
\begin{align*}
h(x)&=1-\Eb[(1-x)^{L(\infty)}]=1-\sum_{\ell=1}^\infty \Pb(L(\infty)=\ell)(1-x)^\ell\\
&=\sum_{\ell=0}^\infty\Pb(L(\infty)>\ell)(1-x)^\ell-(1-x)\sum_{\ell=1}^\infty\Pb(L(\infty)>\ell-1)(1-x)^{\ell-1},
\end{align*}
and Eq. \eqref{represh(x)tailpldasg} follows. It remains to prove \eqref{fr}. From Lemma \ref{sd} we infer that $a_n=d_{n+1}$, where
\[d_n\coloneqq \mathbb{P}(\exists \beta>0: D(\beta)=1\mid D(0)=n), \quad n\geq 1.\]
Applying a first step decomposition to the process $D$, we obtain, for $n>1$,
\begin{equation}\label{r1}
\left[(n-1)(\sigma +\theta +n ) +{\gamma_{n,n-1}}\right]d_n= (n-1)\sigma d_{n-1}+(n-1)(\theta\nu_1+n)d_{n+1} + \sum\limits_{j=1}^{n-1} (\gamma_{n,j}-\gamma_{n,j-1})d_j.
\end{equation}
Using summation by parts and rearranging terms in \eqref{r1} yields
\begin{equation}\label{r2}
(\sigma +\theta +n ) d_n=  \sigma d_{n-1}+(\theta\nu_1+n)d_{n+1} +\frac{1}{n-1} \sum\limits_{j=1}^{n-1} \gamma_{n,j}(d_j-d_{j+1}),\quad n>1.
\end{equation}
The result follows.
\end{proof}

\begin{proof}[Proof of Corollary \ref{cor2.7}]
Since $\theta = 0$, the line-counting processes $R$ and $L$ have the same distribution. Hence, combining Lemma \ref{exprhtannealed} and {\eqref{md} (from Theorem \ref{thm2.3})} applied to $n=1$, we obtain
\begin{equation}\label{idzm}
 h_T(x)=\Eb[X(T)\mid X(0)=x],
\end{equation}
which proves the first part of the statement.
Moreover, for $\theta=0$, $X$ is a bounded submartingale, and hence $X(T)$ has almost surely a limit as $T\to\infty$, which we denote by $X(\infty)$. Letting $T\to \infty$, in the identity \eqref{idzm} yields
\begin{equation}\label{auxh0}
 h(x)=\Eb[X(\infty)\mid X(0)=x].
\end{equation}
Moreover, using {\eqref{md} (from Theorem \ref{thm2.3})} with $n=2$, we get
\[\Eb[(1-X(T))^2\mid X(0)=x]=\Eb[(1-x)^{L(T)}\mid L(0)=2].\]
Letting $T\to \infty$ and using that $L$ is positive recurrent, we obtain
\[\Eb[(1-X(\infty))^2\mid X(0)=x]=1-h(x).\]
Plugging \eqref{auxh0} in the previous identity yields the desired result.
\end{proof}
\begin{proof}[Proof of Proposition \ref{comparison}]
Using Eq. \eqref{fr} in Theorem \ref{thm2.6} for the two models, we obtain, for $n\in\Nb$,
\[(n+1)\rho_{n+1}^{\rm{sel}}=\sigma_\mu \rho_{n}^{\rm{sel}},\quad\textrm{and}\quad
(n+1)\rho_{n+1}^{\rm{env}}= \frac{1}{n}\sum_{j=1}^n \gamma_{n+1,j}\,\rho_j^{\rm{env}}.\]
Multiplying separately these equations with $z^n$, $z\in[0,1]$, and summing over $n\in\Nb$, one obtains
\[(p^{\rm{sel}})'(z)=\rho_1^{\rm{sel}}+\sigma_\mu\, p^{\rm{sel}}(z),\quad\textrm{and}\quad (p^{\rm{env}})'(z)=\rho_1^{\rm{env}}+\\sum_{j=1}^\infty \rho_j^{\rm{env}} g_j(z),\]
with $g_j(z)\coloneqq \sum_{n=j}^{2j-1}\gamma_{n+1,j}\frac{z^n}{n}.$ Solving the ODE for $p^{\rm{sel}}$ via variation of constants, and using that $p^{\rm{sel}}(0)$ and $p^{\rm{sel}}(1)=1$, yields the desired formulas for $\rho_1^{\rm{sel}}$ and $p^{\rm{sel}}$ (see also \cite[Thm. 6.1]{CM19}).
Now, using the definition of the coefficients $\gamma_{n+1,j}$ (defined below \eqref{fr}) followed by a straightforward calculation, one obtains
\[g_j(z) =\sum_{k=1}^{j}\binom{j}{k}\sigma_{j,k} \int_0^z \frac{u^{j-1}-u^{k+j-1}}{1-u} \dd u=\int_0^z\dd u\, \frac{u^{j-1}}{1-u}\int_{(0,1)}\mu(\dd y)\,(1-(1-y(1-u))^j),\]
where we have also used the definition of the coefficients $\sigma_{m,k}$ (see \eqref{smk}). Since $(1-h)^j\geq 1-jh$ for $h\in(0,1)$, we infer that $g_j(z)\leq \sigma_\mu z^{j}$ with equality only if $z=0$ or $j=1$. We conclude that
\[(p^{\rm{env}})'(z)<\rho_1^{\rm{env}}+\sigma_\mu \, p^{\rm{env}}(z), \quad z\in(0,1].\]
Letting $f(z)\coloneqq \rho_1^{\rm{env}}+\sigma_\mu \, p^{\rm{env}}(z)$ we then have $f'(z)/f(z) \leq \sigma_\mu$ so, after integration, $\log(f(z)/f(0)) \leq \sigma_\mu z$. Since $p^{\rm{env}}(0)=0$, this yields
%Using that $p^{\rm{env}}(0)=0$ and comparing with the solution of the ODE $f'(z)=\rho_1^{\rm{env}}+\sigma_\mu \, f(z)$ with $f(0)=0$ (which can be computed explicitly), we get 
\[p^{\rm{env}}(z)\leq \rho_1^{\rm{env}}\left(\frac{e^{\sigma_\mu z}-1}{\sigma_\mu}\right)=\frac{\rho_1^{\rm{env}}}{\rho_1^{\rm{sel}}}\,p^{\rm{sel}}(z).\]
Moreover, since $p^{\rm{env}}(1)=p^{\rm{sel}}(1)=1$, we conclude that $\rho_1^{\rm{env}}\geq\rho_1^{\rm{sel}}$. Assume now that $\rho_1^{\rm{env}}=\rho_1^{\rm{sel}}$. It follows that $p^{\rm{env}}(z)\leq p^{\rm{sel}}(z)$, for $z\in[0,1]$. Hence,
\[1=\int_{0}^1 (p^{\rm{env}})'(z)\dd z<\int_{0}^1 (\rho_1^{\rm{env}}+\sigma_\mu \, p^{\rm{env}}(z))\dd z\leq\int_{0}^1 (\rho_1^{\rm{sel}}+\sigma_\mu \, p^{\rm{sel}}(z))\dd z= \int_{0}^1 (p^{\rm{sel}})'(z)\dd z=1,\]
which is a contradiction. Thus, $\rho_1^{\rm{env}}>\rho_1^{\rm{sel}}$. In particular, for $z\neq 0$ sufficiently small, $p^{\rm{env}}(z)>p^{\rm{sel}}(z)$. Hence, the last statement follows using that $h^{\rm{env}}(z)=1-p^{\rm{env}}(1-z)$ and $h^{\rm{sel}}(z)=1-p^{\rm{sel}}(1-z)$. 
\end{proof}
%%%%%%%%%%%%%%%%%%%%%%%%%%%%%%%%%%%%%%%%%%%%%%%%%%%%%%%%%%%%%%%%%%%%%%%%%%%%%%%%%%%%%%%%%%%%%%%%
%%%%%%%%%%%%%%%%%%%%%%%%%%%%%%%%%%%%%%%%%%%%%%%%%%%%%%%%%%%%%%%%%%%%%%%%%%%%%%%%%%%%%%%%%%%%%%%%
%%%%%%%%%%%%%%%%%%%%%%%%%%%%%%%Section 6%%%%%%%%%%%%%%%%%%%%%%%%%%%%%%%%%%%%%%%%%%%%%%%%%%%%%%%%
%%%%%%%%%%%%%%%%%%%%%%%%%%%%%%%%%%%%%%%%%%%%%%%%%%%%%%%%%%%%%%%%%%%%%%%%%%%%%%%%%%%%%%%%%%%%%%%%
%%%%%%%%%%%%%%%%%%%%%%%%%%%%%%%%%%%%%%%%%%%%%%%%%%%%%%%%%%%%%%%%%%%%%%%%%%%%%%%%%%%%%%%%%%%%%%%%
\section{Quenched results}\label{S6}

\subsection{Quenched results related to Section \ref{s25}}\label{s61}
In this section we prove quenched parts of the results stated in Section \ref{s25}. We start with the proof of the second part of Theorem \ref{thm2.3}, which establishes the quenched moment duality \eqref{quenchedual} for almost every environment $\om$.

\begin{proof}[Proof of Theorem \ref{thm2.3} (Part II: quenched moment duality)]
Since both sides of \eqref{quenchedual} are right-continuous in $T$, it is sufficient to prove that, for any bounded measurable function $g:\Db^\star_T\to\Rb$,
 \begin{equation}\label{condex}
 \Eb[(1-X^J(T))^n g((J_s)_{s\in[0,T]})\mid {X^J(0)=x}]=\Eb[(1-x)^{{R_T^J(T-)}} g((J_s)_{s\in[0,T]})\mid {R_T^J(0-)}=n].
\end{equation}
 Let $\Hs:=\{g:\Db_T^\star \to\Rb:\textrm{ such that $\eqref{condex}$ is satisfied}\}$. Thanks to the annealed moment duality, i.e. Eq. \eqref{md}, every constant function {belongs} to $\Hs$. Moreover, $\Hs$ is closed under increasing limits of non-negative bounded functions in $\Hs$. Now, we claim that \eqref{condex} holds for functions of the form $g(\om)=g_1(\om(t_1))\cdots g_k(\om(t_k))$, with $0<t_1<\cdots<t_k<T$ and $g_i\in\Cs^2([0,\infty))$ with compact support. If the claim is true, then thanks to the monotone class theorem, $\Hs$ would contain any measurable function $g$, which would then achieve the proof.
\smallskip

We prove the claim by induction on $k$. For $k=1$, we need to prove that, for $t_1\in(0,T)$, 
\begin{equation}\label{condexsimple}
 \Eb[(1-X^J(T))^n g_1(J(t_1))\mid X^J(0)=x]=\Eb[(1-x)^{{R_T^J(T-)}} g_1(J({t_1}))\mid {R_T^J(0-)}=n].
\end{equation}
Note first that, using the Markov property for $X^J$ in $[0,t_1]$ followed by {Eq. \eqref{md} in Theorem \ref{thm2.3} and the fact that $t_1$ and $T$ are almost surely continuity times for $J$,} we obtain
\begin{align*}
  &\Eb\left[(1-X^J(T))^n g_1(J(t_1))\mid X^J(0)=x\right]\\
  &= \Eb\left[g_1(J(t_1))\hat{\Eb}\left[(1-\hat{X}^{\hat{J}}(T-t_1))^n\mid \hat{X}^{\hat{J}}(0)=X^J(t_1)\right]\mid X^J(0)=x\right]\\
  &= \Eb\left[g_1(J(t_1))\hat{\Eb}\left[(1-X^J(t_1))^{{{\hat{R}}_{T-t_1}^{\hat{J}}((T-t_1)-)}}\mid {\hat{R}_{T-t_1}^{\hat{J}}(0-)}=n\right]\mid X^J(0)=x\right],
\end{align*}
where the subordinator $\hat{J}$ is defined via $\hat{J}(h)\coloneqq J(t_1+h)-J(t_1)$. The processes {$\hat{X}^{\hat{J}}$ and $\hat{R}_{T-t_1}^{\hat{J}}$} are independent copies of {$X^J$ and $R_{T-t_1}^J$,} which are driven by $\hat{J}$ (which is in turn independent of $(J(u))_{u\in[0,t_1]}$). Using first Fubini's theorem, and then Eq. \eqref{rmd} in Theorem \ref{thm2.3} {and the fact that $0$ and $t_1$ are almost surely continuity times for $J$,} the last expression equals
\begin{align*}
&\hat{\Eb}\left[\Eb\left[g_1(J(t_1))(1-X^J(t_1))^{{{\hat{R}}_{T-t_1}^{\hat{J}}((T-t_1)-)}}\mid X^J(0)=x\right]\mid {{\hat{R}}_{T-t_1}^{\hat{J}}(0-)}=n\right]\\
=&\hat{\Eb}\left[\Eb\left[g_1(J(t_1))(1-x)^{{R_{t_1}^J(t_1-)}}\mid {R_{t_1}^J(0-)}={{\hat{R}}_{T-t_1}^{\hat{J}}((T-t_1)-)}\right]\mid {{\hat{R}}_{T-t_1}^{\hat{J}}(0-)}=n\right].
  \end{align*}
The proof of the claim for $k=1$ is achieved using the Markov property for {$R_T^J$} in the (backward) interval $[0,T-t_1]$.
Let us now assume that the claim is true up to $k-1$. We proceed as before to prove that the claim holds for $k$. Using the Markov property for {$X^J$} in $[0,t_1]$ followed by the inductive step, we obtain
\begin{align*}
 &\Eb\left[(1-X^J(T))^n \prod_{i=1}^k g_i(J(t_i))\mid X^J(0)=x\right]\\
 &=\Eb\left[g_1(J(t_1))\hat{\Eb}\left[(1-{\hat{X}}^{\hat{J}}(T-t_1))^n\, G(J(t_1),\hat{J})\mid {\hat{X}}^{\hat{J}}(0)=X^J(t_1)\right]\mid X^J(0)=x\right]\\
 &=\Eb\left[g_1(J(t_1))\hat{\Eb}\left[(1-X^J(t_1))^{{{\hat{R}}_{T-t_1}^{\hat{J}}((T-t_1)-)}}\,G(J(t_1),\hat{J})\mid {{\hat{R}}_{T-t_1}^{\hat{J}}(0-)}=n\right]\mid X^J(0)=x\right], 
 \end{align*}
where $G(J(t_1),\hat{J})\coloneqq\prod_{i=2}^k g_i(J(t_1)+\hat{J}(t_i-t_1))$. Using Fubini's theorem, {the reinforced  duality Eq. \eqref{rmd}}, {and the fact that $0$ and $t_1$ are almost surely continuity times for $J$,} the last expression equals
 \begin{align*}
 &\hat{\Eb}\left[{\Eb}\left[(1-X^J(t_1))^{{{\hat{R}}_{T-t_1}^{\hat{J}}((T-t_1)-)}}g_1(J(t_1))\,G(J(t_1),\hat{J})\mid X^J(0)=x\right] \mid {{\hat{R}}_{T-t_1}^{\hat{J}}(0-)}=n\right]\\
 =&\hat{\Eb}\left[{\Eb}\left[(1-x)^{{R_{t_1}^J(t_1-)}}g_1(J(t_1))\,G(J(t_1),\hat{J})\mid {R_{t_1}^J(0-)}={{\hat{R}}_{T-t_1}^{\hat{J}}((T-t_1)-)}\right] \mid {{\hat{R}}_{T-t_1}^{\hat{J}}(0-)}=n\right],
\end{align*}
and the proof is achieved using the Markov property for {$R^J_T$} in the (backward) interval $[0,T-t_1]$.
\end{proof} 
\begin{proof}[Proof of Theorem \ref{thm2.4}-(2)(Asymptotic type frequency)]
Let $\om$ be such that the quenched moment duality \eqref{quenchedual} holds between $-\tau$ and $0$. In particular,
\begin{equation}\label{mdt0}
 \mathbb{E} \left [ (1-X^\omega(0))^n|X^\omega(-\tau)=x \right ]= \mathbb{E} \left [ (1-x)^{R_0^\omega(\tau-)}|R_0^\omega(0-)=n \right ].
\end{equation}
Since we assume that $\theta >0$ and $\nu_0, \nu_1 \in (0,1)$, the right hand side converges to $\Pi_n(\omega)$ (defined in \eqref{defpin}), which proves that the moment of order $n$ of $1-X^\omega(0)$ conditionally on $\{ X^\omega(-\tau) = x \}$ converges to $\Pi_n(\omega)$. Since we deal with random variables supported on $[0,1]$, the convergence of the positive integer moments proves the convergence in distribution and the fact that the limit distribution $\mathcal{L}^\om$ satisfies \eqref{dualimite}. 
\smallskip

It remains to prove \eqref{approxwn}. For $\upsilon\in\Ms_1(\mathbb{N}_0^\dagger)$ with finite support, let $\upsilon^\omega_s$ denote the distribution of $R_0^\omega(s-)$ given that $R_0^\om(0-)\sim\upsilon$. 
%Note that as far as $R_0^\omega$ is not absorbed, it will either get absorbed at $\dagger$ at a rate that is greater than or equal to $\theta\nu_0 > 0$, due to the appearance of a mutation of type $0$, or go to $0$ before such a mutation occurs. 
Let $T_{0, \dagger}^\om$ be the absorption time of $R_0^\omega$ at $\{ 0, \dagger \}$. Note that $T_{0, \dagger}^\om$ is stochastically bounded by an exponential random variable with parameter $\theta\nu_0$. Therefore,
\begin{align*}
\upsilon^\omega_\tau(\mathbb{N}) & = \mathbb{P}_{\upsilon} \left ( R_0^\omega(\tau-) \in \mathbb{N} \right ) = \mathbb{P}_{\upsilon} \left ( T_{0, \dagger}^\om > \tau \right ) \leq e^{- \theta\nu_0 \tau}. \label{binomialisation}
\end{align*}
Hence,
%when $R^0_{0+}(\omega) \sim \mu$ 
we have 
\begin{align*}
\mathbb{P}_{\upsilon} \left ( \exists s\geq 0 \ \text{s.t.} \ R_0^\omega(s)=0 \right ) & = \upsilon_{\tau}^\omega(\{0\}) + \sum_{k \geq 1} \mathbb{P}_{\upsilon} \left ( R_0^\omega(\tau-) = k \ \text{and} \ \exists s\geq \tau \ \text{s.t.} \ R_0^\omega(s)=0 \right ) \\
& \leq \upsilon_{\tau}^\omega(\{0\}) + \upsilon_\tau^\omega(\mathbb{N}) \leq \upsilon_\tau^\omega(\{0\}) + e^{-\theta \nu_0 \tau}. 
\end{align*}
Thus, we obtain  
%and when $R^0_{0+}(\omega) \sim \mu$, 
\begin{eqnarray}
 \upsilon_{\tau}^\omega(\{0\}) \leq \mathbb{P}_{\upsilon} \left ( \exists s\geq 0 \ \text{s.t.} \ R_0^\omega(s)=0 \right ) \leq \upsilon_{\tau}^\omega(\{0\}) + e^{-\theta\nu_0 \tau}. \label{absetmu0}
\end{eqnarray}
Similarly, we have 
\begin{align*}
\mathbb{E}_{\upsilon} \left [ (1-x)^{R_0^\omega(\tau-)} \right ] & = \upsilon_{\tau}^\omega(\{0\}) + \sum_{k \geq 1} (1-x)^k \upsilon_{\tau}^\omega(\{k\}) \leq \upsilon_{\tau}^\omega(\{0\}) + \upsilon_\tau^\omega(\mathbb{N}) \leq  \upsilon_{\tau}^\omega(\{0\}) + e^{-\theta\nu_0 \tau}.
\end{align*}
Hence,
\begin{eqnarray}
\upsilon_{\tau}^\omega(\{0\})\leq \mathbb{E}_{\upsilon} [ (1-x)^{R_0^\omega(\tau-)} ] \leq \upsilon_{\tau}^\omega(\{0\}) + e^{-\theta\nu_0 \tau}. \label{absetmu00}
\end{eqnarray}
Recall from Section \ref{s25} that $\Pi_n(\omega)\coloneqq \mathbb{P}(\exists s\geq 0 \ \text{s.t.} \ R_0^\omega(s)=0 \mid R_0^\omega(0-)=n)$. Choosing {$\upsilon = \delta_n$} in \eqref{absetmu0} and in \eqref{absetmu00} and subtracting both inequalities, we get 
\[ \left | \mathbb{E} \left [ (1-x)^{R_0^\omega(\tau-)}\mid R_0^\omega(0-)=n \right ] - \Pi_n(\omega) \right | \leq e^{-\theta\nu_0 \tau}. \]
This inequality together with \eqref{quenchedual} (i.e. the quenched moment duality) yields the desired result. 
\end{proof}
\begin{proof}[Proof of Proposition \ref{prop2.5}]
Let $\om\in\Db^\star$ such that \eqref{quenchedual} holds. Let $J_\om\coloneqq J\otimes_{\tau_\star} \om$. Consider the process $X^{J_\om}$ in $[-\tau,0]$ with $\tau>{\tau_\star^{}}$. Using the Markov property, we obtain 
\[\Eb\!\left[(1-X^{J_\om}(0))^n\mid X^{J_\om}(-\tau)=x\right]=\!\int_0^1\!\!\Eb\!\left[(1-X^{\om}(0))^n\mid X^{\om}(-{\tau_\star^{}})=y\right]\Pb(X(-{\tau_\star^{}})\in \dd y\mid X(-\tau)=x),\]
where $X$ is the solution of \eqref{WFSDE} with subordinator $J$. Combining the previous identity with \eqref{quenchedual} for $X^\om$ in $(-\tau_\star,0)$, and using translation invariance of $X$, we obtain
\[\Eb\!\left[(1-X^{J_\om}(0))^n\mid X^{J_\om}(-\tau)=x\right]=\!\int_0^1\!\!\Eb\!\left[(1-y)^{R_0^\om(\tau_\star^{} -)}\mid R_0^{\om}(0-)=n\right]\Pb(X(\tau-\tau_\star)\in \dd y\mid X(0)=x).\]
Hence, letting $\tau\to\infty$ and using Theorem \ref{thm2.4}-(1), we get
\[\lim_{\tau\to\infty}\Eb\left[(1-X^{J_\om}(0))^n\mid X^{J_\om}(-\tau)=x\right]=\!\int_0^1\!\!\Eb\left[(1-y)^{R_0^\om(\tau_\star^{} -)}\mid R_0^{\om}(0-)=n\right]\Pb(X(\infty)\in \dd y),\]
and the result follows from Eq. \eqref{cvmomentsannealed} in Theorem \ref{thm2.4}-(1).
\end{proof}

\subsection{Quenched results related to Section \ref{s26}}\label{s62}
This section is devoted to the proof of Theorem \ref{thm2.6}-(2), which describes the asymptotic behavior of the ancestral type distribution. 
\begin{lemma}\label{hTq}
For all $T \geq 0$, $x\in[0,1]$ and $\om\in\Db^\star$, we have
\begin{align}
 h^{\omega}_T(x)&=1-\mathbb{E}[(1-x)^{L_T^\omega(T-)}\mid L_T^\omega(0-)=1].\label{pldasgatdq}
 \end{align}
\end{lemma}
\begin{proof}
The proof is analogous to the proof of Lemma \ref{exprhtannealed}.
\end{proof}
Now, we proceed to prove Theorem \ref{thm2.6}-(2).

\begin{proof}[Proof of Theorem \ref{thm2.6}(Ancestral type distribution)-(2)]
Recall that by assumption $\theta\nu_0>0$. For $\mu\in\Ms_1(\mathbb{N})$, we denote by $\mu^\omega_T(\beta)$ the distribution of $L_T^\omega(\beta-)$ given that $L_T^\omega(0-)\sim\mu$. 
Let $t > s > 0$. Note that we have $\mu^{\omega}_{t}(t) = (\mu^{\omega}_{t}(t - s))^{\omega}_{s}(s)$ so 
\begin{eqnarray}
d_{TV} (\mu^{\omega}_{t}(t),\mu^{\omega}_{s}(s)) = d_{TV} ( (\mu^{\omega}_{t}(t - s))^{\omega}_{s}(s), \mu^{\omega}_{s}(s)), \label{shifchanges startinglaw}
\end{eqnarray}
where $d_{TV}(\mu_1,\mu_2)$ stands for the total variation distance between $\mu_1$ and $\mu_2$.
\smallskip

Assume now that $L_T^\omega(0-)\sim\mu$. By construction, $L_T^\om$ jumps from any state $i$ to the state $1$ with rate $q^0(i,1)\geq \theta\nu_0 > 0$ {(see \eqref{kratespldasg})}. Let $\hat{L}_T^\om$ be a process with initial distribution $\mu$, evolving as $L_T^\om$, but jumping from $i$ to $1$ at rate $q^0(i,1) - \theta\nu_0 \geq 0$. {We decompose the dynamic of $L_T^\om$} as follows: (1) $L_T^\omega$ evolve as $\hat{L}_T^\om$ on $[0, \xi]$, where $\xi$ is an independent exponential random variable with parameter $\theta\nu_0$, (2) at time $\xi$, $L_T^\om$ jumps to the state $1$ {regardless of its current position}, and (3) conditionally on $\xi$, $L_T^\om$ has the same law on $[\xi, \infty)$ as an independent copy of $L_{T-\xi}^\omega$ started with one line. This idea allows us to couple $L_T^\om$ to a copy of it $\tilde{L}_T^\omega$ with starting law $\tilde \mu$, so that the two processes are equal on $[\xi, \infty)$. Since $L_T^\omega(T-) \sim \mu^{\omega}_T(T)$ and $\tilde{L}_T^\omega(T-) \sim \tilde\mu^{T}_T(\omega)$, we have 
\begin{eqnarray}
d_{TV} (\mu^{\omega}_T(T), \tilde \mu^{\omega}_T(T)) \leq \mathbb{P} \left ( \tilde{L}_T^\omega(T-) \neq {L}_T^\omega(T-) \right ) \leq \mathbb{P} \left ( \xi > T \right ) = e^{-\theta\nu_0  T}. \label{rapprochementexpo}
\end{eqnarray}
This together with \eqref{shifchanges startinglaw}, implies that, for any $\mu\in\Ms_1(\Nb)$ and any $t > s > 0$, 
\begin{eqnarray}
\ d_{TV} (\mu^{\omega}_{t}(t), \mu^{\omega}_{s}(s)) \leq e^{-\theta\nu_0 s}. \label{cauchyexpo}
\end{eqnarray}
In particular $(\mu^{\omega}_{t}(t))_{t > 0}$ is Cauchy as $t \rightarrow \infty$ for the total-variation distance. Therefore, $(\mu^{\omega}_{t}(t))_{t > 0}$ has a limit $\mu^{\omega}\in\Ms_1(\Nb)$. Moreover, \eqref{rapprochementexpo} implies that $\mu^\omega$ does not depend on $\mu$, and the first {part} of Theorem \ref{thm2.6}-(2) is proved. Identity \eqref{pldasgatdtpsinftyq} follows then by Lemma \ref{hTq}.
\smallskip

Setting $s = T$ and letting $t\to\infty$ in \eqref{cauchyexpo} yields 
\[d_{TV} (\mu^\omega, \mu^{\omega}_{T}(T)) \leq e^{-\theta\nu_0 T}.\] 
Since $h^{\omega}(x) = 1 - \mathbb{E} \left[(1-x)^{Z_{\infty}^\om} \right]$ and $h^{\omega}_T(x) = 1 - \mathbb{E} \left[(1-x)^{Z_{T}^\om} \right]$, where $Z_{\infty}^\om \sim (\delta_1)^{\omega}$ and $Z_{T}^\om \sim (\delta_1)^{\om}_{T}(T)$, we get
\[\lvert h_T^\om(x)-h^\om(x)\rvert\leq d_{TV}((\delta_1)^{\omega}, (\delta_1)^{\om}_{T}(T))\leq e^{-\theta\nu_0 T},\]
achieving the proof. 
\end{proof}

\section{Further quenched results for simple environments}\label{S7}
In this section we provide, for simple environments, extensions and refinements of the results obtained in Sections \ref{s25} and \ref{s26} in the quenched setting. Recall the quenched diffusion $X^\omega$ defined in Section \ref{s23}. 
\subsection{Extensions of quenched results in Section \ref{s25} }\label{s71}
First, we extend the main quenched results in Section \ref{s25}, which hold for almost every environment, to any simple environment. 
\begin{theorem}[Quenched moment duality for simple environments]\label{thmf1}
 The quenched moment duality \eqref{quenchedual} holds for any simple environment.
\end{theorem}
The proof of Theorem \ref{thmf1} has two main ingredients: a moment duality between the jumps of the environment, and a moment duality at the jumps. These results are covered by the next two lemmas. 
\begin{lemma}[Quenched moment duality between the jumps] \label{momdualbetween} 
Let $0\leq s<t\leq T$ and assume that $\om$ has no jumps in $(s,t)$. For all $x\in[0,1]$ and $n\in \mathbb{N}$, we have 
\[ \mathbb{E} \left [ (1-X^\omega(t-))^n \mid X^\omega(s)=x \right ]= \mathbb{E} \left [ (1-x)^{R_T^\omega((T-s)-)}\mid R_{T}^\omega(T-t)=n \right ]. \]
%Recall that if $\om$ has no jump at $t$, then $X^\omega(t-)=X^\omega(t)$ and $ R_{t}^\omega(0) = R_{t}^\omega(0-)$. 
\end{lemma}
\begin{proof}
In $(s,t)$, the processes $X^\omega$ and $R_t^\omega$ evolve as in the annealed case with $\mu=0$. Therefore, the result follows applying Theorem \ref{thm2.3} with $\mu = 0$. 
\end{proof}
\begin{lemma}[Quenched moment duality at jumps] \label{momdualatjumps} 
Assume that $\om\in\Db^\star$ is simple and has a jump at time $t<T$. Then, for all $x\in[0,1]$ and $n\in \mathbb{N}$, we have 
\[ \mathbb{E} \left [ (1-X^\omega(t))^n \mid X^\omega(t-)=x \right ]= \mathbb{E}\left [ (1-x)^{R_{T}^\omega(T-t)}\mid R_{T}^\omega((T-t)-)=n \right ]. \]
\end{lemma}
\begin{proof}

On the one hand, since $X^\omega(t) = X^\omega(t-) + X^\omega(t-)(1-X^\omega(t-))\Delta \omega(t)$ almost surely, we have 
\begin{equation}\label{jumponx}
\mathbb{E}\left [ (1-X^\omega(t))^n\mid X^\omega(t-)=x \right ] = \left [ 1 - x (1 + (1-x) \Delta \omega(t)) \right ]^n  = \left [(1-x)(1-x\Delta\om(t))\right]^n.
\end{equation}
{On the} other hand, conditionally on $\{R_{T}^\omega((T-t)-)=n\}$, we have $R_{T}^\omega(T-t) \sim n + Y$ where $Y \sim \bindist{n}{\Delta \omega(t)}$. Therefore 
\begin{align}
\mathbb{E} \left [ (1-x)^{R_{T}^\omega(T-t)}\mid R_{T}^\omega((T-t)-)=n \right ] & = \mathbb{E}\left [ (1-x)^{n+Y} \right ]  = (1-x)^n \left [ 1-\Delta \omega(t) + \Delta \omega(t)(1-x) \right ]^n \nonumber \\
& =\left [(1-x)(1-x\Delta\om(t))\right]^n. \label{jumponr}
\end{align}
The combination of \eqref{jumponx} and \eqref{jumponr} yields the result. 
\end{proof}
\begin{proof}[Proof of Theorem \ref{thmf1}]
Let $\om$ be a simple environment. Let $(t_i)_{i=1}^m$ be the increasing sequence of jump times of $\omega$ in $[0,T]$. Without loss of generality we assume that $0$ and $T$ are both jump times of $\omega$. In particular, $t_1 = 0$ and $t_m=T$. Let $(X^\omega( s))_{ s \in[0,T]}$ and $(R_T^\omega(\beta))_{\beta\in[0,T]}$ be independent realizations of the Wright--Fisher process and the line-counting process of the k-ASG, respectively. For $s\in[0,T]$, denote by $\mu_s^\om(x,\cdot)$ and $\bar{\mu}_s^\om(x,\cdot)$ the law of $X^\om(s)$ and $X^\om(s-)$, respectively, given that $X^\om(0)=x$. Partitioning with respect to the values of $X^\omega(t_m-)$ and using Lemma \ref{momdualatjumps} at $t=T$, we get 
\begin{align*}
\mathbb{E} \left [ (1-X^\omega(T))^n\mid X^\omega(0)=x \right ] 
&=  \int_0^1 \mathbb{E} \left [ (1-X^\omega(t_m))^{n}\mid X^\omega(t_m-)=y \right ] \bar{\mu}_{t_m}^\om(x,\dd y)\\ 
&=  \int_0^1 \mathbb{E} \left [ (1-y)^{R_T^\omega(0)}\mid R_{T}^\omega(0-) = n \right ]   \bar{\mu}_{t_m}^\om(x,\dd y) \\ 
&= \mathbb{E} \left [ (1-X^\omega(t_m-))^{R_T^\omega(0)}\mid R_T^\omega(0-)=n, X^\omega(0)=x \right ]\eqdef \,I_T^\om(x,n). 
\end{align*}
Set, for $t<T$, $q_{n,k}^\om(T,t)\defeq\mathbb{P} \left ( R_T^\omega(t) = k \mid R_T^\omega(0-)=n \right ) $. Partitioning with respect to the values of $X^\omega(t_{m-1})$ and $R_T^\omega(0)$, and using Lemma \ref{momdualbetween}, we get 
\begin{align*}
I_T^\om(x,n)=& \sum_{k \in \mathbb{N}_0^\dagger} \mathbb{E} \left [ (1-X^\omega(t_m-))^{k}\mid X^\omega(0)=x \right ] q_{n,k}^\om(T,0) \\
= & \sum_{k \in \mathbb{N}_0^\dagger}q_{n,k}^\om(T,0) \int_0^1 \mathbb{E}\left [ (1-X^\omega(t_m-))^k \mid X^\omega(t_{m-1})=y \right ]   {\mu}_{t_{m-1}}^\om(x,\dd y) \\
= & \sum_{k \in \mathbb{N}_0^\dagger}q_{n,k}^\om(T,0) \int_0^1 \mathbb{E} \left [ (1-y)^{R_T^\omega((T-t_{m-1})-)}\mid R_T^\omega(0)=k \right ]  \mu_{t_{m-1}}^\om(x,\dd y) \\
= & \mathbb{E} \left [ (1-X^\omega(t_{m-1}))^{R_T^\omega((T-t_{m-1}^{})-)}\mid R_T^\omega(0-)=n, X^\omega(0)=x \right ]. 
\end{align*}
If $m=2$ the proof of \eqref{quenchedual} is already complete. If $m > 2$, we continue as follows. Partitioning with respect to the values of $R_T^\omega((T-t_{m-1})-)$ and of $X^\omega(t_{m-1}^{}-)$, and using Lemma \ref{momdualatjumps}, we obtain 
\begin{align*}
&I_T^\om(x,n)= \sum_{k \in \mathbb{N}_0^\dagger} \mathbb{E} \left [ (1-X^\omega(t_{m-1}))^{k}\mid X^\omega(0)=x \right ] q_{n,k}^\om(T,(T-t_{m-1})-) \\
= & \sum_{k \in \Nb_0^\dagger} q_{n,k}^\om(T,(T-t_{m-1})-)\int_0^1 \mathbb{E} \left [ (1-X^\omega(t_{m-1}))^{k}\mid X^\omega(t_{m-1}-)=y \right] \bar{\mu}_{t_{m-1}}^\om(x,\dd y) \\ 
= & \sum_{k \in \mathbb{N}_0^\dagger}q_{n,k}^\om(T,(T-t_{m-1})-) \int_0^1 \mathbb{E}\left [ (1-y)^{R_T^\omega(T-t_{m-1})}\mid R_{T}^\omega((T-t_{m-1})-) = k \right ] \bar{\mu}_{t_{m-1}}^\om(x,\dd y) \\ 
=&  \mathbb{E}\left [ (1-X^\omega(t_{m-1}-))^{R_T^\omega(T-t_{m-1})}\mid R_T^\omega(0-)=n, X^\omega(0)=x \right ]. 
\end{align*}
Iterating this procedure, using successively Lemma \ref{momdualbetween} and Lemma \ref{momdualatjumps} (the first one is applied on the intervals {$(t_{i-1}, t_i)$}, while the second one is applied at the times {$t_{i}$}), we finally obtain
\[\mathbb{E}[ (1-X^\omega(T))^n\mid X^\omega(0)=x ]= \mathbb{E} \left [ (1-x)^{R_T^\omega(T-)}|R_T^\omega(0-)=n \right ], \]
which ends the proof.
\end{proof}
\begin{theorem}[Quenched asymptotic type frequency for simple environments]\label{thmf2}
The statement of Theorem \ref{thm2.4}-(2) holds for any simple environment.
\end{theorem}
\begin{proof}
Analogous to the proof of Theorem \ref{thm2.4}-(2), but using Theorem \ref{thmf1} instead of Theorem \ref{thm2.3}. 
\end{proof}
\textbf{Refinements for $\sigma=0$.} Under this additional assumption, we provide a more explicit expression of $\Pi_n(\om)$ (defined in \eqref{defpin}). This is possible thanks to the following explicit diagonalization of the matrix $Q_\dagger^0$ (the transition matrix of the process $R$ under the null environment).
\begin{lemma}\label{diag}
 Assume that $\sigma=0$ and set, for $k\in\Nb_0^\dagger$, $\lambda_k^\dagger\coloneqq -q_\dagger^0(k,k)$, and, for $k\in\Nb$, $\gamma_k^\dagger\coloneqq q_\dagger^0(k,k-1)$, {where $q_\dagger^\mu(\cdot,\cdot)$ is defined in \eqref{krates}}. In addition, let 
 \begin{itemize}
  \item[(i)] $D_\dagger$ be the diagonal matrix with diagonal entries $(-\lambda_i^\dagger)_{i\in\Nb_0^\dagger}$, 
  \item[(ii)] $U_\dagger\coloneqq (u_{i,j}^\dagger)_{i,j\in\Nb_0^\dagger}$, where $u_{\dagger,\dagger}^\dagger \coloneqq  1$ and $u_{\dagger,j}^\dagger\coloneqq  0$ for $j\in\Nb_0$, and, for $i\in\Nb_0$
\begin{eqnarray}
u_{i,j}^\dagger \coloneqq  \prod_{l=j+1}^{i} \left ( \frac{\gamma_{l}^\dagger}{\lambda_l^\dagger - \lambda_j^\dagger} \right ) \,\textrm{for } j\in[i]_0, \ u_{i,j}^\dagger\coloneqq 0,\, \textrm{for } j> i\ \textrm{and} \ \ u_{i,\dagger}^\dagger \coloneqq  \theta \nu_0 \sum_{k = 1}^{i} \frac{k}{\lambda_{k}^\dagger} \prod_{l=k+1}^{i} \frac{\gamma_{l}^\dagger}{\lambda_l^\dagger}, \label{recrelalphani3}
\end{eqnarray}
  \item[(iii)] $V_\dagger\coloneqq (v_{i,j}^\dagger)_{i,j\in\Nb_0^\dagger}$, where $v_{\dagger,\dagger}^\dagger \coloneqq  1$ and $v_{\dagger,j}^\dagger\coloneqq  0$ for $j\in\Nb_0$, and, for $i\in\Nb$
\begin{eqnarray}
v_{i,j}^\dagger \coloneqq   \prod_{l=j}^{i-1} \left ( \frac{-\gamma_{l+1}^\dagger}{\lambda_i^\dagger - \lambda_l^\dagger} \right ) \,\textrm{for } j\in[i]_0, \ v_{i,j}^\dagger\coloneqq 0,\, \textrm{for } j> i\ \textrm{and} \ \ v_{i,\dagger}^\dagger \coloneqq  \frac{- \theta \nu_0}{ \lambda_i^\dagger} \sum_{k = 1}^{i} k \prod_{l=k}^{i-1} \left ( \frac{- \gamma_{l+1}^\dagger}{\lambda_i^\dagger - \lambda_l^\dagger} \right ), \label{recrelalphani2}
\end{eqnarray} 
 \end{itemize}
with the convention that an empty sum equals $0$ and an empty product equals $1$. Then, we have 
\[Q_\dagger^0=U_\dagger D_\dagger V_\dagger\quad \textrm{and}\quad U_\dagger V_\dagger=V_\dagger U_\dagger=Id.\]
\end{lemma}
\begin{proof}[Proof of Lemma \ref{diag}]
For any $i\in\Nb_0^\dagger$, let $e_i\coloneqq (e_{i,j})_{j\in\Nb_0^\dagger}$ be the vector defined via $e_{i,i}\coloneqq1$ and $e_{i,j}\coloneqq0$ for $j\neq i$. Order $\Nb_0^\dagger$ as $\{ \dagger, 0, 1, 2,... \}$, so that the matrix $(Q_\dagger^0)^{\top}$ is upper triangular with diagonal elements $(-\lambda_{\dagger}, -\lambda_{0}, -\lambda_{1}, -\lambda_{2}, \ldots)$. For $n\in\Nb_0^\dagger$, let $v_n \in \textrm{Span} \{ e_i: i \in[n]_0\cup\{\dagger\} \}$ be the eigenvector of $(Q_\dagger^0)^{\top}$ associated with the eigenvalue $-\lambda_n$ normalized so that its coordinate with respect to $e_n$ is $1$. It is not difficult to see that these eigenvectors exist and that we have $v_{\dagger} = e_{\dagger}$ and $v_{0} = e_{0}$. For $n \geq 1$, writing $v_n = c_\dagger e_\dagger + c_{0} e_{0} +  ... + c_{n-1} e_{n-1} + e_{n}$ and multiplying by $\frac1{-\lambda_n} (Q_\dagger^0)^{\top}$ on both sides, we obtain another expression of $v_n$ as a linear combination of $e_\dagger, e_{0},\ldots, e_{n-1}, e_{n}$. Identifying both expressions, we obtain that $c_k=v_{n,k}^\dagger$, for $k\leq n-1$. In particular, we have 
\[ v_n = v_{n,\dagger}^\dagger e_{\dagger}+ v_{n,0}^\dagger e_0 +\cdots+ v_{n,n-1}^\dagger e_{n-1} + v_{n,n}^\dagger e_{n}.\]
Proceeding in a similar way, one obtains that
\[ e_n = u_{n,\dagger}^\dagger v_\dagger + u_{n,0}^\dagger v_{0} + \cdots + u_{n,n-1}^\dagger v_{n-1} + u_{n,n}^\dagger v_{n}. \]
We thus get that $V_\dagger^{\top} U_\dagger^{\top} = U_\dagger^{\top} V_\dagger^{\top}=Id$ and $(Q_\dagger^0)^{\top} = V_\dagger^{\top} D_\dagger U_\dagger^{\top}$  (the matrix products are well-defined, because they involve sums of finitely many non-zero terms). This ends the proof. 
\end{proof}

Now, consider the polynomials $S_k^\dagger$, $k \in\Nb_0$, defined via
\begin{eqnarray}
S_k^\dagger(x) \coloneqq  \sum_{i=0}^{k} v_{k,i}^\dagger\, x^i,\quad x\in[0,1]. \label{newbasis9}
\end{eqnarray}
In addition, for $z\in(0,1)$, we define the matrices $\Bs(z)\coloneqq (\Bs_{i,j}(z))_{i,j\in\Nb_0^\dagger}$ and $\Phi^\dagger(z)\coloneqq (\Phi^\dagger_{i,j}(z))_{i,j\in\Nb_0^\dagger}$ via
\begin{equation} \label{defbetaki}
\Bs_{i,j}(z)\coloneqq \left\{\begin{array}{ll}
            \Pb(i+B_i(z)=j)&\textrm{for $i,j\in\Nb,$}\\
            1&\textrm{for $i=j\in\{0,\dagger\},
            $}\\
            0&\textrm{otherwise},
            \end{array}\right.\quad\textrm{and}\quad
            {\Phi^\dagger(z)\coloneqq} U_\dagger^\top \Bs(z)^\top V_\dagger^\top,
\end{equation}
where $B_i(z)\sim\bindist{i}{z}$. We will see in the proof of Theorem \ref{thmf3} that $\Phi^\dagger(z)$ is well-defined.
\begin{theorem} \label{thmf3}
Assume that $\sigma=0$, $\theta >0$ and $\nu_0, \nu_1 \in (0,1)$. Let $\om$ be a simple environment. Denote by $N\coloneqq N(\tau)$ the number of jumps of $\om$ in $(-\tau,0)$ and let $(T_i)_{i=1}^N$ be the sequence of the jump times in decreasing order, and set $T_0 \coloneqq  0$. For any $m\in[N]$, define the matrix $A_m^{\dagger}(\omega)\coloneqq (A_{i,j}^{\dagger,m}(\omega))_{i,j\in\Nb_0^\dagger}$ via 
\begin{eqnarray} 
A^{\dagger}_m(\omega) \coloneqq  \Phi^\dagger(\Delta \omega(T_m)) \exp \left ( (T_{m-1} - T_m) D_\dagger \right ). \label{defmatA}
\end{eqnarray}
Then, for all $x \in (0,1)$ and $n \in \mathbb{N}$, we have 
\begin{eqnarray} 
\mathbb{E}\left [ (1-X^\omega(0))^n\mid X^\omega(-\tau)=x \right ] = \sum_{k = 0}^{n 2^N} C^\dagger_{n,k}(\om,\tau) S_k^\dagger(1-x), \label{exprmomtpst}
\end{eqnarray}
where the matrix $C^\dagger(\om,\tau)\coloneqq (C^\dagger_{n,k}(\om,\tau))_{k,n\in\Nb_0^\dagger}$ is given by
\begin{equation} 
C^\dagger(\om,\tau)\coloneqq U_\dagger \left[A_N^\dagger(\om)A_{N-1}^\dagger(\om)\cdots A_1^\dagger(\om)\right]^\top \exp \left ((T_N+\tau)D_\dagger \right ), \label{defcoeffmom}
\end{equation} 
with the convention that an empty product of matrices is the identity matrix. Moreover, for all $n\in\Nb$,
\begin{equation} 
\ \Pi_n(\omega) = C_{n,0}^\dagger(\omega,\infty)\coloneqq \lim_{\tau\to\infty} C_{n,0}^\dagger(\om,\tau)=\lim_{\tau \to \infty} \left(U_\dagger \left[A_{N(\tau)}^\dagger(\om)A_{{N(\tau)}-1}^\dagger(\om)\cdots A_1^\dagger(\om)\right]^\top\right)_{n,0}, \label{exprfctgenlt4infmom}
\end{equation}
where the previous limits are well-defined. 
\end{theorem}
% \begin{remark}
% It will be justified in the proof of Theorem \ref{x0condpastenvII} that the matrix products appearing in \eqref{defmatA} and \eqref{defcoeffmom} are well-defined.
% \end{remark}

\begin{proof}
Let us first show that the matrix products in \eqref{defbetaki}, \eqref{defmatA} and \eqref{defcoeffmom} are well-defined and that $C^\dagger_{n,k}(\om,\tau) = 0$ for all $k > n2^N$. To this end, order $\Nb_0^\dagger$ as $\{ \dagger, 0, 1, 2,... \}$, so that the matrices $U_\dagger^\top$ and $V_\dagger^\top$ are upper triangular. Note also that $\Bs_{j,i}(z)=0$ for $i>2j$. Therefore, for any $n \in\Nb$ and any $v= (v_i)_{i \in \Nb_0^\dagger}$ such that $v_i = 0$ for all $i > n$, the vector $\tilde v \coloneqq  U_\dagger^\top (\Bs(z)^\top (V_\dagger^\top v))$ is well-defined and satisfies $\tilde v_i = 0$ for all $i > 2n$. It follows that the matrix $\Phi^\dagger(z)$ in \eqref{defbetaki} is well-defined. Moreover, since $\exp ( (T_{m-1} - T_m) D_\dagger )$ is diagonal, the product defining the matrix $A^{\dagger}_m(\omega)$ in \eqref{defmatA} is also well-defined. Furthermore, for any $n \in\Nb$ and any vector $v = (v_i)_{i \in \Nb_0^\dagger}$ such that $v_i = 0$ for all $i > n$, the vector $\tilde v \coloneqq  A^{\dagger}_m(\omega) v$ satisfies $\tilde v_i = 0$ for all $i > 2n$. In particular, for any $m \geq 1$, the product $\exp ( -(T_N+\tau)D_\dagger) A_m^\dagger(\om)A_{m-1}^\dagger(\om)\cdots A_1^\dagger(\om) U_\dagger^\top$ is well-defined. Additionally, for $n \geq 1$ and a vector $v = (v_i)_{i \in \Nb_0^\dagger}$ such that $v_i = 0$ for all $i > n$, the vector $\tilde v \coloneqq  \exp ( -(T_N+\tau)D_\dagger) A_m^\dagger(\om)A_{m-1}^\dagger(\om)\cdots A_1^\dagger(\om) U_\dagger^\top v$ satisfies $\tilde v_i = 0$ for all $i > 2^m n$. Transposing, we see that the matrix $C^\dagger(\om,\tau)$ in \eqref{defcoeffmom} is well-defined and satisfies $C^\dagger_{n,k}(\om,\tau) = 0$ for all $k > n2^N$.
\smallskip

Define, for $s>0$, the stochastic matrix $\Ps^\dagger_s(\om)\coloneqq (p_{i,j}^\dagger(\om,s))_{i,j\in\Nb_0^\dagger}$ via 
 \[p_{i,j}^\dagger(\om,s)\coloneqq \Pb(R_0^\om(s-)=j\mid R_0^\om(0-)=i).\]
 Hence, defining $\rho(y)\coloneqq (y^i)_{i\in\Nb_0^\dagger}$, $y\in[0,1]$ (with the convention $y^\dagger\coloneqq 0$), we obtain
 \begin{equation}\label{gfrt}
   \mathbb{E}[y^{R_0^\omega(\tau-)} \mid R_0^\omega(0-)=n] = (\Ps_\tau^\dagger(\omega) \rho(y))_n=(\Ps_\tau^\dagger(\om)U_\dagger {S_\dagger(y)})_n,
 \end{equation}
where we used that $\rho(y)=U_\dagger {S_\dagger(y)}$ with ${S_\dagger(y)\coloneqq}(S_k^\dagger(y))_{k\in\Nb_0^\dagger}$.
Thus, Theorem \ref{thmf1} and Eq. \eqref{gfrt} yield 
\begin{equation}\label{mxt}
\mathbb{E} [ (1-X^\omega(0))^n\mid X^\omega(-\tau)=x ] =  \sum_{k = 0}^{\infty} \left(\Ps_\tau^\dagger(\om)U_\dagger\right)_{n,k} S_k^\dagger(1-x).
\end{equation}
Now, consider the semi-group $M_\dagger\coloneqq (M_\dagger(s))_{s\geq 0}$ of the line-counting process of the k-ASG in the null environment, which is defined via $M_{\dagger}(s)\coloneqq \exp (sQ_\dagger^0 )$. Thanks to Lemma \ref{diag}, $M_\dagger(\beta)=U_\dagger E_\dagger(\beta)V_\dagger$, where $E_\dagger(\beta)$ is the diagonal matrix with diagonal entries $(e^{-\lambda^{\dagger}_j \beta})_{j\in\Nb_0^\dagger}$. 
\smallskip

Assume first that $N(\tau)=0$ (i.e. $\om$ has no jumps in $[-\tau,0]$). In this case, we have
\[\Ps_\tau^\dagger(\om)U_\dagger=M_\dagger(\tau)U_\dagger=U_\dagger E_\dagger(\tau)V_\dagger U_\dagger=U_\dagger E_\dagger(\tau)=C^\dagger(\om,\tau),\]  
where we used that $V_\dagger U_\dagger=Id$. Hence, \eqref{exprmomtpst} follows from \eqref{mxt}.
\smallskip

Assume now that $N(\tau)\geq 1$ (i.e. $\om$ has at least one jump in $[-\tau,0]$). Disintegrating with respect to the values of $R_0^\om((-T_i)-)$ and $R_0^\om(-T_i)$, $i\in [N]$, we get
 \begin{equation}\label{ptdec}
  \Ps_\tau^\dagger(\om)=M_\dagger(-T_1)\Bs(\Delta\om (T_1))M_\dagger(T_1-T_2)\Bs(\Delta\om(T_2))\cdots \Bs(\Delta\om(T_N))M_\dagger(T_N+\tau).
 \end{equation}
Using this, the relation $M_\dagger(\beta)=U_\dagger E_\dagger(\beta)V_\dagger$, the definition of the matrices $\Phi^\dagger$ and $A_i^\dagger$ (see \eqref{defbetaki} and \eqref{defmatA}), and the fact that {$V_\dagger U_\dagger=Id$}, we obtain
\begin{align}
\Ps_\tau^\dagger(\om)U_\dagger&=U_\dagger E_\dagger(-T_1) \Phi^\dagger(\Delta \om(T_1))^\top E_\dagger(T_1-T_2)\Phi^\dagger(\Delta \om(T_2))^\top\cdots\Phi^\dagger(\Delta \om(T_N))^\top E_\dagger(T_N+\tau) \nonumber \\
 &= U_\dagger A_1^\dagger(\om)^\top A_2^\dagger(\om)^\top\cdots A_N^\dagger(\om)^\top E_\dagger(T_N+\tau) \nonumber \\
 &= U_\dagger \left[A_N^\dagger(\om)A_{N-1}^\dagger(\om)\cdots A_1^\dagger(\om)\right]^\top E_\dagger(T_N+\tau)=C^\dagger(\om,\tau), \label{ptut=ct}
\end{align}
which proves \eqref{exprmomtpst} also in this case. 
\smallskip

It remains to prove that $C^\dagger_{n,0}(\om,\tau)$ converges to $\Pi_n(\om)$ as $\tau\to\infty$. For $\om=\ze$ (i.e. the null environment), {on one hand \eqref{defcoeffmom} yields $C_{n,0}^\dagger(\om,\tau)=e^{-\lambda_0^\dagger \tau} u_{n,0}^\dagger$ and on the other hand \eqref{gfrt} together with $M_\dagger(\beta)=U_\dagger E_\dagger(\beta)V_\dagger$ and $V_\dagger U_\dagger=Id$} yields
\[\mathbb{E}[y^{R_0^\ze(\tau-)} \mid R_0^\ze(0-)=n]=\sum_{k=0}^n e^{-\lambda_k^\dagger \tau} u_{n,k}^\dagger S_k^\dagger(y).\]
Since $\lambda_k^\dagger>0$ for $k\in\Nb$ and $\lambda_0^\dagger=0$, the desired convergence follows by letting $\tau \to\infty$ in the previous identity. {For later use, note that we have $\Pi_n(\ze)=u_{n,0}^\dagger$.} The general case is a direct consequence of the following proposition.
\end{proof}
\begin{proposition} \label{approxmn} 
Assume that $\sigma=0$, $\theta >0$, $\nu_0, \nu_1 \in (0,1)$, and $\om$ is a simple environment. We have
\[ \left | C^\dagger_{n,0}(\om,\tau) - \Pi_n(\omega) \right | \leq e^{-\theta\nu_0 \tau}. \]
\end{proposition}
\begin{proof}
Let $\om_\tau$ be the environment that coincides with $\om$ in $(-\tau,\infty)$ and that is constant and equal to $\om(-\tau)$ in $(-\infty,-\tau]$ (which means that $\om_\tau$ has no jumps in $(-\infty,-\tau]$). Since $\Ps_\tau^\dagger(\om_\tau)=\Ps_\tau^\dagger(\om)$ and $\om_\tau$ has no jumps in  $(-\infty,-\tau]$, we obtain
\begin{align}\label{wow}
 \Pi_n(\om_\tau)&=\sum_{k\geq 0}p_{n,k}^\dagger(\om_\tau,\tau)\,\Pb(\exists \beta\geq \tau \ \text{s.t.} \ R_0^{\om_\tau}(\beta)=0 \mid R_0^{\om_\tau}(\tau-)=k)\nonumber\\
 &=\sum_{k\geq 0}p_{n,k}^\dagger(\om,\tau)\,\Pi_k(\ze){=\sum_{k\geq 0}p_{n,k}^\dagger(\om,\tau)\,u_{k,0}^\dagger}=(\Ps_\tau^\dagger(\om)U_\dagger)_{n,0}=C_{n,0}^\dagger(\om,\tau),
\end{align}
where in the last line we used {$\Pi_k(\ze)=u_{k,0}^\dagger$ (from the end of the previous proof) and \eqref{ptut=ct}}.
Now combining \eqref{wow} with \eqref{absetmu0} applied to $\om_\tau$ with $\upsilon=\delta_n$ yields
\begin{equation}\label{coefetmu0}
 p_{n,0}^\dagger(\om,\tau)=p_{n,0}^\dagger(\om_\tau,\tau)\leq C_{n,0}^\dagger(\om,\tau) \leq p_{n,0}^\dagger(\om_\tau,\tau)+e^{-\theta\nu_0 \tau}=p_{n,0}^\dagger(\om,\tau)+e^{-\theta\nu_0 \tau}.
\end{equation}
Then, combining \eqref{absetmu0} applied to $\om$ with $\upsilon=\delta_n$ and \eqref{coefetmu0}, we get
\[C_{n,0}^\dagger(\om,\tau)-e^{-\theta\nu_0 \tau}\leq \Pi_n(\om)\leq C_{n,0}^\dagger(\om,\tau)+e^{-\theta\nu_0 \tau},\]
and the result follows. 
\end{proof}
\begin{remark}
If $\om$ has no jumps in $(-\tau,0)$, then $C^\dagger(\om,\tau)=U_\dagger\exp(\tau D_\dagger)$. In particular, $\Pi_n(\ze)=u_{n,0}^\dagger$.
\end{remark}
\begin{remark} 
Under the assumptions of Theorem \ref{thmf3} the Simpson index (see Remark \ref{simpsonindexannealed}) is given by 
\[ \mathbb{E}[\Sim^\om(\infty)]=\mathbb{E}[X^\om(\infty)^2 + (1-X^\om(\infty))^2] = 1 - 2 C_{1,0}^\dagger(\omega,\infty) + 2C_{2,0}^\dagger(\omega,\infty). \]
\end{remark}

\begin{remark} \label{periodic2}
If $\omega$ is a simple periodic environment with period $T_p > 0$, then \eqref{exprfctgenlt4infmom} can be re-written as $\Pi_n(\omega) =\lim_{m \to \infty} (U_\dagger B(\om)^m )_{n,0}$ where $B(\om) \coloneqq  [A_{N(T_p)}^\dagger(\om)A_{N(T_p)-1}^\dagger(\om)\cdots A_1^\dagger(\om)]^\top$. 
\end{remark}
As an application of Theorem \ref{thmf3} we obtain the following refinement of Proposition \ref{prop2.5} for mixed environments composed of a pure-jump subordinator $J$ and a simple environment $\om$ (see Fig. \ref{fig:mixed}).
\begin{proposition}\label{mixf4}
Assume that {$\sigma=0$,} $\theta>0$ and $\nu_0,\nu_1\in(0,1)$. For any $\tau_\star>0$, $n\in\Nb$, $x\in[0,1]$, and any simple environment $\om$, we have
\begin{equation}\label{formulecompacteexw}
\lim_{\tau\to\infty}\Eb\left[(1-X^{J\otimes_{\tau_\star} \om}(0))^n\mid X^{J\otimes_{\tau_\star} \om}(-\tau)=x\right]= \sum_{j = 0}^{n 2^N} \left ( \sum_{k = j}^{n 2^N} C^\dagger_{n,k}(\om,\tau_\star) v_{k,j}^\dagger \right ) \pi_j,
\end{equation}
where $N$ denotes the number of jumps of $\om$ in $[-\tau_\star,0]$.
\end{proposition}
\begin{proof}
Let $\om$ be a simple environment. Proceeding as in the proof of Proposition \ref{prop2.5}, but using Theorem \ref{thmf1} instead of Theorem \ref{thm2.3}, we obtain
 \begin{equation*}
\lim_{\tau\to\infty}\Eb\left[(1-X^{J\otimes_{\tau_\star} \om}(0))^n\mid X^{J\otimes_{\tau_\star} \om}(-\tau)=x\right]=\E\left[\pi_{R_0^{\om}(\tau_\star -)}\mid R_0^\om(0-)=n\right].
\end{equation*}
Since $U^\dagger V^\dagger= Id$ and the stochastic matrix $\Ps^\dagger_{\tau_\star}(\om)\coloneqq (p_{i,j}^\dagger(\om,{\tau_\star}))_{i,j\in\Nb_0^\dagger}$ defined via \[p_{i,j}^\dagger(\om,{\tau_\star})\coloneqq \Pb(R_0^\om({\tau_\star}-)=j\mid R_0^\om(0-)=i),\] 
% \[p_{i,j}^\dagger(\om,\tau)\coloneqq \Pb(R_0^\om(\tau-)=j\mid R_0^\om(0-)=i),\]
satisfies $C^\dagger(\om,{\tau_\star})=\Ps_{\tau_\star}^\dagger(\om)U_\dagger$ (see \eqref{ptut=ct}), the results follows.
\end{proof}

\subsection{Extensions of quenched results in Section \ref{s26}}\label{s72}
In this section we assume that $\sigma=0$ and extend some of the quenched results stated in Section \ref{s26} about the ancestral type distribution for simple environments. The next result allows us to get rid of the condition $\theta\nu_0>0$ in Theorem \ref{thm2.6}-(2). 
 
\begin{theorem}[Ancestral type distribution for simple environments]\label{thmfa}
Assume that $\sigma=0$ and let $\om\in\Db^\star$ be a simple environment with infinitely many jumps in $[0,\infty)$ and such that the distance between the successive jumps does not converge to $0$. Then the statement of {Theorem \ref{thm2.6}-(2),} excepting by Eq. \eqref{approxh(x)4}, remains true.
\end{theorem}
\begin{proof}
The case $\theta\nu_0>0$ is already covered by Theorem \ref{thm2.6}-(2). Assume now that $\theta \nu_0=0$, $\sigma=0$, and that $\om$ is as in the statement. For $\mu\in\Ms_1(\mathbb{N})$, we denote by $\mu^\omega_T(\beta)$ the distribution of $L_T^\omega(\beta-)$ given that $L_T^\omega(0-)\sim\mu$. We claim that, for all $\mu,\tilde\mu\in\Ms_1(\Nb)$,
\begin{equation}\tag{Claim 5}\label{tvd}
d_{TV}(\mu_T^\om(T),\tilde\mu_T^\om(T))\xrightarrow[T\to\infty]{}0, 
\end{equation}
{where $d_{TV}(\mu_1,\mu_2)$ stands for the total variation distance between $\mu_1$ and $\mu_2$.} If \eqref{tvd} is true, the rest of the proof follows as in the proof of Theorem \ref{thm2.6}-(2). In what follows we prove  \eqref{tvd}.

\smallskip
Let $0 < T_1 < T_2 <\cdots$ be the sequence of the jump times of $\omega$ and set $T_0 \coloneqq  0$ for convenience. On $(T_i, T_{i+1})$, $L_T^\omega$ has transition rates given by $(q^0(k,j))_{k,j\in\Nb}$ (see \eqref{kratespldasg}). For any $k > l$, let $H(k,l)$ denote the hitting time of $l$ by a Markov chain starting at $k$ and with transition rates given by $(q^0(i,j))_{i,j\in\Nb}$. Let $(S_l)_{l \geq 2}$ be a sequence of independent exponential random variables with parameter $(l-1) \theta\nu_1 + l(l-1)/2$ and note that $S_l \sim H(l,l-1)$ for $l\geq 2$. Using the Markov property, one can easily see that $H(k,1)$ is equal in distribution to $\sum_{l=2}^k S_l$. Therefore, for any $i$ such that $T_{i+1} < T$ and any $k \geq 1$, we have 
\begin{equation}\label{bou1}
\mathbb{P} \left (L_T^\omega((T-T_i)-) = 1 \mid L_T^\omega(T-T_{i+1}) = k \right ) {=} \mathbb{P} \left ( \sum_{l=2}^k S_l \leq T_{i+1} - T_{i} \right ) \geq \mathbb{P} \left ( \sum_{l=2}^{\infty} S_l \leq T_{i+1} - T_{i} \right ). 
\end{equation}
%Let $l_0$ be such that $\sigma l\leq (l-1)l/4$ for all $l \geq l_0$. For $l \geq l_0$, we define a random walk $Z^l \coloneqq  (Z^l(t))_{t \geq 0}$ on $\{l-1,l,l+1,\ldots\}$ starting at $l$ and such that: a) it jumps from $n\geq l$ at rate $(n-1)n$ to either $n-1$ or $n+1$, with probability $3/4$ and $1/4$, respectively, and b) it is absorbed at $l-1$. Let $W_l$ {denote} the hitting time of $l-1$ by $Z^l$. We can see that $\mathbb{E}[W_l] \leq 2/l(l-1)$. For $l \geq l_0$, let $Y^l \coloneqq  (Y^l(t))_{t \geq 0}$ be the Markov chain with transition rates given by $(q^0(i,j))_{i,j\in\Nb}$, starting at $l$ and killed at the hitting time of $[l-1]$. Note that $Y^l$ always jumps to the right with a smaller rate than $Z^l$ and jumps to the left at a higher rate. Moreover, the jumps of $Z^l$ and $Y^l$ are both of size $1$. Therefore, it is possible to couple $Y^l$ and $Z^l$ so that, for every $t \geq 0, Y^l(t) \leq Z^l(t)$. In this coupling the hitting time of $[l-1]$ by $Y^l$ (which is equal in law to $S_l$) is almost surely smaller or equal than $W_l$. We thus get, for all $l\geq l_0$,
%\[\mathbb{E}[S_l] \leq \mathbb{E}[W_l] \leq \frac{2}{l(l-1)}.\] 
Clearly $\sum_{l=2}^{\infty} \mathbb{E}[S_l] < \infty$ so $S^\infty\coloneqq \sum_{l=2}^{\infty} S_l<\infty$ almost surely. Moreover, since, for $l \geq 2$, the support of $S_l$ contains $0$, the support of $S^\infty$ contains $0$ as well. In particular $q(s) \coloneqq   \mathbb{P} ( \sum_{l=2}^{\infty} S_l \leq s)$ is positive for all $s > 0$. Thus, we obtain from \eqref{bou1} that, for $k\geq1$, 
\begin{eqnarray}
 \mathbb{P} \left (L_T^\omega((T-T_i)-) = 1 \mid L_T^\omega(T-T_{i+1}) = k \right ) \geq q(T_{i+1} - T_{i}). \label{probatouch0}
\end{eqnarray}
Let $L_T^\om$, $\tilde L_T^\om$ and $L_{T_i}^\om$, $i\geq 0$, be independent realizations of the line-counting process of the pLD-ASG with environment $\om$ (the subscript indicates the sampling time) and $L_T^\om(0-)\sim \mu$, $\tilde L_T^\om(0-)\sim\tilde \mu$, and $L_{T_i}^\om(0-)=1$. Let 
$i(T)\defeq \max\{i\in\Nb_0: T_i<T,\, L_T^\om((T-T_i)-)=\tilde L_T^\om((T-T_i)-)=1\}$, with the convention that the maximum of an empty set is $-\infty$. Set $T_{-\infty}\coloneqq-\infty$ for convenience. We define $(U_T^\omega(\beta))_{\beta \geq 0}$ and $(\tilde U_T^\omega(\beta))_{\beta \geq 0}$ by setting $U_T^\omega(\beta)\coloneqq L_T^\omega(\beta)$ and $\tilde U_T^\omega(\beta)\coloneqq \tilde L_T^\omega(\beta)$ for $\beta<T-T_{i(T)}$ and $U_T^\omega(\beta)\coloneqq\tilde U_T^\omega(\beta)=L_{T_{i(T)}}^\om(\beta-(T-T_{i(T)}))$ for $\beta \geq T-T_{i(T)}$.
Note that $U_T^\om$ and $\tilde U_T^\omega$ have the same distribution as $L_T^\om$ and $\tilde L_T^\om$, respectively. In particular, we have $U_T^\omega(T-) \sim \mu_{T}^\omega(T)$ and $\tilde U^\omega_T(T-)\sim \tilde \mu^{\omega}_T(T)$. Moreover, we have $U_T(\omega,\beta)=\tilde U_T(\omega,\beta)$ for all $\beta \geq T-T_{i(T)}$. Therefore, 
\begin{eqnarray}
d_{TV}(\mu_T^\om(T),\tilde\mu_T^\om(T)) \leq \mathbb{P} \left ( U_T^\omega(T-) \neq \tilde U_T^\omega(T-) \right ) \leq \mathbb{P} \left ( i(T) = -\infty \right ). \label{rapprochementsm}
\end{eqnarray}
Let $N(T)$ be the number of jumps of $\om$ in $[0,T]$. According to \eqref{probatouch0}, we have, for $k_1, k_2 \geq 1$ with $k_1 \neq k_2$, 
\begin{align*}
 \mathbb{P} &\left (L_T^\omega((T-T_i)-)= 1,\, \tilde L_T^\omega((T-T_i)-) = 1 \mid  L_T^\omega(T-T_{i+1}) = k_1, \tilde L_T^\omega(T-T_{i+1}) = k_2 \right ) \geq q(T_{i+1} - T_{i})^2.
\end{align*}
Therefore, using \eqref{rapprochementsm}, we obtain
\[ d_{TV}(\mu_T^\om(T),\tilde\mu_T^\om(T)) \leq \mathbb{P} \left ( I_0(T) = -\infty \right ) \leq \prod_{i=1}^{N(T)} \left ( 1-q(T_{i} - T_{i-1})^2 \right ) =: \varphi_{\omega}(T). \]
Note that $\varphi_{\omega}$ does not depend on $\mu$ and $\tilde \mu$. Recall that by assumption the sequence of jump times $T_1, T_2,\ldots$ is infinite and the distance between the successive jumps does not converge to $0$. Therefore, there is $\epsilon > 0$ such that, for infinitely many indices $i$, we have $T_{i+1}-T_i > \epsilon$. Thus, the number of factors smaller than $1-q(\epsilon)^2 < 1$ in the product defining $\varphi_{\omega}(T)$ converges to infinity as $T\to\infty$. We deduce that $\varphi_{\omega}(T)\to 0$ as $T\to\infty$, which proves \eqref{tvd}, concluding the proof. 
\end{proof}
%\textbf{Refinements for $\sigma=0$.}
%Under this additional assumption we provide a more explicit expression for $h^{\omega}_T(x)$. This is possible thanks to the following explicit diagonalization of the matrix $Q^0$ (the transition
%matrix of the process $L$ under the null environment).
The following explicit diagonalization of the matrix $Q^0$ (the transition
matrix of the process $L$ under the null environment) will allow us to obtain a more explicit expression for $h^{\omega}_T(x)$.
\begin{lemma}\label{diagpldasg}
Assume that $\sigma=0$ and set, for $k\in\Nb$, $\lambda_k\coloneqq -q^0(k,k)$, and, for $k\in\Nb$, $\gamma_k\coloneqq q^0(k,k-1)$, {where $q^\mu(\cdot,\cdot)$ is defined in \eqref{kratespldasg}}. In addition, let 
 \begin{itemize}
  \item[(i)] $D$ be the diagonal matrix with diagonal entries $(-\lambda_i)_{i\in\Nb}$. 
  \item[(ii)] $U\coloneqq (u_{i,j})_{i,j\in\Nb}$ where, for all $i\in\Nb$, $u_{i,i} \coloneqq 1$, $u_{i,j} \coloneqq 0$ for $j > i$ and, when $i \geq 2$, $u_{i,i-1} \coloneqq  \gamma_{i}/(\lambda_{i} - \lambda_{i-1})$ and the coefficients $(u_{i,j})_{j \in [i-2]}$ are defined via the recurrence relation 
\begin{eqnarray}
u_{i,j} \coloneqq  \frac{1}{\lambda_i - \lambda_j} \left ( \gamma_i {u_{i-1,j}} + \theta\nu_0 \left ( \sum_{l=j}^{i-2} {u_{l,j}} \right ) \right ). \label{recrelalphani4killed}
\end{eqnarray}
\item[(iii)] $V\coloneqq (v_{i,j})_{i,j\in\Nb}$ where, for all $i\in\Nb$, $v_{i,i} \coloneqq 1$, $v_{i,j} \coloneqq 0$ for $j > i$ and, when $i \geq 2$, the coefficients $(v_{i,j})_{j \in [i-1]}$ are defined via the recurrence relation 
\begin{eqnarray}
v_{i,j} \coloneqq  \frac{-1}{(\lambda_i - \lambda_j)} \left [ \left ( \sum_{l=j+2}^i v_{i,l} \right ) \theta\nu_0 + v_{i,j+1} \gamma_{j+1} \right ]. \label{recrelani4killed}
\end{eqnarray}
 \end{itemize}
with the convention that an empty sum equals $0$. Then, we have 
\[Q=U D V\quad \textrm{and}\quad U V=V U=Id.\]
\end{lemma}

\begin{proof}
Analogous to the proof of Lemma \ref{diag}. 
\end{proof}

Now, we consider the polynomials $S_k, k \in\Nb$ defined via
\begin{eqnarray}
S_k(x) \coloneqq {\sum_{i=1}^{k} v_{k,i} x^i}. \label{newbasis9pldasg}
\end{eqnarray}
%%%
In addition, for $z\in(0,1)$, we define the matrices $\Bs(z)\coloneqq (\Bs_{i,j}(z))_{i,j\in\Nb}$ and $\Phi(z)\coloneqq (\Phi_{i,j}(z))_{i,j\in\Nb}$ via

\begin{equation} \label{defbetakipldasg}
 \Bs_{i,j}(z)\coloneqq \Pb(i+B_i(z)=j),\quad \textrm{$i,j\in\Nb$}\quad\textrm{and}\quad \Phi(z)\coloneqq U^\top \Bs(z)^\top V^\top,
\end{equation}
where $B_i(z)\sim\bindist{i}{z}$. The fact that the matrix product defining $\Phi(z)$ is well-defined can be justified similarly as in the proof of Theorem \ref{thmf3}. The same is true for the matrix products in \eqref{defmatApldasg} and \eqref{defcoeff4}.
\begin{theorem} \label{thmf4}
Assume that $\sigma=0$ and let $\om$ be a simple environment with infinitely many jumps on $[0,\infty)$ and such that the distance between the successive jumps does not converge to $0$. Let $N$ be the number of jumps of $\om$ in $(0,T)$ and let $(T_i)_{i=1}^N$ be {the sequence of the jump times} in increasing order. We set $T_0 \coloneqq  0$ for convenience. For any $m\in[N]$, we define the matrix $A_m(\omega)\coloneqq (A_{i,j}^{m}(\omega))_{i,j\in\Nb}$ by 
\begin{eqnarray} 
A_m(\omega) \coloneqq  \exp \left ( (T_m - T_{m-1}) D \right ) \Phi(\Delta \omega(T_m)). \label{defmatApldasg}
\end{eqnarray}
Then for all $x \in (0,1), n \in \mathbb{N}$, we have 
\begin{eqnarray} 
h^{\omega}_T(x) = 1 - \sum_{k = 1}^{2^N} C_{1,k}(\om,T) S_k(1-x), \label{exprfctgenlt4}
\end{eqnarray}
where the matrix $C(\om,T)\coloneqq (C_{n,k}(\om,T))_{k,n\in\Nb}$ is given by
\begin{align} 
C(\om,T)\coloneqq U \exp \left ( (T-T_N)D \right ) \left[A_1(\om)A_{2}(\om)\cdots A_N(\om)\right]^\top. \label{defcoeff4}
\end{align}
%ON A ENLEVE LE FACTEUR $e^{- \lambda_{1} (t - T_{N})}$ CAR $\lambda_{1}=0$. 
Moreover, for any $x\in(0,1)$,
\begin{eqnarray} 
 \ h^{\omega}(x) = 1 - \sum_{k = 1}^{\infty} C_{1,k}(\omega, \infty) S_k(1-x), \label{exprfctgenlt4inf}
\end{eqnarray}
where the series in \eqref{exprfctgenlt4inf} is convergent 
%for almost every realization $\omega$ of the L\'evy environment, 
and where 
\begin{align} 
C_{1,k}(\omega, \infty) \coloneqq  {\lim_{m \rightarrow \infty}} \left ( U \left[A_1(\om)A_{2}(\om)\cdots A_m(\om)\right]^\top \right )_{1,k}, \label{defcoeff4inf}
\end{align}
and the above limit is well-defined. 
\end{theorem}

% \begin{remark}
% The fact that the matrix products \eqref{defmatApldasg} and \eqref{defcoeff4} are well-defined can be justified similarly as in the proof of Theorem \ref{x0condpastenvII}.
% \end{remark}
\begin{proof}
We are interested in the generating function of $L_T^\omega(T-)$.
For $s>0$, we define the stochastic matrix $\Ps^T_s(\om)\coloneqq (p^T_{i,j}(\om,s))_{i,j\in\Nb}$ via
\[p^T_{i,j}(\om,s)\coloneqq \Pb(L_T^\om(s-)=j\mid L_T^\om(0-)=i).\]
 We also define $(M(s))_{s\geq 0}$ via $M(s)\coloneqq \exp (sQ^0 )$, i.e $M$ is the semi-group of $L^{\ze}$. Let $T_1<T_{2}<\cdots<T_N$ be the sequence of jump times of $\om$ in $[0,T]$. Disintegrating with respect to the values of $L_T^\om( (T-T_i)-)$ and $L_T^\om( T - T_i)$, $i\in [N]$, we obtain 
 \begin{equation}\label{ptdecpldasg}
  {\Ps^T_T(\om)}=M(T-T_N)\Bs(\Delta\om (T_N))M(T_N-T_{N-1})\Bs(\Delta\om(T_{N-1}))\cdots \Bs(\Delta\om(T_1))M(T_1).
 \end{equation}
 In addition,
 \begin{equation}\label{gfrtpldasg}
   \mathbb{E}[y^{L_T^\omega(T-)} \mid L_T^\omega(0-)=n] = ({\Ps^T_T(\om)} \rho(y))_n,\quad \textrm{where}\quad \rho(y)\coloneqq (y^i)_{i\in\Nb}.
 \end{equation}
Thanks to Lemma \ref{diagpldasg}, we have $M(\beta)=U E(\beta)V$, where $E(\beta)$ is the diagonal matrix with diagonal entries $(e^{-\lambda_j \beta})_{j\in\Nb}$. Moreover, $\rho(y)=U {S(y)}$, where {$S(y)\coloneqq(S_k(y))_{k\in\Nb}$.} Using this together with Eq. \eqref{ptdecpldasg} and the relations $M(\beta)=U E(\beta)V$ and {$V U=Id$}, we obtain 
\begin{equation}\label{gfrtpldasg2}
 {\Ps^T_T(\om)} \rho(y)=U E(T-T_N) \Phi(\Delta \om(T_N))^\top E(T_N-T_{N-1})\Phi(\Delta \om(T_{N-1}))^\top\cdots\Phi(\Delta \om(T_1))^\top E(T_1){S(y)}. 
 \end{equation}
 Thus, using the definition of the matrices $A_i(\om)$, we get
\begin{align}
 {\Ps^T_T(\om)} \rho(y)&= U E(T-T_N) A_N(\om)^\top A_{N-1}(\om)^\top\cdots A_1(\om)^\top {S(y)} \nonumber \\
 &= U E(T-T_N) \left[A_1(\om)A_{2}(\om)\cdots A_N(\om)\right]^\top {S(y)}=C(\om,T){S(y)}. \nonumber 
\end{align}

Now, using the previous identity, Lemma \ref{hTq} and Eq. \eqref{gfrtpldasg}, we obtain 
\begin{align*}
h^{\omega}_T(x)=1-\mathbb{E}\left[(1-x)^{L_T^\omega(T-)} \mid L_T^\omega(0-)=1\right] =  1 - {\sum_{k = 1}^{\infty}} C_{1,k}(\omega,T) S_k(1-x). 
\end{align*}
Proceeding as in the proof of Theorem \ref{thmf3}, one shows that $C_{1,k}(\omega,T)=0$ for $k > 2^N$, and \eqref{exprfctgenlt4} follows. 
\smallskip

Let us now analyze $C_{1,k}(\omega,T)$ as $T\to\infty$. First, note that, on the one hand, from \eqref{exprfctgenlt4}, we have 
\[\Eb[y^{L_T^\omega(T-)}\mid L_T^\omega(0-)=1]=\sum_{k=1}^{\infty} C_{1,k}(\omega,T) S_k(y).\]
On the other hand, we have
\[\Eb[y^{L_T^\omega(T-)}\mid L_T^\omega(0-)=1]=\sum_{k=1}^{\infty} \mathbb{P}({L_T^\omega(T-)=k \mid L_T^\omega(0-)=1}) y^k.\]
Since $U^\top$ is the transition matrix from the basis $(y^k)_{k \in \Nb}$ to the basis $(S_k(y))_{k \in \Nb}$, we deduce that 
\begin{equation}\label{hamma}
C_{1,k}(\omega,T) = \sum_{i \in \Nb} u_{i,k} \mathbb{P}(L_T^\omega(T-)=i \mid L_T^\omega(0-)=1) = \mathbb{E} [ u_{L_T^\omega(T-),k} \mid L_T^\omega(0-)=1 ].
\end{equation}
From Theorem \ref{thmfa}, we know that the distribution of $L_T^\omega(T-)$ converges when $T\to\infty$. In addition, {according to} Lemma \ref{majoalphaki} the function $i \mapsto u_{i,k}$ is bounded, and hence $C_{1,k}(\omega,T)$ converges to a real {number.}
Recall that $T_1 < T_2 < \cdots$ is the increasing sequence of the jump times of $\om$ and that this sequence converges to infinity. Therefore 
\[ \lim_{T \rightarrow \infty} C_{1,k}(\omega,T) = \lim_{m \rightarrow \infty} C_{1,k}(\omega,T_m) = \lim_{m \rightarrow \infty} {\left ( U \left[A_1(\om)A_{2}(\om)\cdots A_m(\om)\right]^\top \right )_{1,k}}, \] 
where we used \eqref{defcoeff4} in the last step. This shows that the limit on the right hand side of \eqref{defcoeff4inf} exists and equals $\lim_{T \rightarrow \infty} C_{1,k}(\omega,T)$. 

It remains to prove \eqref{exprfctgenlt4inf} together with the convergence of the corresponding series. We already know from Theorem \ref{thmfa} that $h^{\omega}_T(x)$ converges to $h^{\omega}(x)$ when $T\to\infty$ and we have just proved {\eqref{exprfctgenlt4} and} that for any $k \geq 1$, $C_{1,k}(\omega,T)$ converges to $C_{1,k}(\omega,\infty)$, defined in \eqref{defcoeff4inf}, when $T\to\infty$. Now, we claim that, for all $y\in[0,1]$ and $T>T_1$, 
\begin{equation}\tag{Claim 6}
 | C_{1,k}(\omega,T) S_k(y) | \leq 4^k \times (2ek)^{(k+\theta)/2} e^{- \lambda_k T_1}. \label{bornecvdom}
\end{equation}
Assume that \eqref{bornecvdom} is true. Then \eqref{exprfctgenlt4inf} and the convergence of the series follow using the dominated convergence theorem. It only remains to prove \eqref{bornecvdom}. As in the proof of \eqref{hamma}, one shows that
\[{\Ps^T_{(T-T_1)+}(\omega)} \rho(y)=E[y^{L_T^\omega(T-T_1)}\mid L_T^\omega(0-)=1]=\sum_{k=1}^{\infty} \tilde{C}_{1,k}(\omega,T) S_k(y),\]
where {$\tilde{C}_{1,k}(\omega,T) = \mathbb{E} [ u_{L_T^\omega(T-T_1),k} \mid L_T^\omega(0-)=1 ]$}. Proceeding as in the proof of \eqref{gfrtpldasg2}, we can prove that 
\begin{eqnarray}
{\Ps^T_{(T-T_1)+}(\omega)} \rho(y)=U E(T-T_N) \Phi(\Delta \om(T_N))^\top E(T_N-T_{N-1})\Phi(\Delta \om(T_{N-1}))^\top\cdots\Phi(\Delta \om(T_1))^\top {S(y)}. \label{gfrtpldasg3}
\end{eqnarray}
Since $E(T_1)$ is diagonal with entries $(e^{-\lambda_j T_1})_{j\in\Nb}$, we conclude from \eqref{gfrtpldasg2} and \eqref{gfrtpldasg3} that $C_{1,k}(\omega,T) = e^{-\lambda_k T_1} \tilde C_{1,k}(\omega,T)$. Therefore 
\[ C_{1,k}(\omega,T) = e^{-\lambda_k T_1} \mathbb{E} [ u_{L_T^\omega(T-T_1),k} \mid L_T^\omega(0-)=1 ]. \]
This together with Lemma \ref{majoalphaki} (see below) implies that, for all $k\geq 1$ and $t\geq 0$,
\[  \ |C_{1,k}(\omega,T)| \leq (2ek)^{(k+\theta)/2} e^{- \lambda_k T_1}. \]
Combining this with Lemma \ref{majosommedesaki}, we obtain \eqref{bornecvdom}, which concludes the proof. 
\end{proof}
\begin{lemma} \label{majoalphaki} For all $k\geq 1$
\[  \ \sup_{j \geq 1} u_{j,k} \leq (2ek)^{(k+\theta)/2}. \]
\end{lemma}
\begin{proof} 
Let $k \geq 1$. By the definition of the matrix $U$ in Lemma \ref{diagpldasg}, the sequence $(u_{j,k})_{j \geq 1}$ satisfies 
\[u_{j,k}  = 0 \ \text{for} \ j < k, \ u_{k,k} = 1, \ u_{k+1,k} = \frac{\gamma_{k+1}}{\lambda_{k+1} - \lambda_{k}},\]
\[u_{k+l,k}  = \frac{1}{\lambda_{k+l} - \lambda_k} \left ( \gamma_{k+l} u_{k+l-1,k} + \theta\nu_0  \sum_{j=0}^{l-2} u_{k+j,k} \right ) \ \text{for} \ l \geq 2.\] 
Let $M_k^{j} \coloneqq  \sup_{i \leq j} u_{i,k}$. By the definitions of $\gamma_{j+1}$, $\lambda_{k+1}$, $\lambda_k$ (see Lemma \ref{diagpldasg}), we have
\[\gamma_{k+1} = \lambda_{k+1}- (k-1) \theta\nu_0 > \lambda_{k+1} - \lambda_{k}.\] 
This together with the above expressions yields that 
\[M_k^k  = 1, \qquad M^{k+1}_k = \frac{\lambda_{k+1}- (k-1) \theta\nu_0}{\lambda_{k+1} - \lambda_{k}} \leq \frac{\lambda_{k+1}}{\lambda_{k+1} - \lambda_{k}} = 1 + \frac{\lambda_{k}}{\lambda_{k+1} - \lambda_k}.\]
{Moreover, for $l\geq 2$, we have \[u_{k+l,k}\leq M^{k+l-1}_k \frac{\gamma_{k+l} + (l-1) \theta\nu_0}{\lambda_{k+l}-\lambda_k}=M^{k+l-1}_k \frac{\lambda_{k+l} - (k-1) \theta\nu_0}{\lambda_{k+l}-\lambda_k}\leq M^{k+l-1}_k \frac{\lambda_{k+l}}{\lambda_{k+l}-\lambda_k}.\]
Hence, we have
\[M^{k+l}_k =M^{k+l-1}_k\vee u_{k+l,k} \leq M^{k+l-1}_k \times \frac{\lambda_{k+l}}{\lambda_{k+l} - \lambda_k} =M^{k+l-1}_k \times \left ( 1 + \frac{\lambda_{k}}{\lambda_{k+l} - \lambda_k} \right ).\] 
}
As a consequence, we have 
\begin{eqnarray}
\sup_{j \geq 1} u_{k,j} \leq \prod_{l=1}^{\infty} \left ( 1 + \frac{\lambda_{k}}{\lambda_{k+l} - \lambda_k} \right ) =: M^{\infty}_k. \label{majoparprodinf}
\end{eqnarray}
Since $\lambda_{k+l} \underset{}{\sim} l^2$ as $l\to\infty$, it is easy to see that the infinite product $M^{\infty}_k$ is finite. Then, 
\[ M^{\infty}_k = \exp \left [ \sum_{l=1}^{\infty} \log \left ( 1 + \frac{\lambda_{k}}{\lambda_{k+l} - \lambda_k} \right ) \right ] \leq \exp \left [ \sum_{l=1}^{\infty} \frac{\lambda_{k}}{\lambda_{k+l} - \lambda_k} \right ] \leq \exp \left [ \frac{\lambda_{k} \log(2ek)}{2(k-1)} \right ], \]
%for any $l \geq 1$, 
%\begin{align*}
%\frac{\lambda_k}{\lambda_{k+l}} & = \frac{k(k-1) + (k-1) \theta}{(k+l)(k+l-1) + (k+l-1) \theta} \\
%& \leq \frac{k(k-1)}{(k+l)(k+l-1)} + \frac{(k-1) \theta}{(k+l)(k+l-1)} \\
%\end{align*}
where we used Lemma \ref{sommeinverselambda} (see below) in the last step. Since $\lambda_{k} = (k-1)(k+\theta)$ (see Lemma \ref{diagpldasg}), the desired result follows. 
\end{proof}
\begin{lemma} \label{sommeinverselambda} For all $k\in\Nb$
\[ {\sum_{l=1}^{\infty}} \frac{1}{\lambda_{k+l} - \lambda_k} \leq \frac{\log(2ek)}{2(k-1)}. \]
\end{lemma}

\begin{proof}
Using the definition of $\lambda_k$ in Lemma \ref{diagpldasg} we have 
\begin{align*}
{\sum_{l=1}^{\infty}} \frac{1}{\lambda_{k+l} - \lambda_k} & \leq {\sum_{l=1}^{\infty}} \frac{1}{(k+l)(k+l-1) - k(k-1)} 
 \leq \frac{1}{(k+1)k - k(k-1)} + \int_k^{\infty} \frac{1}{x(x+1) - k(k-1)} \dd x \\
& = \frac{1}{2k} + \int_1^{\infty} \frac{1}{u^2 + (2k-1)u} \dd u = \frac{1}{2k} + \lim_{a \rightarrow \infty} \int_1^{a} \frac{1}{u(u + 2k-1)} \dd u \\
& \leq \frac{1}{2k-1}\left[1 + \lim_{a \rightarrow \infty} \left ( \int_1^{a} \frac{1}{u} du - \int_1^{a} \frac{1}{u + 2k-1} \dd u \right ) \right]\\
&= \frac{1}{2k-1}\left[1 + \lim_{a \rightarrow \infty}  \log\left(\frac{a2k}{a+2k-1}\right ) \right]\leq  \frac{\log(2ek)}{2(k-1)}. 
\end{align*}
\end{proof}
\begin{lemma} \label{majosommedesaki} For all $k\in\Nb$, we have
\[ \ \sup_{y \in [0,1]} \left | S_k(y) \right | \leq 4^{k}. \]
\end{lemma}

\begin{proof} 
By definition of the polynomials $S_k$ in \eqref{newbasis9pldasg}, we have for $k\geq 1$
\begin{eqnarray}
\ \sup_{y \in [0,1]} \left | S_k(y) \right | \leq \sum_{i=1}^{k} | v_{k,i} |. \label{sup<sumcoef}
\end{eqnarray}
Let us fix $k \geq 1$ and define $S^k_j \coloneqq  \sum_{i \geq j}^{k} | v_{k,i} |$. Note that $S^k_j = 0$ for $j > k$ and that $S^k_k = 1$ by the definition of the matrix $(v_{i,j})_{i,j \in \Nb}$ in Lemma \ref{diagpldasg}. In particular, the result is true for $k=1$. Thus, we assume that $k > 1$ from now on. Using \eqref{recrelani4killed}, we see that, for any $j \in [k-1]$, 
\begin{align*}
S^k_j = S^k_{j+1} + | v_{k,j} | & = S^k_{j+1} + \left | \frac{-1}{(\lambda_k - \lambda_j)} \left [ \left ( \sum_{l=j+2}^k v_{k,l} \right ) \theta\nu_0 + v_{k,j+1} \gamma_{j+1} \right ] \right | \\
& \leq S^k_{j+1} + \frac{1}{(\lambda_k - \lambda_j)} \left [ S^k_{j+2} \theta\nu_0 + (S^k_{j+1}-S^k_{j+2}) \gamma_{j+1} \right ] \\
& \leq \left ( 1 + \frac{\gamma_{j+1}}{\lambda_k - \lambda_j} \right ) S^k_{j+1} + \frac{\theta\nu_0 - \gamma_{j+1}}{\lambda_k - \lambda_j} S^k_{j+2}. 
\end{align*}
Note that $\frac{\theta\nu_0 - \gamma_{j+1}}{\lambda_k - \lambda_j}\leq 0$, because of the definition of the coefficients $\gamma_i$ in Lemma \ref{diagpldasg}. Thus, for $j\in[k-1]$, 
\begin{eqnarray}
\ S^k_j \leq \left ( 1 + \frac{\gamma_{j+1}}{\lambda_k - \lambda_j} \right ) S^k_{j+1}. \label{relrecskj}
\end{eqnarray}
By the definitions of $\gamma_{j+1}$, $\lambda_k$, $\lambda_j$ in Lemma \ref{diagpldasg}, and using that $j<k$, we have 
\begin{align*}
\frac{\gamma_{j+1}}{\lambda_k - \lambda_j} & = \frac{(j+1)j + j {\theta \nu_1} + \theta\nu_0}{k(k-1) - j(j-1) + (k-j) {\theta \nu_1} + (k-j) \theta\nu_0} \\
& \leq \frac{(j+1)j}{j(k-1) - j(j-1)} + \frac{j {\theta \nu_1}}{(k-j) {\theta \nu_1}} + \frac{\theta\nu_0}{(k-j) \theta\nu_0} \leq 2\, \frac{j+1}{k-j}. 
\end{align*}
In particular,
\[ 1 + \frac{\gamma_{j+1}}{\lambda_k - \lambda_j} \leq \frac{k + j + 2}{k-j}. \]
Plugging this into \eqref{relrecskj} yields, for all $j\in[k-1]$,
\[ \ S^k_j \leq \frac{k + j + 2}{k-j} S^k_{j+1}. \]
Then, applying the previous inequality recursively and combining with $S^k_k = 1$, we get 
\[ \sum_{i=1}^{k} | v_{k,i} | = S^k_1 \leq \prod_{j=1}^{k-1} \frac{k + j + 2}{k-j} =\binom{2k+1}{k-1}= \frac{\binom{2k+1}{k-1} + \binom{2k+1}{k+2}}{2} \leq 2^{2k+1}/2 = 4^k. \]
 Combining with \eqref{sup<sumcoef}, we obtain the desired result. 
 \end{proof}

\appendix
\section{\texorpdfstring{$J_1$-Skorokhod topology}{} and weak convergence}

\subsection{{Definitions and remarks on the \texorpdfstring{$J_1$-Skorokhod topology}{}}{}}\label{A1}
For $T>0$, as in the beginning of Section~\ref{S2} we denote by $\Db_{0,T}$ the space of c\`{a}dl\`{a}g functions in $[0,T]$ with values on $\Rb$. Let $\Cs_T^\uparrow$ denote the set of increasing, continuous functions from $[0,T]$ onto itself. For $\lambda\in\Cs_T^\uparrow$, we set
\begin{eqnarray}
\lVert \lambda\rVert_T^0\coloneqq \sup_{0\leq u<s\leq T }\left \lvert \log\left(\frac{\lambda(s)-\lambda(u)}{s-u}\right)\right\rvert. \label{normbij}
\end{eqnarray}
We define the Billingsley metric $d_T^0$ in $\Db_{0,T}$ via
{\begin{eqnarray}
d_T^0(f,g)\coloneqq \inf_{\lambda\in\Cs_T^\uparrow}\{\lVert \lambda\rVert_T^0\vee \lVert f-g\circ\lambda\rVert_{T,\infty} \}, \ \text{where} \ \lVert f\rVert_{T,\infty}\coloneqq \sup_{s\in[0,T]}|f(s)|. \label{defdt0}
\end{eqnarray}}
The metric $d_T^0$ induces the $J_1$-Skorokhod topology in $\Db_{0,T}$. An important feature is that the space $(\Db_{0,T},d_T^0)$ is separable and complete. The role of the time-change $\lambda$ in the definition of $d_T^0$ is to capture the fact that two c\`{a}dl\`{a}g functions can be close in spite of a small difference between their jumping times.
\smallskip

For $T>0$, a function $\om\in\Db_{0,T}$ is said to be pure-jump if $\sum_{u\in(0,T]}|\Delta \om(u)|<\infty$ and 
for all $t\in (0,T]$, \[\om(t)-\om(0)=\sum_{u\in(0,t]}\Delta \om(u),\]
where $\Delta \om(u)\coloneqq \om(u)-\om(u-)$, $u\in[0,T]$. In the set of pure-jump functions, we consider the following metric 
{\begin{eqnarray}
d_T^\star (\om_1,\om_2)\coloneqq \inf_{\lambda\in\Cs_T^\uparrow}\left\{\lVert \lambda\rVert_T^0\vee \sum_{u\in[0,T]} \left\lvert \Delta \om_1(u)-\Delta (\om_2\circ \lambda)(u)\right\rvert \right\}.\label{defdtstar}
\end{eqnarray}}
The next result provides comparison inequalities between the metrics $d_T^0$ and $d_T^\star$. 
\begin{lemma}\label{zero-star}
 Let $\om_1$ and $\om_2$ be two pure-jump functions with $\om_1(0)=\om_2(0)=0$, then
\[d_T^0(\om_1,\om_2)\leq d_T^\star(\om_1,\om_2).\]
If in addition, $\om_1$ and $\om_2$ are non-decreasing, and $\om_1$ jumps exactly $n$ times in $[0,T]$, then
\[d_T^\star(\om_1,\om_2)\leq (4n+3) d_T^0(\om_1,\om_2).\]
\end{lemma}
\begin{proof}
Let $\lambda\in \Cs_T^\uparrow$ and set $f\coloneqq \om_1$ and $g\coloneqq \om_2\circ\lambda$. Since $f$ and $g$ are pure-jump functions with the same value at $0$, we have, for any $t\in[0,T]$,
\[\lvert f(t)-g(t)\rvert=\left\lvert\sum_{u\in[0,t]}(\Delta f(u)-\Delta g(u))\right\rvert\leq \sum_{u\in[0,t]}\left\lvert\Delta f(u)-\Delta g(u)\right\rvert.\]
The first inequality follows. Now, assume that $\om_1$ and $\om_2$ are non-decreasing and that $\om_1$ has $n$ jumps in $[0,T]$. Let $t_1<\cdots<t_n$ be {the consecutive jump times} of $\om_1.$ We first prove that, for any $k\in[n]$,
{\begin{equation}\label{boundsumjumps}
\sum_{u\in[0,t_k]}\lvert\Delta f(u)-\Delta g(u) \rvert \leq (4k+1) \lVert f-g\rVert_{t_k,\infty},
\end{equation}}
where $\lVert \cdot \rVert_{t,\infty}$, $t>0$, is defined in \eqref{defdt0}. We proceed by induction on $k$. Note that
\[\sum_{u\in[0,t_1]}\lvert\Delta f(u)-\Delta g(u) \rvert=\sum_{u\in[0,t_1)}\Delta g(u) +\lvert\Delta f(t_1)-\Delta g(t_1) \rvert
\leq g(t_1-)+2\lVert f-g\rVert_{t_1,\infty}\leq 3\lVert f-g\rVert_{t_1,\infty},\]
which proves {\eqref{boundsumjumps} for $k=1$.} Now, assuming that {\eqref{boundsumjumps}} is true for $k\in[n-1]$, we obtain 
\begin{align*}
\sum_{u\in[0,t_{k+1}]}\!\!\!\!\lvert\Delta f(u)-\Delta g(u) \rvert&=\sum_{u\in[0,t_k]}\lvert\Delta f(u)-\Delta g(u) \rvert +\sum_{u\in(t_{k},t_{k+1})}\!\!\!\Delta g(u)+\lvert\Delta f(t_{k+1})-\Delta g(t_{k+1}) \rvert\\
\leq& (4k+1)\lVert f-g\rVert_{t_k,\infty}+g(t_{k+1}-)-g(t_k)+2\lVert f-g\rVert_{t_{k+1},\infty}\\
=& (4k+1)\lVert f-g\rVert_{t_k,\infty}+(g(t_{k+1}-)-f(t_{k+1}-))-(g(t_k)-f(t_{k}))+2\lVert f-g\rVert_{t_{k+1},\infty}\\
 \leq& (4k+1)\lVert f-g\rVert_{t_k,\infty}+4\lVert f-g\rVert_{t_{k+1},\infty}\leq  (4(k+1)+1)\lVert f-g\rVert_{t_{k+1},\infty}.
\end{align*} 
Hence, {\eqref{boundsumjumps}} also holds for $k+1$. This ends the proof of {\eqref{boundsumjumps} by induction.} Finally, using {\eqref{boundsumjumps},} we get
\begin{align*}
 \sum_{u\in[0,T]}\lvert\Delta f(u)-\Delta g(u) \rvert&=\sum_{u\in[0,t_{n}]}\lvert\Delta f(u)-\Delta g(u) \rvert+\sum_{u\in(t_{n},T]}\Delta g(u)\\
 &\leq (4n+1)\lVert f-g\rVert_{t_{n},\infty}+ g(T)-g(t_n)\leq (4n+3)\lVert f-g\rVert_{T,\infty},
\end{align*}
ending the proof.
\end{proof}
\subsection{Bounded Lipschitz metric and weak convergence}\label{A2}
Let $(E,d)$ denote a complete and separable metric space. It is well known that the topology of weak convergence of probability measures on $E$ is induced by the Prokhorov metric. An alternative metric inducing this topology is given by the bounded Lipschitz metric, whose definition is recalled in this section.  
\begin{definition}[Lipschitz function] A real valued function $F$ on $(E,d)$ is said to be Lipschitz if there is $K>0$ such that
\[|F(x)-F(y)|\leq K d(x,y),\quad\textrm{for all $x,y\in E$}.\]
 We denote by $\textrm{BL}(E)$ the vector space of bounded Lipschitz functions on $E$. The space $BL(E)$ is equipped with the norm 
{\begin{equation}
\lVert F\rVert_{\textrm{BL}}\coloneqq \sup_{x\in E}|F(x)| \vee \sup_{x,y\in E:\, x\neq y}\left\{\frac{|F(x)-F(y)|}{d(x,y)}\right\},\quad F\in BL(E). \label{defnormbl}
 \end{equation}}
\end{definition}
\begin{definition}[Bounded Lipschitz metric] \label{defblmenv}
 Let $\mu,\nu$ be two probability measures on $E$. The bounded Lipschitz distance between $\mu$ and $\nu$ is defined by
 {\begin{equation}\varrho_{E}(\mu,\nu)\coloneqq \sup\left\{\left\lvert \int F d\mu- \int F d\nu\right\rvert: F\in \textrm{BL}(E),\, \lVert F\rVert_{\textrm{BL}}\leq 1 \right\}.\label{defblm} \end{equation}}
\end{definition}
The bounded Lipschitz distance defines a metric on the space of probability measures on $E$. Moreover, the bounded Lipschitz distance metrizes the weak convergence of probability measures on $E$, i.e.
\[\varrho_E(\mu_n,\mu)\xrightarrow[n\to\infty ]{}0\quad \Longleftrightarrow\quad \mu_n \xrightarrow[n\to\infty]{(d)}\mu.\]

\section{Table of notations}\label{Nota}

\begin{table}[h!]
\begin{tabularx}{\textwidth}{@{}XXX@{}}
\toprule
  \textit{Notation} & \textit{Meaning} & \textit{First appearance} \\ 
  \toprule
     $X, (X(t))_{t \geq 0}$ & solution of \eqref{WFSDE} & Introduction\\
     $\sigma, \theta, \nu_0, \nu_1$ & parameters of $X$ & Introduction \\
     $J, (J(t))_{t \geq 0}$ & cumulative effect of environment & Introduction\\
     $\bar{J}^T, (\bar{J}^T(\beta))_{\beta\in[0,T]}$ & time reversal of $J$ & Section \ref{s24} \\
     $\zeta, (t_i,p_i)_{i\in I}$ & collection of jumps of $J$ & Introduction \\
     $\mu$ & L\'evy measure of $J$ & Introduction\\
     $S(t)$ & $\sigma t+J(t)$ & Introduction \\
     $\omega, (\omega(t))_{t \geq 0}$ & fixed realization of $J$ & Section \ref{s21} \\
     $\Delta \omega(t)$ & jump of $\omega$ at time $t$ & Section \ref{s21} \\
     $\om^\delta/\om_\delta$ & $\om$ removed of small/large jumps & Eq. \eqref{defomdelta}\\
     $\textbf{0}$ & null environment & Section \ref{s21} \\
     $(X^\om(t))_{t \geq 0}$ & quenched version of $X$ & Section \ref{s23} \\
     $\sigma_N, \theta_N$ & parameters of Moran model & Section \ref{s21} \\
     $(Z_N^\om(t))_{t \geq 0}$, $(Z_N^J(t))_{t\geq 0}$ & $\sharp$ of fit individuals (Moran model) & Section \ref{s21}, Section \ref{s22} \\
     $\As_N$ & generator of $(Z_N^J(t))_{t\geq 0}$ & Section \ref{s22} \\
     $\As_N^0$ & $\As_N$ in the case $\mu=0$ & Section \ref{s22} \\
     $\mu_N(\om)$ & law of $(Z_N^\om(t))_{t\in[0,T]}$ & Section \ref{s31} \\
  $X_N, J_N, X_N^\om, \om_N$ & renormalizations of $Z_N^J, J, Z_N^\om, \om$ & Theorem \ref{thm2.2} \\
  $\As_N^*/\As$ & generator of $X_N/X$ & Section \ref{s32} \\
         \bottomrule
\end{tabularx}
\end{table}
  \begin{table}

\begin{tabularx}{\textwidth}{@{}XXX@{}}
\toprule
  \textit{Notation} & \textit{Meaning} & \textit{First appearance} \\ 
  \toprule
  $E_N$ & state space of $X_N$ & Eq. \eqref{defen}\\
  $(\Gs(\beta))_{\beta\geq 0}/(\Gs_T^\om(\beta))_{\beta\geq 0}$ & annealed/quenched ASG & Definition \ref{defannealdasg}\\
  $\sigma_{m,k}$ & rates of simultaneous branchings & Eq. \eqref{smk}\\
  $(\bar{\Gs}(\beta))_{\beta\geq 0}/(\bar{\Gs}_T^\om(\beta))_{\beta\geq 0}$ & annealed/quenched killed ASG & Section \ref{s25} \\
  $(R(\beta))_{\beta\geq 0}/(R_T^\om(\beta))_{ \beta \geq 0}$ & line-counting process of $\bar{\Gs}/\bar{\Gs}_T^\om$ & Section \ref{s25} \\
  $Q^\mu_\dagger, (q^\mu_\dagger(i,j))_{i,j\in\Nb_0^\dagger}$ & generator of $(R(\beta))_{\beta\geq 0}$ & Eq. \eqref{krates}\\
  $\pi_n$/$\Pi_n(\omega)$ & absorption probabilities of $R$/$R_T^\om$ & Eq. \eqref{defpin}\\
  $\eta_X/\Ls^\omega$ & limit distribution of $X(t)/X^\omega(0)$ & Theorem \ref{thm2.4}\\
  $X(\infty)$ & random variable with law $\eta_X$ & Theorem \ref{thm2.4}\\
  $J\otimes_\tau^{} \om$ & mixed environment & Section \ref{s25} \\
  $h_T(x)/h_T^\om(x)$ & ancestral type distribution at $T$ & Section \ref{s26} \\
  $h(x)/h^{\omega}(x)$ & ancestral type distribution & Theorem \ref{thm2.6}\\
  $(L(\beta))_{\beta\geq 0}/(L_T^\om(\beta))_{ \beta \geq 0}$ & pLD-ASG's line-counting process & Section \ref{s26} \\
  $Q^\mu, (q^\mu(i,j))_{i,j\in\Nb}$ & generator of $(L(\beta))_{\beta\geq 0}$ & Eq. \eqref{kratespldasg}\\
  $\eta_L/\mu^\om$ & limit law of $L(T)/L_T^\om(T-)$ & Theorem \ref{thm2.6}\\
  $L(\infty)$ & random variable with law $\eta_L$ & Theorem \ref{thm2.6}\\
  $a_n$ & $\mathbb{P}(L(\infty) > n)$, coefficients of $h(x)$ & Theorem \ref{thm2.6}\\
  $\gamma_{i,j}$ & coefficients in recursion \eqref{fr} & Theorem \ref{thm2.6}\\
  $\Db_{s,t}$, $\Db$ & spaces of c\`{a}dl\`{a}g functions & Section \ref{S2} \\
  $\Db_T^\star$/$\Db^\star$ & set of possible $(\omega(t))_{t \in [0,T]}$/$\omega$ & Section \ref{s21} \\
  $d_T^0$/$d_T^\star$ & metric on $\Db_{0,T}$/$\Db_T^\star$ & Appendix \ref{A1} \\
  $\Cs_T^\uparrow, BL(E)$ & functional sets & Appendix \ref{A1}, \ref{A2} \\ 
  $\lVert \cdot \rVert_T^0, \lVert \cdot \rVert_{T,\infty}, \lVert \cdot \rVert_{\textrm{BL}}$ & functional norms & Appendix \ref{A1}, \ref{A2} \\
  $\varrho_{E}(\cdot ,\cdot )$ & Bounded Lipschitz metric & Definition \ref{defblmenv}\\
\bottomrule
\end{tabularx}
\end{table}
\addtocontents{toc}{\protect\setcounter{tocdepth}{0}}
\vspace{.5cm}
\textbf{Acknowledgements}
We are grateful to E. Baake for many interesting discussions. We also thank Sebastian Hummel and two anonymous referees for their valuable suggestions to improve the manuscript. F. Cordero received financial support from Deutsche Forschungsgemeinschaft~(CRC 1283 ``Taming Uncertainty'', Project~C1). This paper is supported by NSFC grant No. 11688101. Gr\'egoire V\'echambre acknowledges the support from Deutsche Forschungsgemeinschaft~(CRC 1283 ``Taming Uncertainty'', Project~C1) for his visit to Bielefeld University in January 2018. Fernando Cordero acknowledges the support of NYU-ECNU Institute of Mathematical Sciences at NYU Shanghai for his visit to NYU Shanghai in July 2018. 

\addtocontents{toc}{\protect\setcounter{tocdepth}{2}}
\bibliographystyle{abbrvnat}
\bibliography{referenced}

\end{document}